\documentclass[12pt]{amsart}
\usepackage{amsmath}
\usepackage{amsthm}
\usepackage{amssymb}
\usepackage{amscd}
\usepackage{amsfonts}
\usepackage{amsbsy}
\usepackage{graphicx}
\usepackage{tikz}
\usepackage{color}
\usepackage{tikz-cd}

\newtheorem{theorem}{Theorem}[section]
\newtheorem{remark}[theorem]{Remark}
\newtheorem{proposition}[theorem]{Proposition}
\newtheorem{lemma}[theorem]{Lemma}
\newtheorem{corollary}[theorem]{Corollary}
\newtheorem{definition}[theorem]{Definition}

\theoremstyle{definition}
\newtheorem{example}{Example}

\def\ie{{\em i.e.,} }
\def\eg{{\em e.g.} }

\newfont\bbf{msbm10 at 12pt}
\def\eps{\varepsilon}
\def\phi{\varphi}
\def\R{{\mathbb R}}

\def\N{{\mathbb N}}
\def\Z{{\mathbb Z}}

\def\diam{\mbox{\rm diam}\,}
\def\lhe{\mbox{\rm lhe}}
\def\rhe{\mbox{\rm rhe}}
\def\Int{\mbox{\rm Int}\,}

\def\theta{\vartheta}

\def\ovl{\overleftarrow}
\def\ovr{\overrightarrow}

\def\Ls{\prec_L}
\def\Ll{\succ_L}
\def\eps{\varepsilon}

\def\black{\textcolor{black}}

\def\f{[} 
\def\ast{{^*\!\!A^*}}
\def\astr{{A^*}}
\def\astl{{^*\!\!A}}

\def\kone{{\kappa_1(q)}}
\def\ktwo{{\kappa_2(q)}}
\def\kthree{{\kappa_3(q)}}
\def\ki{{\kappa_i(q)}}
\def\BD{{\phi_{\mathcal{C}}}}
\def\Br{{\phi_{\mathcal{R}}}}
\begin{document}
	
	\title{Accessible points of planar embeddings of tent inverse limit spaces}
	
	\author{Ana Anu\v{s}i\'c, Jernej \v{C}in\v{c}}
	\address[A.\ Anu\v{s}i\'c]{Faculty of Electrical Engineering and Computing,
		University of Zagreb,
		Unska 3, 10000 Zagreb, Croatia}
	\email{ana.anusic@fer.hr}
	
	\address[J.\ \v{C}in\v{c}]{National Supercomputing Centre IT4Innovations, Division of the University of Ostrava, Institute for Research and Applications of Fuzzy Modeling, 30. dubna 22, 70103 Ostrava, Czech Republic} 
	\email{jernej.cinc@osu.cz}
	\thanks{AA was supported in part by Croatian Science Foundation under the project IP-2014-09-2285.
		J\v{C} was supported by the FWF stand-alone project P25975-N25 and by the Vienna Doctoral School of Mathematics. J\v{C} also acknowledges the support of University of Ostrava grant IRP201824 “Complex topological
		structures”.
		We gratefully acknowledge the support of the bilateral grant \emph{Strange Attractors and Inverse Limit Spaces},  \"Osterreichische
		Austauschdienst (OeAD) - Ministry of Science, Education and Sport of the Republic of Croatia (MZOS), project number HR 03/2014.}
	\date{\today}
	
	\subjclass[2010]{37B10, 37B45, 37E05, 54H20}
	\keywords{unimodal map, inverse limit space, planar embeddings}

	\begin{abstract}
		In this paper we study a class of embeddings of tent inverse limit spaces. We introduce techniques relying on the Milnor-Thurston kneading theory and use them to study sets of accessible points and prime ends of given embeddings. We completely characterize accessible points and prime ends of standard embeddings arising from the Barge-Martin construction of global attractors. In the other studied embeddings we find phenomena which do not occur in the standard embeddings. \black{Furthermore, for the class of studied non-standard embeddings we prove that shift homeomorphism can not be extended to a planar homeomorphism.} 
	\end{abstract}

	\maketitle
	\section{Introduction}
	The problem of classifying continua that can be embedded in the plane is of substantial interest in Continuum Theory, mainly because it is intrinsically related with the solution of the Fixed Point Property for planar non-separating continua. In the case when a continuum is chainable, \ie it admits an $\eps$-mapping on the interval $[0,1]$ for every $\eps>0$, it follows from an old result of Bing \cite{Bing} that the continuum can be embedded in the plane. Therefore, it is natural to ask how many possible non-equivalent embeddings \black{(see Definition~\ref{def:equivalent})} of a specific chainable continuum there exist and what these embeddings look like. The straightforward way to approach the description of embeddings is through their sets of accessible points or through their prime end structure.\\
	Inverse limit spaces on intervals are chainable. In \cite{ABCemb} Bruin and the authors showed that there exist uncountably many non-equivalent embeddings of tent map inverse limit spaces for all tent maps with slopes greater than $\sqrt{2}$, but they give no insight what the constructed embeddings look like. In this paper we study the class of planar embeddings from \cite{ABCemb} in detail, focusing primarily on accessible sets and the prime end structure in the finite critical orbit case. 
	
	The study of sets of accessible points of planar embeddings of Knaster continuum was given by Mayer in \cite{May}, and the characterization of possible sets of accessible points of embeddings of Knaster continua was given by D\c{e}bski \& Tymchatyn in \cite{DeTy}. The study of embeddings of unimodal inverse limit spaces appears in the literature in two forms; corresponding to attractors of orientation preserving (by Brucks \& Diamond in \cite{BrDi}) and orientation reversing (by Bruin in \cite{Br1}) planar homeomorphisms. We refer to those embeddings as \emph{standard embeddings}. Barge and Martin showed in \cite{BM} that every inverse limit space with a single interval bonding map can be realized as an attractor of an orientation preserving planar homeomorphism which acts on the attractor in the same way as the natural shift homeomorphism acts on the inverse limit. Using the construction from \cite{BM}, Boyland, de Carvalho and Hall recently gave in \cite{3G} the complete classification of the prime end structure and accessible sets of the Brucks-Diamond embedding of unimodal inverse limit spaces (satisfying certain regularity conditions valid for \eg tent map inverse limits). For \black{a class of} non-standard embeddings of tent inverse limit spaces constructed in \cite{ABCemb}, the natural shift homeomorphism cannot be extended to the plane as we show in Section~\ref{sec:ext}. \black{With that result we partially answer a question posed by Boyland, de Carvalho and Hall in \cite{3G} asking for which embeddings of tent inverse limit spaces the natural shift homeomorphism can be extended to a planar homeomorphism. Note that recently authors together with Bruin \cite{ABCplanar} constructed larger class of embeddings of tent inverse limit spaces, for which it is yet unknown whether they can be extended to a planar homeomorphism.}	
	\black{Because the embeddings from \cite{ABCemb} cannot be extended to a planar homomorphism,} we lack dynamical techniques as used in \cite{3G}. Therefore, for the construction and study of embeddings we chose a symbolic approach emerging from the Milnor-Thurston kneading theory in \cite{MiTh} which was already used in constructions of embeddings by Brucks \& Diamond \cite{BrDi}, Bruin \cite{Br1} and Bruin and the authors \cite{ABCemb}. It turns out that such a construction gives straightforward calculation techniques on the itineraries which we exploit throughout the paper, \black{even in the case when the dynamical techniques are present, \ie in the standard embeddings.}
	
	By $\N$ we denote the set of natural numbers and let $\N_0:=\{0\}\cup\N$.
	The \emph{Hilbert cube} is the space $\f 0, 1]^{-\N_0}$ equipped  with the product metric 
	$$
	d(x, y):=\sum_{i\leq 0}2^{i}|\pi_i(x)-\pi_i(y)|,
	$$
	where $\pi_i\colon\f 0, 1]^{-\N_0}\to\f 0, 1]$ denote the coordinate projections for $i\leq 0$.
	
	The \emph{tent map family} $T_s:[0,1] \rightarrow [0,1]$ is defined by $T_s(\black{t}):=\min\{ s\black{t},s(1-\black{t}) \}$ where $\black{t}\in [0,1]$ and $s\in (0,2]$. Let $c=\frac{1}{2}$ denote the {\em critical point} of the map $T_s$. In the rest of the paper we work with tent maps for slopes $s\in (\sqrt{2},2]$ and when there is no need to specify the slope we set for brevity $T:=T_s$. The \emph{inverse limit space with the bonding map $T$} is a subspace of the Hilbert cube defined by
	$$
	X:=\underleftarrow{\lim}(\f 0, 1], T)=\{x\in[0,1]^{-\N_0}: T(\pi_{i}(x))=\pi_{i+1}(x), \ i\leq 0 \}.
	$$
	The space $X$ is a {\em continuum}, \ie  compact and connected metric space. Define the \emph{shift homeomorphism} as 
	$\sigma\colon X\to X$, $\pi_i(\sigma(x)):=T(\pi_{i}(x))$ for every $i\leq 0$. 
	
	The space obtained by restricting the bonding map $T$ to its dynamical core is called the \emph{core} of $X$ and will be denoted by $X'$:
	$$X':=\underleftarrow{\lim}(\f T^2(c), T(c)], T|_{\f T^2(c), T(c)]}).$$
	A continuum is \emph{indecomposable} if it cannot be expressed as a union of two proper subcontinua.
	When $s\in(\sqrt{2}, 2]$, the core $X'$ is indecomposable  and by Bennett's theorem from \cite{Ben}, $X=X'\cup\mathcal{C}$, where $\mathcal{C}$ is a ray which contains the fixed point $(\ldots, 0, 0)$ and it compactifies on $X'$ (for details see \eg \cite{InMa}). The ray $\mathcal{C}$ shields off some points of the continuum $X$ and thus has an important effect on the set of accessible points in embeddings of $X$. However, the interesting phenomena regarding the structure of sets of accessible points occur in $X'$ and thus we will mostly ignore $\mathcal{C}$ in the remainder of the paper.  The structure of embedded $X$ (including $\mathcal{C}$) will be briefly discussed in Section~\ref{sec:intro}.

	The \emph{composant} $\mathcal{V}_x$ of a point $x\in K$ is the union of all proper subcontinua in $K$ that contain $x$.
	If a continuum is indecomposable it consists of uncountably many pairwise disjoint composants and every composant is dense in the continuum, see \cite{Na}. The arc-component $\mathcal{U}_x$ of a point $x\in K$ is the union of all arcs from $K$ that contain $x$.

	A point $a\in K \subset\R^2$ from a continuum $K$ is \emph{accessible (i.e., from the complement of $K$)} if there exists an arc $A\subset\R^2$ such that $A\cap K=\{a\}$.
	We say that an arc-component $\mathcal{U}_x$ is \emph{fully accessible} if every point from $\mathcal{U}_x$ is accessible. Mainly we will be interested in embeddings of inverse limits of indecomposable cores of tent maps with finite critical orbits. In these cases every arc-component of a point coincides with the composant of that point (see Proposition 3 from \cite{BB}). 
	
	We denote the class of embeddings of tent inverse limit spaces $X$ and their cores $X'$ constructed in \cite{ABCemb} by $\mathcal{E}$ and refer to them as $\mathcal{E}$-embeddings. In \cite{ABCemb}, every $\mathcal{E}$-embedding of $X$ is represented as a union of uncountably many horizontal segments (called basic arcs) which are aligned along vertically embedded Cantor set with prescribed identifications between some endpoints of basic arcs (see Section~\ref{sec:embed} of this paper and \cite{ABCemb} for details). An $\mathcal{E}$-embedding of $X$ is then uniquely determined by the left infinite itinerary $L=\ldots l_2l_1$, which is a symbolic description of the largest basic arcs among all basic arcs.
	
	\black{In Sections~\ref{sec:prelim},~\ref{sec:embed}, and \ref{sec:arc-comp} we recall basic notions from symbolic representation of unimodal inverse limits, recap the construction of embeddings of tent inverse limit spaces from \cite{ABCemb}, and give a symbolic characterization of arc-components in $X$, generalizing the result from the paper by Brucks \& Diamond \cite{BrDi}.} In Section~\ref{sec:accarcs}, we characterize the possible sets of accessible points in an arc-component of any indecomposable plane non-separating continuum $K$. In Section~\ref{sec:primeends} we briefly introduce Carath\'eodory's prime end theory and discuss the existence of fourth kind prime ends in special cases which occur for tent map inverse limits.  In Section~\ref{sec:intro}, we begin our study of embeddings $\mathcal{E}$. We introduce the notion of cylinders of basic arcs and techniques to explicitly calculate their extrema. We show that two $\mathcal{E}$-embeddings of the same space $X$ are equivalent when they are determined by eventually the same left infinite tail $L$. Given an $\mathcal{E}$-embedding of $X$, we prove that the arc-component of the top basic arc with symbolic description $L$ (throughout the paper this arc-component is denoted by $\mathcal{U}_L$) is fully accessible, if the top basic arc is not a spiral point (see Definition~\ref{def:spiral} and Figure~\ref{fig:spiral}). However, we also show that $\mathcal{U}_L$ is not necessarily the unique fully accessible arc-component. In the same section we briefly discuss $\mathcal{E}$-embeddings of decomposable continuum $X$ and characterize the set of accessible points up to two points on the corresponding circle of prime ends. From Section~\ref{sec:tops/bottoms} onwards we study $\mathcal{E}$-embeddings of indecomposable continuum $X'$. In Section~\ref{sec:tops/bottoms} we give sufficient conditions on itineraries of $L$ and kneading sequences $\nu$ associated with $X'$ so that the embeddings of $X'$ allow more than one fully accessible arc-component and give some interesting examples of such embeddings.
	
	We say that $x\in X$ is a \emph{folding point} if for every $\eps>0$ there exists a neighbourhood $U_{\eps}$ of $x$ which is not homeomorphic to the $C\times (0,1)$, where $C$ is the Cantor set.
	A point $x\in X$ is called an \emph{endpoint} if for every two subcontinua $X_1, X_2\subset X$ such that $x\in X_1\cap X_2$, either $X_1\subset X_2$ or $X_2\subset X_1$.
	Note that endpoints are also folding points.
	In Section~\ref{sec:FP} we characterize accessible folding points of $\mathcal{E}$-embeddings when the critical orbit of the tent map is finite. Surprisingly, no endpoints will be accessible in any $\mathcal{E}$-embedding of $X'$ with the exception of Brucks-Diamond embedding. Another surprising phenomenon is the occurrence of \emph{Type 3} folding points (see Definition~\ref{def:FPtypes} and Figure~\ref{fig:preper}) when the orbit of the third iterate of the critical point is periodic but the critical point itself is not periodic. Such a phenomenon does not occur in the standard embeddings of any tent map inverse limit space.
	
	In Section~\ref{sec:ext} we prove that for every embedding constructed in \cite{ABCemb} except for the \black{standard embeddings (ones} constructed by Brucks \& Diamond~\cite{BrDi} and Bruin~\cite{Br1} \black{respectively)}, the natural shift homeomorphism cannot be extended from the $\mathcal{E}$-embedding of $X'$ to the whole plane. Showing that, we \black{partially} answer the question posed by Boyland, de Carvalho and Hall in the paper \cite{3G}\black{, whether only for the standard embeddings of tent inverse limit spaces the shift homeomorphism can be extended to a planar homeomorphism}. In Section~\ref{sec:fullyacc} we study special examples of embeddings of $X'$. We explicitly show that every $X'$ can be embedded with at least two non-degenerate fully accessible arc-components. In a finite orbit case when we have exactly two fully accessible arc-component we show that there exists an embedding of $X'$ with exactly two simple dense canals. 
	
	We conclude the paper with the complete characterization of sets of accessible points (and thus also the prime end structure of the corresponding circle of prime ends) of the standard two embeddings: Bruin's embedding of $X'$ (Section~\ref{sec:Br}) and the Brucks-Diamond embedding of $X'$ (Section~\ref{sec:BD}) using symbolic dynamics.
	In Section~\ref{sec:Br} we show that for Bruin's embedding of $X'$ there is exactly one fully accessible non-degenerate arc-component and no other point from the embedding of $X'$ is accessible, if $X'$ is different from the Knaster continuum. We show that if $X'$ is not the Knaster continuum, then Bruin's embedding of $X'$ has exactly one simple dense canal.  \black{ If $X'$ is Knaster continuum then there is exactly one fully accessible non-degenerate arc-component and the endpoint of $\mathcal{C}$ is also accessible. Specially, there are no simple dense canals in this embedding of the Knaster continuum.}  
	In Section~\ref{sec:BD} we explicitly calculate the extrema of cylinders and neighbourhoods of folding points in the second standard embedding and  obtain equivalent results as obtained recently by Boyland, de Carvalho and Hall in \cite{3G}. Moreover, since the symbolic description makes it possible to distinguish endpoints within the set of folding points, our results extend the classification given in \cite{3G}.

	\section{Preliminaries on symbolic dynamics}\label{sec:prelim}
	In \cite{ABCemb}, uncountably many non-equivalent planar embeddings of indecomposable $X'$ were constructed with the use of symbolic dynamics by making any given $x\in X'$ accessible.
	We give a short overview of symbolic dynamics but we refer to \cite{ABCemb} and \cite{BrDi} for the more complete picture.
	
	The \emph{kneading sequence} of a map $T$ is a right-infinite sequence $\nu=c_1c_2\ldots\in\{0, 1\}^{\infty}$, where 
	$$c_i=\left\{
	\begin{array}{ll}
	0, &  T^i(c)\in[0, c],\\
	1, &  T^i(c)\in\f c, 1],
	\end{array}
	\right.
	$$
	for all $i\in\N$. If \black{$T^n(c)=c$} for some $n\in\N$, the critical point $c$ is periodic and the ambiguity in the definition of $\nu$ is resolved by defining $\nu$ to be the smaller of $(c_1\ldots c_{n-1}0)^{\infty}$ and $(c_1\ldots c_{n-1}1)^{\infty}$ in the \emph{parity-lexicographical ordering on $\{0, 1\}^{\infty}$} defined below.
	
	By $\#_1(a_1\ldots a_n)$ we denote the number of ones in a finite word $a_1\ldots a_n\in\{0, 1\}^n$; it can be either even or odd. 
	Choose $t=t_1 t_2\ldots \in\{0, 1\}^{\infty}$ and  $s=s_1 s_2\ldots \in\{0, 1\}^{\infty}$\black{, where we also permit $s$ and $t$ to be finite sequences of the same length. Specially, through the paper we will not compare finite words with infinite itineraries.} Set $0<1$.
	Take the smallest $k\in \N$ such that $s_k\neq t_k$\black{, if existent; otherwise  $s=t$.} Then the
	\emph{parity-lexicographical ordering} is defined as
	$$
	s\prec t \Leftrightarrow\left\{
	\begin{array}{ll}
	s_k<t_k \text{ and } \#_1(s_1\ldots s_{k-1}) \text{ is even, or }\\
	s_k>t_k \text{ and } \#_1(s_1\ldots s_{k-1}) \text{ is odd. }
	\end{array}
	\right.
	$$
	
	Fix the kneading sequence $\nu=c_1c_2\ldots$. The finite word $a_1\ldots a_n\in\{0, 1\}^n$ is called \emph{admissible} if \black{$c_2c_3\ldots c_{2+n-i}\preceq a_i\ldots a_n\preceq c_1c_2\ldots c_{1+n-i}$} for every $i\in\{1, \ldots, n\}$. Two-sided infinite sequence $\ldots s_{-2}s_{-1}.s_0s_1\ldots\in\{0, 1\}^{\Z}$ is called \emph{admissible} if every finite subword is admissible. Analogously we define an admissible left- or right-infinite sequence. Additionally, two-sided sequences $0^{\infty}s_ks_{k+1}\ldots$ will also be called admissible if $s_k=1$ and every finite subword of the right-infinite sequence $s_ks_{k+1}\ldots$ is admissible. Denote the set of all admissible two-sided infinite sequences by $\Sigma_{adm}$. 
	
	The set $\Sigma_{adm}\subset\Sigma=\{0, 1\}^{\Z}$ inherits the topology of $\Sigma$ given by the metric
	$$d((s_i)_{i\in\Z},  (t_i)_{i\in\Z}):=
	\sum_{i\in\Z}\frac{|s_i-t_i|}{2^{|i|}}.$$
	Define the \emph{shift homeomorphism on symbolic sequences} 
	$\sigma_{\Sigma}\colon\Sigma\to\Sigma$ as 
	$$
	\sigma_{\Sigma}(\ldots s_{-2}s_{-1}.s_0s_1\ldots):=\ldots s_{-1}s_0.s_1s_2\ldots
	$$
	
	The continuum $X$ is homeomorphic to the space $\Sigma_{adm}/\!\!\sim$ (see Proposition 2 in \cite{ABCemb}), where $\sim$ is the equivalence relation on $\Sigma_{adm}$ given by
	$$s\sim t \Leftrightarrow \left\{
	\begin{array}{ll}
	\textrm{ either } s_i=t_i \textrm{ for every } i\in \Z,\\
	\textrm{ or if there exists } k\in\Z \textrm{ such that } s_i=t_i \textrm{ for all } i\neq k \textrm{ but } s_k\neq t_k \\
	\textrm{ and } s_{k+1}s_{k+2}\ldots=t_{k+1}t_{k+2}\ldots=\nu.
	\end{array}
	\right.
	$$ 
	Sequences of the form $0^{\infty}s_ks_{k+1}\ldots$, treated differently in the definitions above, correspond to the points from \black{ray} $\mathcal{C}$. By removing these sequences from the definition of $\Sigma_{adm}$, we get a space homeomorphic to the core $X'$. Shifts $\sigma$ and $\sigma_{\Sigma}$ are conjugated (see Theorem 2.5 in \cite{BrDi}). Thus we will from here onwards abuse the notation and denote both $\sigma$ and $\sigma_{\Sigma}$ by $\sigma$. 
	
	\black{The homeomorphism between $X$ and $\Sigma_{adm}/\!\!\sim$ is given by the following. A point $x\in X$ is identified with the equivalence class  $\bar{x}=\ovl{x}.\ovr{x}=(x_i)_{i\in\Z}\in\Sigma_{adm}/\!\!\sim$}  according to the following rule:
	$$x_i=\left\{
	\begin{array}{ll}
	0, &  \pi_i(x)\in[0, c],\\
	1, &  \pi_i(x)\in\f c, 1],
	\end{array}
	\right.
	$$
	for $i\leq 0$ and
	$$x_i=\left\{
	\begin{array}{ll}
	0, &  T^i(\pi_0(x))\in[0, c],\\
	1, &  T^i(\pi_0(x))\in\f c, 1],
	\end{array}
	\right.
	$$
	for $i\in\N$. If the ambiguity in the definition of $x_i$ happens more than once, then $c$ is periodic and we study the itinerary of the modified kneading sequence instead. That way for every $x\in X$ there are at most two corresponding identified itineraries. \black{If $x\in X$ has two different backward itineraries $\ovl{x_1}$ and $\ovl{x_2}$, then we denote by $\ovl x:=\{\ovl{x_1},\ovl{x_2}\}$.}

	An \emph{arc} is a homeomorphic image of the interval $[0,1]\subset \R$.
	A key fact for constructing embeddings in \cite{ABCemb} is that $X\simeq\Sigma_{adm}/\!\!\sim$ can be represented as the union of 
	{\em basic arcs} defined below. \black{A basic arc will be an arc or a point in $X$ consisting of all points that have the same corresponding backward itinerary $\ovl s$, with the exception of two possible endpoints with possibly two backward itineraries, one of which is $\ovl s$.}
	
	From now on, when we speak about left infinite sequences we omit minuses in indices and write $\overleftarrow{s}=\ldots s_{2}s_{1}$ for the sake of brevity.
	
	\begin{definition}\label{def:basicarc}
		Let $\overleftarrow{s}=\ldots s_{2}s_{1}\in\{0, 1\}^{\infty}$
		be an admissible left-infinite sequence. The set 
		$$
		A(\overleftarrow{s}):=\{x\in X: \ovl{s}\black{\in}\ovl{x}\}\subset X.
		$$
		is called a {\em basic arc}. 
	\end{definition}
	\black{It is not difficult to show that $A(\ovl s)$ is indeed an arc in $X$ (can be degenerate). See \eg \cite[Lemma 1]{Br1}.}
	
	\begin{remark}
		\black{Later in the file,} when it will be clear from the context that we refer to the basic arc $A(\ovl{s})$ we will often abbreviate notation and write only $\ovl{s}$.
	\end{remark}
	
	To explain two quantities on basic arcs with which we will often work we need the following definition.
	\black{
		\begin{definition}\label{def:kappa}
			For $\nu=c_1c_2\ldots$ we define
			$$\kappa:=\min\{i-2: i\geq 3, c_i=1\}.$$
		\end{definition}
		\begin{remark}
			Definition~\ref{def:kappa} says that the beginning of the kneading sequence is  $\nu=10^{\kappa}1\ldots$. If $\kappa=1$, since we restrict to non-renormalizable tent maps, we can conclude even more, namely that $\nu=10(11)^n0\ldots$, for some $n\in\N$.
		\end{remark}
	}
	
	For every \black{admissible left-infinite sequence $\ovl s\neq \ovl 0$} we define two quantities as follows:
	\begin{eqnarray*}
		\tau_L(\overleftarrow{s})
		&:=&\sup\{n>1 : s_{n-1}\ldots s_{1}=c_1c_2\ldots c_{n-1}, \#_1(c_1\ldots c_{n-1})\textrm{ odd}\}, \\
		\tau_R(\overleftarrow{s})
		&:=&\sup\{n\geq 1 : s_{n-1}\ldots s_{1}=c_1c_2\ldots c_{n-1}, \#_1(c_1\ldots c_{n-1})\textrm{ even}\}.
	\end{eqnarray*}
	\black{In the definition of $\tau_R$ we allow $n=1$, so it follows immediately that supremum runs over a non-empty set. On the other hand, supremum is well defined  for $\tau_L$ as well. Namely, if $s_1=1$, then $s_1=c_1$ and $\#_1(c_1)$ is odd. In case $s_1=0$ we find the smallest $n>2$ so that $s_{n-1}=1$, which indeed exists since $\ovl s\neq \ovl 0$. If $n>\kappa+2$ the word $s_{n-2}\ldots s_1=0^{n-2}$ is not admissible. Thus $n\leq\kappa+2$ and $s_{n-1}\ldots s_1=c_1\ldots c_{n-1}=10\ldots 0$. }
	\black{Now we restate some lemmas from \cite{Br1} in our setting.}
	\begin{lemma}\label{lem:first}\cite[Lemma 2 \black{and Lemma 3}]{Br1}
		Let $\overleftarrow{s} \in \{ 0,1\}^{\infty}$ be an admissible left-infinite sequence
		such that \black{$\ovl s\neq \ovl 0$}. Then
		$$\max \pi_0(A(\ovl s))=\inf\{T^n(c): s_{n-1}\ldots s_1=c_1\ldots c_{n-1}, \#_1(c_1\ldots c_{n-1}) \black{\text{ even}\}},$$
		$$\min \pi_0(A(\ovl s))=\sup\{T^n(c): s_{n-1}\ldots s_1=c_1\ldots c_{n-1}, \#_1(c_1\ldots c_{n-1}) \black{\text{ odd}\}}.$$
		Specially, if $\tau_L(\overleftarrow{s}), \tau_R(\overleftarrow{s})<\infty$,
		then 
		$$
		\pi_0(A(\overleftarrow{s}))=[T^{\tau_L(\overleftarrow{s})}(c), T^{\tau_R(\overleftarrow{s})}(c)].
		$$
		If $\overleftarrow{t}\in\{0, 1\}^{\infty}$ is another
		admissible left-infinite sequence such that
		$s_i=t_i$ for all $i>0$ except 
		for $i=\tau_R(\overleftarrow{s})=\tau_R(\overleftarrow{t})$ (or $i=\tau_L(\overleftarrow{s})=\tau_L(\overleftarrow{t})$), 
		then $A(\overleftarrow{s})$ and $A(\overleftarrow{t})$ have a common endpoint projecting to the right (left) endpoint of $\pi_0(A(\ovl s))$.
	\end{lemma}
	\black{
		\begin{example}
			Assume $\nu=10010\ldots$. Then $\ovl s=\ldots 001001$ is an admissible left-infinite sequence. Note that $\tau_L(\ovl s)=2$ and $\tau_R(\ovl s)=5$. So $\pi_0(A(\ovl s))=[T^2(c), T^5(c)]$, see Figure~\ref{fig:basic}. 
		\end{example}
		\begin{figure}[!ht]
			\centering
			\begin{tikzpicture}[scale=3]
			%\draw (0,1)--(1,1);
			\draw (0,0.75)--(0.8,0.75);
			\draw[dashed, domain=90:270] plot ({0.125*cos(\x)}, {0.875+0.125*sin(\x)});
			\draw[dashed, domain=270:450] plot ({0.8+0.125*cos(\x)}, {0.625+0.125*sin(\x)});
			\node[circle,fill, inner sep=1] at (0,0.75){};
			\node[circle,fill, inner sep=1] at (0,1){};
			\node[circle,fill, inner sep=1] at (0.8,0.75){};
			\node[circle,fill, inner sep=1] at (0.8,0.5){};
			\node at (0,1.1) {\tiny $\ldots 001011.0010\ldots$};
			\node at (0,0.65) {\tiny $\ldots 001001.0010\ldots$};
			\node at (0.8,0.4) {\tiny $\ldots 011001.0\ldots$};
			\node at (0.8,0.85) {\tiny $\ldots 001001.0\ldots$};
			\end{tikzpicture}
			\black{\caption{An example of a basic arc. Here it consists of an open arc in $X$ consisting of all points with the same backward itinerary and two pairs of identified points. Identifications are represented by dashed semi-circles. Itineraries of joined points are sketched in the figure.}}
			\label{fig:basic}
		\end{figure}
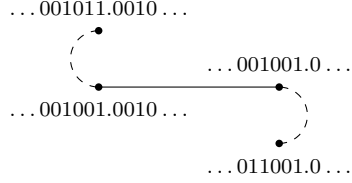
		\begin{remark}
			Note that the basic arc $\ovl s=\ovl 0$ needs to be taken special care of because $\tau_L(\ovl 0)$ is not well defined. However, it is not difficult to see that $\pi_0(A(\ovl 0))=[0,T(c)]$.
		\end{remark}}
	\black{We recall the symbolic characterization of endpoints in $X$. It states that endpoints of $X$ are endpoints of basic arcs which are contained in exactly one basic arc, see Figure~\ref{fig:endpts}. Note that the characterization distinguishes two cases, when $x\in X$ is such that $\pi_i(x)\neq c$ for every $i< 0$ and when there is $i<0$ such that $\pi_i(x)=c$. In the second case we use the fact that $x\in X$ is an endpoint of $X$ if and only if $\sigma^i(x)$ is an endpoint of $X$ for every $i\in \N$.}
	\black{
		\begin{proposition}\label{prop:endpt}\cite[Proposition 2]{Br1}
			Let $x\in X$ be such that $\pi_i(x)\neq c$ for all $i<0$. Then $x$ has a unique backward itinerary $\ovl x$. Point $x$ is an endpoint of $X$ if and only if $\tau_L(\ovl x)=\infty$ and $\pi_0(x)=\min\pi_0(A(\ovl x))$ (or $\tau_R(\ovl x)=\infty$ and $\pi_0(x)=\max\pi_0(A(\ovl x))$). If $x\in X$ is such that $\pi_i(x)=c$ for some $i<0$, there exists $j\in\N$ such that the backward itinerary of $\sigma^{j}(x)$ is unique and we apply the first claim to $\sigma^{j}(x)$ in this case.
		\end{proposition}
		\begin{figure}[!ht]
			\centering
			\begin{tikzpicture}[scale=3]
			\draw (0,1)--(1,1);
			\node[circle,fill, inner sep=1] at (0,1){};
			\node[circle,fill, inner sep=0] at (0,0.75){};
			\node at (-0.1,1) {\small $x$};
			\end{tikzpicture}
			\hspace{10pt}
			\begin{tikzpicture}[scale=3]
			\draw (0,1)--(1,1);
			\draw[dashed, domain=90:270] plot ({0.125*cos(\x)}, {0.875+0.125*sin(\x)});
			\node[circle,fill, inner sep=1] at (0,1){};
			\node[circle,fill, inner sep=1] at (0,0.75){};
			\node at (-0.2,0.85) {\small $x$};
			\end{tikzpicture}
			\hspace{10pt}
			\begin{tikzpicture}[scale=3]
			%\draw[color=white] (0,1)--(1,1);
			\node[circle,fill, inner sep=1] at (0,1){};
			\node[circle,fill, inner sep=0] at (0,0.75){};
			\node at (0.1,1) {\small $x$};
			\end{tikzpicture}
			\hspace{10pt}
			\begin{tikzpicture}[scale=3]
			%\draw (0,1)--(1,1);
			\draw[dashed, domain=90:270] plot ({0.125*cos(\x)}, {0.875+0.125*sin(\x)});
			\node[circle,fill, inner sep=1] at (0,1){};
			\node[circle,fill, inner sep=1] at (0,0.75){};
			\node at (-0.2,0.85) {\small $x$};
			\end{tikzpicture}
			\black{\caption{Endpoints of $X$, all possibilities. Dashed semi-circles connect points  identified under $\sim$.}}
			\label{fig:endpts}
		\end{figure}
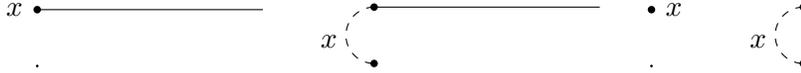
	}
	\black{
		\section{Construction of planar embeddings $\mathcal{E}$}\label{ABCembed}\label{sec:embed}
		In this section we recall the construction of planar embeddings of $X$ constructed in \cite{ABCemb}.}
	
	\black{
		Roughly speaking (we appoint the reader to Figure~\ref{fig:embed}), we first embed the Cantor set in the vertical segment and code points of the Cantor set as left-infinite sequences of 0s and 1s respecting some given order. For every admissible left-infinite sequence $\ovl s$ (corresponding to a point in an embedded Cantor set) we draw a horizontal arc whose projection on $x$-axis is exactly $\pi_0(A(\ovl s))$ and the projection on $y$-axis is exactly the point coded by $\ovl s$. Then we join certain pairs of identified endpoints of basic arcs with semi-circles. If the order on the Cantor set is properly defined, the semi-circles will not intersect, and the figure will be homeomorphic to $X$. The embeddings will be precisely defined in the rest of this section.}

	Fix an admissible left-infinite sequence $L=\ldots l_{2}l_{1}\in\{0, 1\}^{\infty}$. Basic arcs are embedded in the plane with respect to any chosen $L$ as a subset of $\f 0, 1]\times C$, where $C\subset \f 0, 1]$ denotes the Cantor set
	$$C:=[0,1]\setminus \bigcup^{\infty}_{m=1} \bigcup^{3^{m-1}-1}_{k=0}(\frac{3k+1}{3^{m}},\frac{3k+2}{3^m}).$$

	Recall from Lemma~\ref{lem:first} that we know how to compute $\pi_0(A(\ovl s))$ for every admissible $\ovl s$. Every basic arc $A(\overleftarrow{s})$ is embedded as the horizontal arc $\pi_0(A(\overleftarrow{s}))\times\{\psi_L(\overleftarrow{s})\}$, where 
	$$
	\psi_L(\overleftarrow{s}) := \sum_{i=1}^\infty (-1)^{\#_1(l_{i}\ldots l_{1})-\#_1(s_{i}\ldots s_{1})} 3^{-i} + \frac{1}{2} ,
	$$
	This implies a linear order on the basic arcs in which $A(L)$ is the largest. The precise definition is given by:
	
	\begin{definition} \black{Let $L=\ldots l_2l_1\in \{0,1\}^{\infty}$ be fixed.}
		Let $\overleftarrow{s}, \overleftarrow{t}\in\{0, 1\}^{\infty}$ and let $k\in \N$ be the smallest natural number such that $s_{k}\neq t_{k}$. Then 
		\begin{equation}\label{eq:L}
		\overleftarrow{s}\prec_L\overleftarrow{t} \Leftrightarrow 
		\begin{cases}
		t_{k}=l_{k} \text{ and } \#_1(s_{k-1}\ldots s_{1})-\#_1(l_{k-1}\ldots l_{1}) \text{ even, or }\\
		s_{k}=l_{k} \text{ and } \#_1(s_{k-1}\ldots s_{1})-\#_1(l_{k-1}\ldots l_{1}) \text{ odd.}
		\end{cases}
		\end{equation}
	\end{definition}
	
	\black{
		\begin{example}\label{ex:0}
			Assume that the ordering on left-infinite sequences $\{0,1\}^{\infty}$ (\ie on the Cantor set) is determined by $L=\ldots 010101$, see Figure~\ref{fig:coding}. Then $L$ is the largest sequence and $\ldots 010100$ is the smallest. However, if $X$ is a tent inverse limit with bonding map which is not $T_2$, then there will exists sequences which are not admissible. For example let \eg $\nu=(101)^{\infty}$. Then $\ldots 010100$ (which would be, if admissible, the itinerary of the smallest basic arc in the ordering along the vertically embedded Cantor set) is not admissible and finding the smallest admissible left-infinite sequence requires more work. In the example of $\nu=(101)^{\infty}$ the word is admissible if an only if it does not contain two consecutive zeros. The smallest admissible left-infinite sequence will be $S=\ldots 101010$, see Example~\ref{ex:1} for further explanation.
		\end{example}
	}
	\begin{figure}[!ht]
		\unitlength=1mm
		\centering
		\black{\begin{picture}(8,87)(35,-6)
			\put(20,0){\line(0,1){81}}
			\put(30,0){\line(0,1){27}}
			\put(30,54){\line(0,1){27}}
			\put(42,0){\line(0,1){9}}
			\put(42,18){\line(0,1){9}}
			\put(42,54){\line(0,1){9}}
			\put(42,72){\line(0,1){9}}
			\put(55,0){\line(0,1){3}}
			\put(55,6){\line(0,1){3}}
			\put(55,18){\line(0,1){3}}
			\put(55,24){\line(0,1){3}}
			\put(55,54){\line(0,1){3}}
			\put(55,60){\line(0,1){3}}
			\put(55,72){\line(0,1){3}}
			\put(55,78){\line(0,1){3}}	
			\put(31,13){\small $\ldots 0$}
			\put(31,66){\small $\ldots 1$}
			\put(43,3){\small $\ldots 00$}
			\put(43,21){\small $\ldots 10$}
			\put(43,57){\small $\ldots 11$}
			\put(43,75){\small $\ldots 01$}
			\put(56,0.5){\scriptsize $\ldots 100$}
			\put(56,6.5){\scriptsize $\ldots 000$}
			\put(56,18){\scriptsize $\ldots 010$}
			\put(56,24.5){\scriptsize $\ldots 110$}
			\put(56,54.5){\scriptsize $\ldots 111$}
			\put(56,60){\scriptsize $\ldots 011$}
			\put(56,72.5){\scriptsize $\ldots 001$}
			\put(56,78.5){\scriptsize $\ldots 101$}
			\end{picture}}
		\black{\caption{Coding the Cantor set with respect to $\protect L=\ldots 010101$}}
		\label{fig:coding}
	\end{figure}
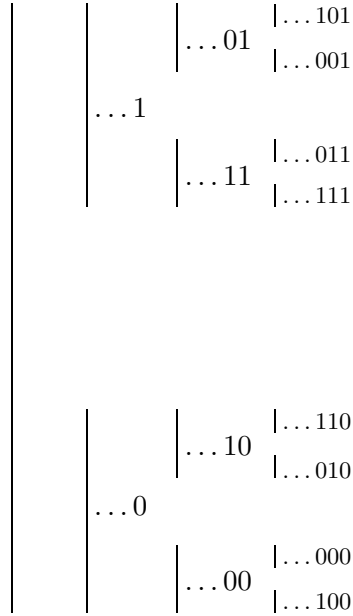
	
	\black{
		We have embedded basic arcs as horizontal segments ordered along the Cantor set. Now we show how to embed $X$ in the plane. We want to find points in embedded horizontal arcs which are identified under $\sim$, see Lemma~\ref{lem:first}.
		If $(s_i)_{i\in\Z}, (t_i)_{i\in\Z}\in\Sigma_{adm}$ are such that $s_i=t_i$ for all $i\in\Z$ except for exactly one $k<0$, and $s_{-k+1}s_{-k+2}\ldots=t_{-k+1}t_{-k+2}\ldots=\nu$, then $A(\ovl s)$ and $A(\ovl t)$ contain endpoints which are identified under $\sim$. If $\#_1(s_{-k+1}\ldots s_{-1})$ is odd (even), then $\tau_L(\ovl s)=\tau_L(\ovl t)$ ($\tau_R(\ovl s)=\tau_R(\ovl t)$) and thus $\pi_0(A(\ovl s))$ and $\pi_0(A(\ovl t))$ have a common endpoint on the left (right). We join embedded $A(\ovl s)$ and $A(\ovl t)$ by a semi-circle on the left (right). See Figure~\ref{fig:embed}.}
	\black{
		\begin{figure}[!ht]
			\begin{tikzpicture}[thick, scale=0.6]
			\draw (1,1)--(9,1);
			\draw (1,3)--(9,3);
			\draw[domain=270:450] plot ({9+cos(\x)},{2+sin(\x)});
			\draw (3,2)--(12,2);
			\end{tikzpicture}
			\hspace{0.5cm}
			\begin{tikzpicture}[thick, scale=0.6]
			\draw (1,1)--(9,1);
			\draw (1,3)--(9,3);
			\draw[domain=270:450] plot ({9+1*cos(\x)},{2+1*sin(\x)});
			\draw (1,2)--(9,2);
			\draw (1,4)--(9,4);
			\draw[domain=270:450] plot ({9+1*cos(\x)},{3+1*sin(\x)});
			\end{tikzpicture}
			\black{\caption{``Self-intersections'' in representations of $X$ that are excluded by Proposition~3 from \cite{ABCemb}.}}
			\label{case1}
		\end{figure}
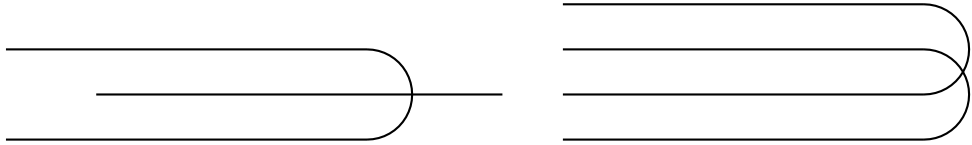
	}

	\black{
		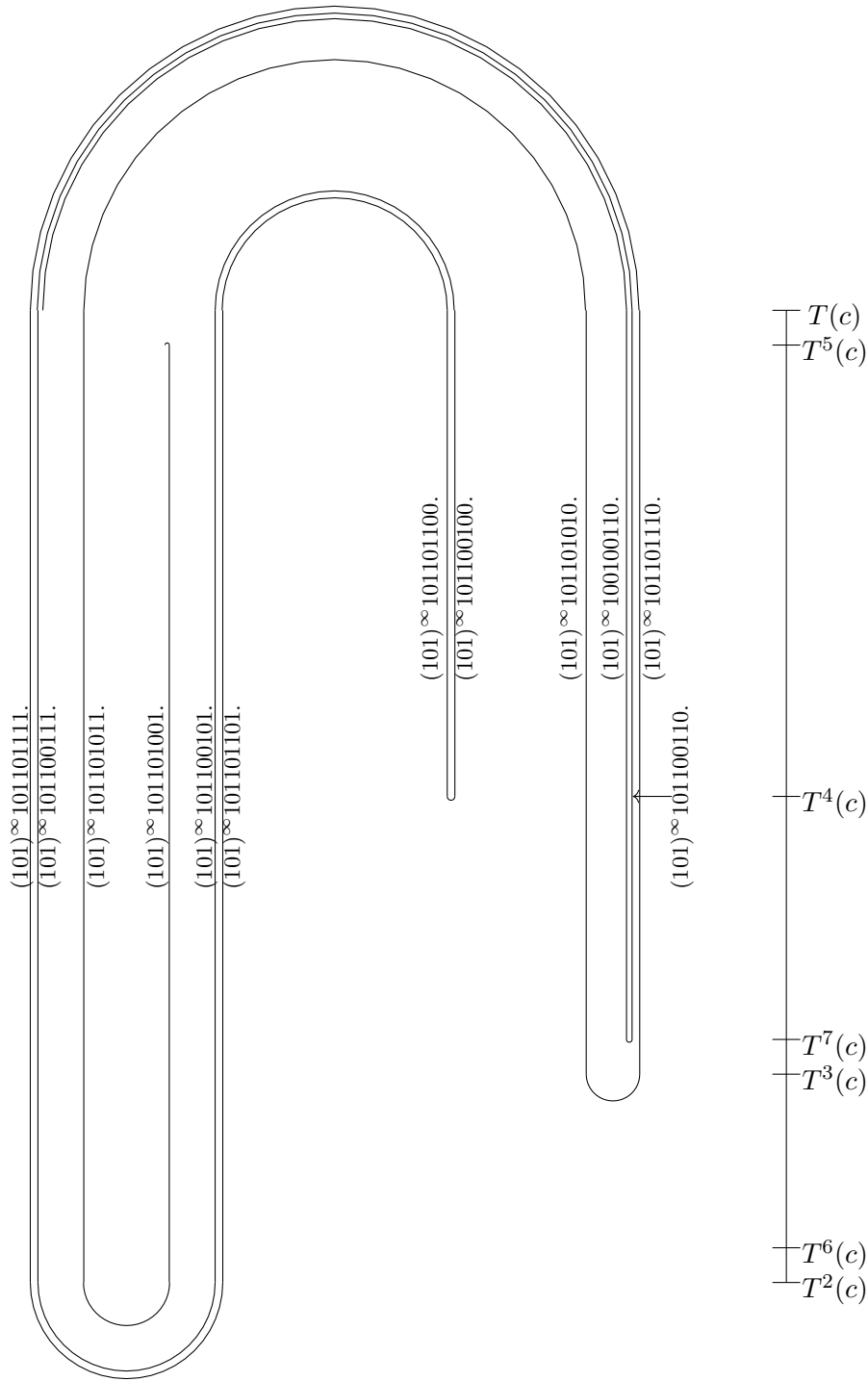
\begin{figure}[!ht]
			\centering
			\begin{tikzpicture}[scale=0.96,rotate=90]	
			\draw (0,-2)--(14,-2);
			\draw (0,-2.2)--(0,-1.8); \node at (-0.1,-2.7){\small $T^2(c)$};
			\draw (0.5,-2.2)--(0.5,-1.8); \node at (0.4,-2.7){\small $T^6(c)$};
			\draw (3,-2.2)--(3,-1.8); \node at (2.9,-2.7){\small $T^3(c)$};
			\draw (3.5,-2.2)--(3.5,-1.8); \node at (3.4,-2.7){\small $T^7(c)$};
			\draw (7,-2.2)--(7,-1.8); \node at (6.9,-2.7){\small $T^4(c)$};
			\draw (13.5,-2.2)--(13.5,-1.8); \node at (13.4,-2.7){\small $T^5(c)$};
			\draw (14,-2.2)--(14,-1.8); \node at (13.9,-2.7){\small $T(c)$};	
			\draw (3,0.11)--(14,0.11);
			\draw (3.5,0.22)--(14,0.22);
			\draw (3.5,0.3)--(14,0.3);
			\draw (3,0.88)--(14,0.88);
			\draw (7,2.77)--(14,2.77);
			\draw (7,2.88)--(14,2.88);
			\draw (0,6.11)--(14,6.11);
			\draw (0,6.22)--(14,6.22);
			\draw (0,6.88)--(13.5,6.88);
			\draw (0,8.11)--(14,8.11);
			\draw (0,8.77)--(14,8.77);
			\draw (0,8.88)--(14,8.88);	
			\draw[domain=270:450] plot ({14+4.38*cos(\x)}, {4.5+4.38*sin(\x)});
			\draw[domain=270:450] plot ({14+4.28*cos(\x)}, {4.5+4.28*sin(\x)});
			\draw[domain=270:450] plot ({14+4.2*cos(\x)}, {4.5+4.2*sin(\x)});
			\draw[domain=270:450] plot ({14+3.61*cos(\x)}, {4.5+3.61*sin(\x)});
			\draw[domain=270:450] plot ({14+1.72*cos(\x)}, {4.5+1.72*sin(\x)});
			\draw[domain=270:450] plot ({14+1.62*cos(\x)}, {4.5+1.62*sin(\x)});
			\draw[domain=270:450] plot ({13.5+0.03*cos(\x)}, {6.91+0.03*sin(\x)});	
			\draw[domain=90:270] plot ({0.62*cos(\x)}, {7.495+0.62*sin(\x)});
			\draw[domain=90:270] plot ({1.275*cos(\x)}, {7.495+1.275*sin(\x)});
			\draw[domain=90:270] plot ({1.385*cos(\x)}, {7.495+1.385*sin(\x)});
			\draw[domain=90:270] plot ({7+0.055*cos(\x)}, {2.825+0.055*sin(\x)});
			\draw[domain=90:270] plot ({3+0.385*cos(\x)}, {0.495+0.385*sin(\x)});
			\draw[domain=90:270] plot ({3.5+0.04*cos(\x)}, {0.26+0.04*sin(\x)});	
			\node[rotate=90] at (7,9){\scriptsize $(101)^{\infty}101101111.$};
			\node[rotate=90] at (7,8.6){\scriptsize $(101)^{\infty} 101100111.$};
			\node[rotate=90] at (7,7.9){\scriptsize $(101)^{\infty} 101101011.$};
			\node[rotate=90] at (7,7.05){\scriptsize $(101)^{\infty} 101101001.$};
			\node[rotate=90] at (7,6.35){\scriptsize $(101)^{\infty} 101100101.$};
			\node[rotate=90] at (7,5.95){\scriptsize $(101)^{\infty} 101101101.$};
			\node[rotate=90] at (10,3.1){\scriptsize $(101)^{\infty} 101101100.$};
			\node[rotate=90] at (10,2.6){\scriptsize $(101)^{\infty} 101100100.$};
			\node[rotate=90] at (10,1.1){\scriptsize $(101)^{\infty} 101101010.$};
			\node[rotate=90] at (10,-0.1){\scriptsize $(101)^{\infty} 101101110.$};
			\node[rotate=90] at (10,0.5){\scriptsize $(101)^{\infty} 100100110.$};
			\node[rotate=90] at (7,-0.5){\scriptsize $(101)^{\infty} 101100110.$};
			\draw[->] (7,-0.35)--(7,0.21);	
			\end{tikzpicture}
			\black{\caption{The planar representation of an arc in $\mathcal{U}_{(101)^{\infty}}\subset X$ with the corresponding kneading sequence $\protect\nu=100110010\ldots$. 
					The ordering on basic arcs is such that the basic arc coded by $\protect L=1^{\infty}.$ is the largest. Figure is taken from \cite{ABCemb}.}}
			\label{fig:embed}
		\end{figure}
	}

	\black{
		Proposition~3 from \cite{ABCemb} shows that representation of $X$ with basic arcs connected by semi-circles drawn as above will not intersect other horizontal arcs or semi-circles, see Figure~\ref{case1}. Moreover,  Lemma~5 and Lemma~6 in \cite{ABCemb} show that the space consisting of horizontal arcs and semi-circles is homeomorphic to $\Sigma_{adm}/\!\!\sim$ and thus gives an embedding $\phi_L$ of $X$ in the plane. 
}	

	Throughout the paper, $L$ will denote the left-infinite sequence of the largest basic arc which determines the planar embedding $\phi_L$ of $X$ by the rules in the equation (\ref{eq:L}).
	Let us fix the inverse limit space $X$.
	Denote by $\mathcal{E}$ the family of all embeddings of $X$ constructed in \cite{ABCemb}, \ie with respect to all admissible tails $L$ and refer to them as $\mathcal{E}$-embeddings. From now onwards we think of $X$ as a planar continuum obtained by an $\mathcal{E}$-embedding of $\Sigma_{adm}/\!\!\sim$ described above. 
	\black{\section{Arc-components}\label{sec:arc-comp}}
	
	We want to describe the sets of accessible points of embedded $X$, focusing primarily on the (fully) accessible arc-components. Since the approach in this study is mostly symbolic, we need to obtain a symbolic description of an arc-component in $X$. Recall that $\mathcal{U}_x$ denotes the arc-component of $x\in X$. 
	
	\begin{definition}\label{def:spiral}
		We say that a point $x\in X$ is a \emph{spiral point} if there exists a ray $R\subset X$ such that $x$ is an endpoint of $R$ and \black{an arc} $\f x, y]\subset R$ contains infinitely many basic arcs for every $x\neq y\in R$.
	\end{definition}
	
	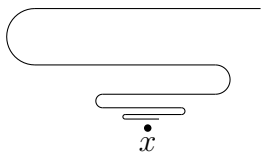
\begin{figure}[!ht]
		\centering
		\begin{tikzpicture}[scale=3]
		\draw (0,1)--(1,1);
		\draw (0,0.75)--(0.8,0.75);
		\draw (0.3,0.62)--(0.8,0.62);
		\draw (0.3,0.56)--(0.65,0.56);
		\draw (0.4,0.53)--(0.65,0.53);
		\draw (0.4,0.51)--(0.55,0.51);
		\draw[domain=90:270] plot ({0.125*cos(\x)}, {0.875+0.125*sin(\x)});
		\draw[domain=270:450] plot ({0.8+0.065*cos(\x)}, {0.685+0.065*sin(\x)});
		\draw[domain=90:270] plot ({0.3+0.03*cos(\x)}, {0.59+0.03*sin(\x)});
		\draw[domain=270:450] plot ({0.65+0.015*cos(\x)}, {0.545+0.015*sin(\x)});
		\draw[domain=90:270] plot ({0.4+0.01*cos(\x)}, {0.52+0.01*sin(\x)});
		\node[circle,fill, inner sep=1] at (0.5,0.47){};
		\node at (0.5,0.4) {$x$};
		\end{tikzpicture}
		\caption{Point $x\in X$ is a spiral point.}
		\label{fig:spiral}
	\end{figure}

	\begin{proposition}\label{prop:spiral}
		If $x\in X$ is a spiral point, then $A(\ovl{x})$ is degenerate and $x$ is an endpoint of $X$.
	\end{proposition}
	\begin{proof}
		\black{Assume that $x$ has a unique left-infinite itinerary $\ovl x$, \ie $\pi_i(x)\neq c$ for every $i<0$.}
		Assume that $A(\ovl{x})$ is not degenerate. Note that $x$ is not in the interior of $A(\ovl{x})$ since then $R\cup A(\ovl{x})$ is a triod, \black{a contradiction with $X$ being chainable}. Without loss of generality assume that $x$ is the right endpoint of $A(\ovl{x})$. \black{Then $\tau_R(A(\ovl{x}))=\infty$, otherwise $x$ has two backward itineraries. Therefore, by Proposition~\ref{prop:endpt}, $x$ is an endpoint of $X$. Since $x\in A(\ovl x)\cup R$, we conclude that $A(\ovl x)$ is degenerate.}
		
		\black{Assume that $x$ does not have a unique left-infinite itinerary, \ie $x$ corresponds to two endpoints of basic arcs glued together under $\sim$. Denote two left-infinite itineraries of $x$ by $\ovl{x_1}$ and $\ovl{x_2}$. Then there is $k\in\N$ such that $\sigma^{-k}(A(\ovl{x_1})\cup A(\ovl{x_2}))$ is contained in a single basic arc $A$ and $\sigma^{-k}(x)\in A$ has a unique left-infinite tail. If $A$ is not degenerate, we get a contradiction as in the previous paragraph. Thus $A(\ovl{x_1})=A(\ovl{x_2})=A(\ovl x)$ is degenerate. Again Proposition~\ref{prop:endpt} implies that point $x$ is an endpoint of $X$.}
	\end{proof}
	\black{
		\begin{remark}
			Note that a planar representation of a degenerate basic arc can be represented either as a point or two points joined by a semi-circle, see the last two pictures in Figure~\ref{fig:endpts}. 
	\end{remark}}

	The following corollary follows directly from Proposition~\ref{prop:spiral} since a spiral point cannot be contained in the interior of an arc.
	
	\begin{corollary}
		Non-degenerate arc-components in $X$ are:
		\begin{itemize}
			\item lines (\ie continuous images of $\R$) with no spiral points,
			\item rays (continuous images of $\R^+$), where only the endpoint can be a spiral point,
			\item arcs, where only endpoints can be spiral points.
		\end{itemize}
	\end{corollary}
	
	\begin{remark}\label{rem:symbolicAC}
		Let \black{$\ovl y\neq \ovl w$ be admissible left-infinite sequences.} By Lemma~\ref{lem:first}, $A(\ovl{y})$ and $A(\ovl{w})$ are connected by finitely many basic arcs if and only if there exists $k\in \N$ such that $\ldots y_{k+1}y_{k}= \ldots w_{k+1}w_{k}$. We say that $\ovl y$ and $\ovl w$ have the same \emph{tail}. Thus every arc-component \black{in $X$} is determined by its tail with the exception of (one or two) spiral points with different tails. This generalizes the symbolic representation of arc-components for finite critical orbit $c$ given in \cite{BrDi} on arbitrary tent inverse limit space $X$.
	\end{remark}
	
	\section{General results about accessibility}\label{sec:accarcs}

	\begin{definition}
		We say that a continuum $K\subset\R^2$ \emph{does not separate the plane} if $\R^2\setminus K$ is connected.
	\end{definition}

	For $K\subset \R^2$ we denote by $\mathrm{Cl}(K)$ the closure of $K$ in $\R^2$.
	The following proposition is a special case of Theorem 3.1. in \cite{Bre}.
	
	\begin{proposition}\label{prop:access}
		Let $K\subset\R^2$ be a non-degenerate indecomposable continuum that does not separate the plane and let $Q=\f x, y]\subset K$ be an arc. If $x$ and $y$ are accessible, then $Q$ is fully accessible.
	\end{proposition}
	\begin{proof}
		Assume by contradiction that arc $Q$ is not fully accessible. Because $x,y\in K$ are both accessible there exists a point $w\in \R^{2}\setminus K$ and arcs $Q_x:=[x,w], Q_y:=[y,w]\subset \R^{2}$ such that $(x,w],(y,w]\subset \R^{2}\setminus K$ \black{and $Q\cup Q_x\cup Q_y=:S$ is a simple closed curve} in $\R^2$, see Figure~\ref{fig:Jordan}. Thus $\R^2\setminus S=S_1\cup S_2$ where $S_1$ and $S_2$ are open sets in $\R^2$ such that $\partial S_1=\partial S_2=S$. Specifically $S_1$ contains no accumulation points of $S_2$ and vice versa. 
		Denote by
		$K_1:=K\cap\mathrm{Cl}(S_1),$ $K_2:=K\cap\mathrm{Cl}(S_2)$. Note that $K_1,K_2$ are subcontinua of $K$ and $K_1,K_2\neq \emptyset$. Because $Q$ is not fully accessible it follows that $K_1, K_2\neq K$. Furthermore $K_1\cup K_2=K$, which is a contradiction with $K$ being indecomposable.	
	\end{proof}

	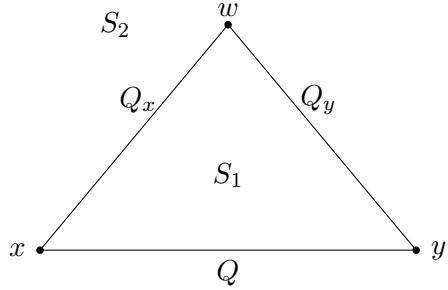
\begin{figure}[!ht]
		\centering
		\begin{tikzpicture}
		\node at (-0.3,1) {\small $x$};
		\node[circle,fill, inner sep=1] at (0,1){};
		\draw (0,1)--(5,1);
		\node at (5.3,1) {\small $y$};
		\node[circle,fill, inner sep=1] at (5,1){};
		\node at (2.5,0.7) {\small $Q$};
		\node at (2.5,4.2) {\small $w$};
		\node[circle,fill, inner sep=1] at (2.5,4){};
		\draw (0,1)--(2.5,4);
		\node at (1.3,3) {\small $Q_x$};
		\draw (5,1)--(2.5,4);
		\node at (3.7,3) {\small $Q_y$};
		\node at (2.5,2){\small $S_1$};
		\node at (1, 4){\small $S_2$};
		\end{tikzpicture}
		\caption{Simple closed curve from the proof of Theorem~\ref{prop:access}.}
		\label{fig:Jordan}
	\end{figure}

	\begin{corollary}\label{cor:class}
		Let $K$ be an indecomposable planar continuum that does not separate the plane and let $\mathcal{U}$ be an arc-component of $K$. There are four possibilities regarding the accessibility of $\mathcal{U}$:
		\begin{itemize}
			\item $\mathcal{U}$ is fully accessible.
			\item There exists an accessible point $u\in \mathcal{U}$ such that one component of $\mathcal{U}\setminus\{u\}$ is not accessible, and the other one is fully accessible.
			\item There exist two (not necessarily different) accessible points $u, v\in \mathcal{U}$ such that $\mathcal{U}\setminus[u, v]$ is not accessible and $[u, v]\subset \mathcal{U}$ is fully accessible. 
			\item $\mathcal{U}$ is not accessible.
		\end{itemize}
	\end{corollary}
	\begin{proof}
		By Proposition~\ref{prop:access}, the set of accessible points in $\mathcal{U}$ is connected. To see it is closed, take a \black{monotone sequence $(x_i)_{i\in\N}$ (note that we can parametrize every arc-component $\mathcal{U}$ and thus it makes sense to speak about monotone sequences)} of accessible points in $\mathcal{U}$ such that  $\lim_{i\to\infty}x_i=:x\in \mathcal{U}$. Let $w\in\R^2\setminus K$ and let $Q_i\subset \R^2$ be \black{disjoint} arcs with endpoints $x_i$ and $w$ and such that $Q_i\cap K=x_i$ for every $i\in\N$. \black{
			Using Remark (i) after the proof of \cite[Theorem 6, \S 61, IV]{Kuratowski} we can find a planar homeomorphism so that $[x_1,x]$ is a straight arc and then pick a point $w$ and arcs $Q_i$ so that they are straight arcs, and thus they limit on a straight arc as well.}
		Denote by $S_i$ the bounded open set in $\R^2$ with boundary $Q_1\cup Q_i\cup [x_1, x_i]$, where $[x_1, x_i]\subset \mathcal{U}$.  Note that $K\cap S_i=\emptyset$ for every $i\in\N$, since otherwise $K$ is decomposable by analogous arguments as in the proof of Proposition~\ref{prop:access}. Then also $K\cap (\cup_{i\in\N}S_i)=\emptyset$. Since $x$ is contained in the boundary of $\cup_{i\in\N}S_i$, which is arc-connected, we conclude that $x$ can be accessed with a ray from $\cup_{i\in\N}S_i\subset\R^2\setminus K$.
	\end{proof}
	
	\begin{remark}
		Note that it follows from the third item of Corollary~\ref{cor:class} that there can exist an endpoint $u=v\in \mathcal{U}$ which is accessible and every $x\in \mathcal{U}\setminus \{u\}$ is not accessible. For instance such embeddings for Knaster continuum are described in \cite{Sch} and the endpoint is the only accessible point in the arc-component $\mathcal{C}$. In the course of this paper we show that all cases from Corollary~\ref{cor:class} indeed occur in some embeddings of tent inverse limit spaces.
	\end{remark}
	
	\section{Basic notions from the prime end theory}\label{sec:primeends}
	
	In this section we briefly recall Carath\'eodory's prime end theory. Although the focus of this paper is not on the characterization of prime ends of studied embeddings of continua, we will include the study of prime ends of some interesting  examples throughout the paper. A general study of prime ends of standard planar embeddings appears at the end of the paper. 
	
	\begin{definition}
		Let $K\subset\R^2$ be a plane non-separating continuum. A \emph{crosscut} of $\R^2\setminus K$ is an arc $Q\subset\R^2$ which intersects $K$ only in its endpoints. Note that $K\cup Q$ separates the plane into two components, one bounded and the other unbounded. Denote the bounded component by $B_Q$. A sequence $\{Q_i\}_{i\in\N}$ of crosscuts is called a \emph{chain}, if the crosscuts are pairwise disjoint, $\diam Q_i\to 0$ as $i\to \infty$ and $B_{Q_{i+1}}\subset B_{Q_i}$ for every $i\in\N$. We say that two chains $\{Q_i\}_{i\in\N}$ and $\{R_i\}_{i\in\N}$ are \emph{equivalent} if for every $i\in\N$ there exists $j\in\N$ such that $B_{R_j}\subset B_{Q_i}$ and for every $j\in\N$ there exists $i'\in\N$ such that $B_{Q_{i'}}\subset B_{R_j}$. An equivalence class $[\{Q_i\}_{i\in\N}]$ is called a \emph{prime end}. A basis for the natural topology on the set of all prime ends consists of sets $\{[\{R_i\}_{i\in\N}]: B_{R_i}\subset B_Q \text { for all } i\}$ for all crosscuts $Q$. The set of prime ends equipped with the natural topology is a topological circle, called the \emph{circle of prime ends}, see \eg Section 2 in \cite{Bre}.
	\end{definition}
	
	\begin{definition}
		Let $P=[\{R_i\}_{i\in\N}]$ be a prime end. The \emph{principal set of $P$} is 
		$\Pi(P)=\{\lim Q_i: \{Q_i\}_{i\in\N}\in P\ \textrm{is convergent}\}$ and the \emph{impression of $P$} is $I(P)=\cap_i \mathrm{Cl}(B_{R_i})$. Note that both $\Pi(P)$ and $I(P)$ are subcontinua in $X'$ and $\Pi(P)\subseteq I(P)$. We say that $P$ is of the
		\begin{enumerate}
			\item \emph{first kind} if $\Pi(P)=I(P)$ is a point.
			\item \emph{second kind} if $\Pi(P)$ is a point and $I(P)$ is non-degenerate.
			\item \emph{third kind} if $\Pi(P)=I(P)$ is non-degenerate.
			\item \emph{fourth kind} if $\Pi(P)\subsetneq I(P)$ are non-degenerate.
		\end{enumerate}
	\end{definition}
	
	\begin{theorem}[Iliadis \cite{Ili}] Let $K$ be a plane non-separating indecomposable continuum. The circle of prime ends corresponding to $K$ can be decomposed into open intervals and their boundary points such that every open interval $J$ uniquely corresponds to a composant of $K$ which is accessible in more than one point and $I(e)\subsetneq K$ for every $e\in J$. For the boundary points $e$ it holds that $I(e)=K$.
	\end{theorem}
	
	\begin{proposition}\label{prop:fourth}
		Let $K$ be a plane non-separating \black{indecomposable} continuum such that every proper subcontinuum of $K$ is an arc and such that every composant contains at most one folding point. Then $\Pi(P)$ is degenerate or equal to $K$ for every prime end $P$. Specially, there exist no prime ends of the fourth kind.
	\end{proposition}
	\begin{proof} Assume there exists a prime end $P$ such that $\Pi(P)$ is non-degenerate and not equal to $K$. Then $\Pi(P)=[a,b]$ is an arc in $K$. We claim that both $a$ and $b$ are folding points. Assume that there exists $\eps>0$ such that $B(a, \eps)\cap K=C\times (0,1)$, where $C$ is the Cantor set and $B(a,\eps)$ denotes the open planar ball of radius $\eps$ around the point $a$. Since $a\in\Pi(P)$, there exist a chain of crosscuts $\{Q_i\}_{i\in\N}\in P$ such that $Q_i\to a$ as $i\to\infty$. Note that $Q_i\in B(a, \eps)$ for large enough $i$, so the endpoints of $Q_i$ are contained in $C\times (0,1)$ and the interior of $Q_i$ does not intersect $K$. \black{If there exists $N\in \N$ such that the arc $Q_N$ has both endpoints in the same component of $C\times (0,1)$, then $\Pi(P)$ is degenerate, a contradiction. Thus the endpoints of $Q_i$ do not lie in the same component of $C\times (0,1)$, and since $\diam Q_i\to 0$, we can find a subsequence $\{Q_{i_k}\}$ such that all endpoints of $Q_{i_k}$ are contained in different component of $C\times (0,1)$, see Figure~\ref{fig:fourth}}. Therefore, it is possible to translate every \black{$Q_{i_k}$} along $(0,1)$ and find a point $z\not\in[a, b]$ in the arc-component of $[a,b]$ for which there exists a chain of crosscuts $\{R_k\}_{k\in\N}$ equivalent to \black{$\{Q_{i_k}\}_{k\in\N}$} such that $R_k\to z$ as $k\to\infty$, see Figure~\ref{fig:fourth}. This contradicts the assumption, \ie point $a$ is a folding point. The proof for the point $b$ is analogous. We conclude that there exists a composant with at least two folding points, which is a contradiction.
	\end{proof}
	
	\begin{figure}[!ht]
		\centering
		\begin{tikzpicture}[scale=1.5]
		\draw (5,1)--(9,1);
		\draw[thick] (6,1)--(8,1);
		\node[circle,fill, inner sep=1] at (5.3,1){};
		\node[circle,fill, inner sep=1] at (6,1){};
		\node[circle,fill, inner sep=1] at (8,1){};
		\draw (5,0.90)--(9,0.90);
		\draw (5,0.65)--(9,0.65);
		\draw (5,0.45)--(9,0.45);
		\draw (5,0)--(9,0);
		\draw (6,0.90)--(6,0.65);
		\draw (6,0.45)--(6,0);
		%\draw (8,0.90)--(8,0.70);
		%\draw (8,0.45)--(8,0.10);
		\draw[dashed] (5.3,0.90)--(5.3,0.65);
		\draw[dashed] (5.3,0.45)--(5.3,0);
		\node at (6,1.2) {\small $a$};
		\node at (8,1.2) {\small $b$};
		\node at (5.3,1.2) {\small $z$};
		\node at (6.3,0.77) {\small $Q_{i_k}$};
		\node at (6.4,0.18) {\small $Q_{i_{k+1}}$};
		\node at (5.6,0.77) {\small $R_k$};
		\node at (5.7,0.18) {\small $R_{k+1}$};
		\end{tikzpicture}
		\caption{Translating the chain of crosscuts along $(0,1)$ in Proposition~\ref{prop:fourth}.}
		\label{fig:fourth}
	\end{figure}
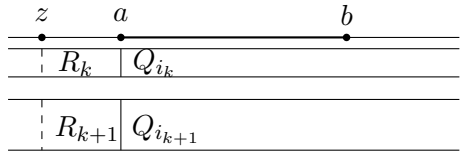
	
	\begin{definition}
		Let $K$ be a plane non-separating continuum. A prime end $P$ such that $\Pi(P)$ is non-degenerate but different than $K$ is called an \emph{infinite canal}. A third kind prime end $P$ such that $\Pi(P)=I(P)=K$ is called a \emph{simple dense canal}. 
	\end{definition}
	
	We obtain the following corollary, which we use later in the paper for discussing the prime end structure of  $\mathcal{E}$-embeddings of $X'$ when the critical orbit is finite.
	
	\begin{corollary}\label{cor:Iliadis}
		Let $K$ be an indecomposable plane non-separating continuum such that its every subcontinuum is an arc and every composant contains at most one folding point. Then the circle of prime ends corresponding to $K$ can be partitioned into open intervals and their endpoints. Open intervals correspond to accessible open \black{lines} in $K$. The endpoints of open intervals are the second or the third kind prime ends for which the impression is $K$. The second kind prime end corresponds to an accessible folding point in $K$ and the third kind prime end corresponds to a simple dense canal in $K$.
	\end{corollary}
	
	{\bf Question.} If $X'$ is the core of a tent map inverse limit, is there a planar embedding $\phi\colon X'\to\R^2$ such that $\phi(X')$ has fourth kind prime end?

\section{An Introduction to the study of accessible points of $\mathcal{E}$-embeddings}\label{sec:intro}

\black{In this section we reduce generality and focus again on our original study of inverse limit spaces of tent maps.}

By Corollary~\ref{cor:class}, if $x\in \mathcal{U}_x\subset X$ is accessible it does not a priori follow that every point from $\mathcal{U}_x$ is accessible, see \eg  Figure~\ref{fig:access2}. Recall that $X=\mathcal{C}\cup X'$. In this paper we study the sets of accessible points of embeddings of either $X$ or $X'$ and the two cases substantially differ as we shall see in this section. In the rest of the paper we are concerned only with embeddings of the cores $X'$.

	\begin{figure}[!ht]
		\centering
		\begin{tikzpicture}[scale=2.5]
		\draw (0,0.5)--(2,0.5);
		\draw (1,1)--(2,1);
		\draw (1,0.75)--(2,0.75);
		\draw (1,0.62)--(2,0.62);
		\draw (1,0.56)--(2,0.56);
		\draw (1,0.53)--(2,0.53);
		\draw (1,0.5)--(2,0.5);
		\draw (1,0.47)--(2,0.47);
		\draw (1,0.44)--(2,0.44);
		\draw (1,0.38)--(2,0.38);
		\draw (1,0.25)--(2,0.25);
		\draw (1,0)--(2,0);
		\draw (0.2,0.48)--(0.2,0.52);
		\node at (0.2,0.43) {\small $x$};
		\draw (1.7, 0.48)--(1.7,0.52);
		\draw[->] (2.3, 0.8)--(1.7,0.5);
		\node at (2.35,0.8){\small $y$};
		\end{tikzpicture}
		\caption{Point $x$ is accessible from the complement while point $y$ which has neighbourhood of Cantor set of arcs is not.}
		\label{fig:access2}
	\end{figure}
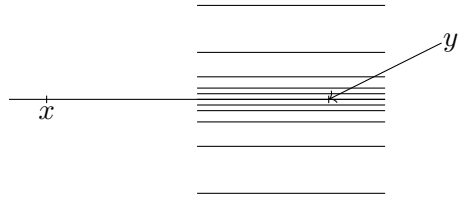

We will denote the smallest admissible left-infinite tail in $X'$ with respect to $\prec_L$ by $S$. The arc-component of points from $L$ ($S$) will be denoted from now onwards by $\mathcal{U}_L$ ($\mathcal{U}_S$). The following examples show that $\mathcal{U}_L$ and $\mathcal{U}_S$ do not necessarily coincide. Later in this section we will especially be concerned with the accessibility of $\mathcal{U}_L$ and $\mathcal{U}_S$.

\black{Recall that we consider $X$ as a continuum $\Sigma_{adm}/\!\!\sim$ embedded in the plane with respect to some admissible $L$. Recall that a basic arc consists of all points with the same backward itinerary and that each basic arc is embedded as a horizontal arc in the plane. We abuse the notation and identify basic arcs with their left-infinite sequences $\ovl s$.}

\begin{example}\label{ex:1} 
	Assume that the kneading sequence is given by $\nu=(101)^{\infty}$. Embed $X'$ in the plane according to the ordering in which $L=(01)^{\infty}$ is the largest. \black{As we commented in Example~\ref{ex:0}, the smallest sequence is then $S=(10)^{\infty}$. Note that the backward itinerary of $L$ and $S$ differ on infinitely many places. So, the results of Section~\ref{sec:arc-comp} imply that $S\not\subset\mathcal{U}_L$.}
\end{example}

\begin{example}\label{ex:2}
	 Take the kneading sequence $\nu=1001(101)^{\infty}$. Embed $X'$ in the plane according to the ordering in which $L=((001)(001101))^{\infty}$ is the largest. The smallest is then $S=((100)(101100))^{\infty}\not\subset\mathcal{U}_L$. Note that in comparison with the previous example this time $S\neq \sigma^k(L)$ for every $k\in\N$.
\end{example}
\begin{definition}
	Let $\nu$ be a kneading sequence. For any admissible finite word $a_n\ldots a_1\in\{0,1\}^n$ define the \emph{cylinder} $\f a_n\ldots a_1]$ as
	$$[a_n\ldots a_1]:=\{\overleftarrow{s}=\ldots s_{n+2}s_{n+1}a_n\ldots a_1: \overleftarrow{s} \text{ is an admissible left infinite sequence}\}.$$
\end{definition}

\begin{lemma}\label{lem:cylindersnonempty}
	If $a_n\ldots a_1$ is admissible, then $[a_n\ldots a_1]$ is not an empty set.
\end{lemma}
\begin{proof}
	 Say that $1a_n\ldots a_1$ is not admissible. In that case $1a_n\ldots a_1\succ c_1\ldots c_{n+1}$, so $a_n\ldots a_1\prec c_2\ldots c_{n+1}$, which is a contradiction with $a_n\ldots a_1$ being admissible. \black{Analogously}, the left infinite tail $1^{\infty}a_n\ldots a_1$ is admissible, which concludes this proof.
\end{proof}

\begin{definition}
	Assume $X$ is embedded in the plane with respect to $L=\ldots l_{2}l_{1}$
	and take an admissible finite word $a_n\ldots a_1$. The \emph{top} of the cylinder $[a_n\ldots a_1]$ is the left infinite sequence denoted by $L_{a_n\ldots a_1}\in[a_n\ldots a_1]$ such that $L_{a_n\ldots a_1}\succeq_L\overleftarrow{s}$, for all $\overleftarrow{s}\in[a_n\ldots a_1]$. Analogously we define the \emph{bottom} of the cylinder $[a_n\ldots a_1]$, denoted by $S_{a_n\ldots a_1}$, as the smallest left infinite sequence in $[a_n\ldots a_1]$ with respect to the order $\preceq_L$.
\end{definition}

\begin{remark}
	Note that each cylinder is a compact set (as a subset of the plane). Thus for admissible finite words $a_n\ldots a_1$ there always exist $L_{a_n\ldots a_1}$ and $S_{a_n\ldots a_1}$ (they can be equal).
\end{remark}

\begin{lemma}\label{lem:cyl}
	Assume $X$ is embedded in the plane with respect to $L$. For every admissible finite word $a_n\ldots a_1$ the arcs $A(L_{a_n\ldots a_1})$ and $A(S_{a_n\ldots a_1})$ are fully accessible. 
\end{lemma}
\begin{proof}
		Take a point $x\in A(L_{a_n\ldots a_1})$ and denote by $p_x=\psi\black{_L}(L_{a_n\ldots a_1})$ \black{(for the definition of $\psi_L$ see the beginning of Section~\ref{ABCembed})} the point in the Cantor set $C$ corresponding to the $y$-coordinate of $x$.
		Then the arc 
		$$
		Q=\left\{\left (\pi_0(x), p_x+\frac{\black{t}}{2\cdot 3^{n+1}}\right ), \black{t}\in[0,1]\right\}
		$$
		has the property that $Q\cap X=\{x\}$, see Figure~\ref{fig:neigh}. When $x\in A(S_{a_n\ldots a_1})$, we can similarly construct the arc 
		$Q'$, \black{accessing $x$ from below}, such that $Q'\cap X=\{x\}$ and conclude that $x$ is accessible.
\end{proof}

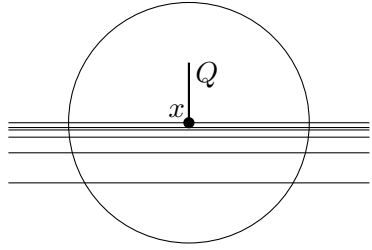
\begin{figure}[!ht]
	\centering
	\begin{tikzpicture}[scale=1.6]
	\draw[solid] (1,1) circle (1);
	\draw (-0.5,1)--(2.5,1);
	\draw (-0.5,0.5)--(2.5,0.5);
	\draw (-0.5,0.75)--(2.5,0.75);
	\draw (-0.5,0.88)--(2.5,0.88);
	\draw (-0.5,0.94)--(2.5,0.94);
	\draw (-0.5,0.96)--(2.5,0.96);
	\draw[thick] (1, 1)--(1,1.5);
	\node at (1.15,1.4){\small $Q$};
	\node[circle,fill, inner sep=1.5] at (1,1){};
	\node at (0.9,1.1){\small $x$};
	
	\end{tikzpicture}
	\caption{Point at the top of the cylinder $\f a_n\ldots a_1]$ is accessible by an arc $Q$.}
	\label{fig:neigh}
\end{figure}

From Lemma~\ref{lem:cyl} it follows specially that $A(L)$ and $A(S)$ in Example~\ref{ex:1} and Example~\ref{ex:2} are fully accessible as they are the largest and the smallest arcs respectively among all the arcs in embedding of $X'$ determined by $L$.\\
The following proposition is the first step in determining the set of accessible points of $\mathcal{E}$-embeddings.

\begin{proposition}\label{prop:fully}
	Take $L=\ldots l_{2} l_{1}$ and construct the embedding of $X$ with respect to $L$. Then every point in $X$ with the same symbolic tail as $L$ is accessible. If $A(L)$ is not a spiral point, then $\mathcal{U}_L$ is fully accessible.
\end{proposition}

\begin{proof}	
	Take a point $x\in X$, where $\ovl{x}=\ldots x_2x_1$ and there exists $n>0$ 
	such that $\ldots x_{n+2}x_{n+1}= \ldots l_{n+2}l_{n+1}$.
	If $\#_1(x_{n}\ldots x_{1})$ and $\#_1(l_{n}\ldots l_{1})$ have the same parity, then $\ldots l_{n+2}l_{n+1}x_{n}\ldots x_{1}=L_{x_{n}\ldots x_{1}}$ and it is equal to the $S_{x_{n}\ldots x_{1}}$ otherwise.
	Lemma~\ref{lem:cyl}, Corollary~\ref{cor:class} and Remark~\ref{rem:symbolicAC} conclude the proof.
\end{proof}

\begin{definition}\label{def:equivalent}
	Let $\phi, \psi\colon K\to \R^2$ be two embeddings of a continuum $K$ in the plane. We say that the embeddings are \emph{equivalent} if the homeomorphism $\psi\circ\phi^{-1}\colon\phi(K)\to\psi(K)$ can be extended to a homeomorphism of the plane.
\end{definition}

By $\phi_L$ we denote the $\mathcal{E}$-embedding of $X$ so that the arc $A(L)$ is the largest among all basic arcs.
In the following proposition we observe that given two left-infinite sequences $L^1, L^2$ with eventually the same tail, we get equivalent embeddings. 
\begin{proposition}
	Let $L^1=\ldots l^1_{2}l^1_{1}$ and $L^2=\ldots l^2_{2}l^2_{1}$ be such that there exists $n\in\N$ so that for every $k> n$ it holds that $l^1_{k}=l^2_{k}$. Then the embeddings $\phi_{L^1}$ and $\phi_{L^2}$ of $X$ are equivalent.
\end{proposition}
\begin{proof}
		If $\#_1(l^1_{n}\ldots l^1_{1})$ and $\#_1(l^2_{n}\ldots l^2_{1})$ are of the same (different) parity, then for every admissible $\ovl x=\ldots x_{2}x_{1}$ and $\ovl y=\ldots y_{2}y_{1}$ such that $x_{n}\ldots x_{1}=y_{n}\ldots y_{1}$ it follows that $\ovl x\prec_{L^1} \ovl y$ if and only if $\ovl x\prec_{L^2} \ovl y$ ($\ovl x\succ_{L^2} \ovl y$).\\
		We conclude that $\phi_{L^2}\circ \phi_{L^1}^{-1}\colon \phi_{L^1}(X)\to \phi_{L^2}(X)$ preserves (reverses) the order in every $n$-cylinder $\f a_n\ldots a_1]$. There exists a planar homeomorphism $h$ so that  $h|_{\phi_{L^1}(X)}=\phi_{L^2}(X)$ and $h$ permutes $n$-cylinders from the order determined by $L^1$ to the order determined by $L^2$, which concludes the proof.
\end{proof}

Now we briefly comment on $\mathcal{E}$-embeddings of $X$ (including the ray $\mathcal{C}$).
For the rest of the section assume that $X$ is not the Knaster continuum (since then $X=X'$, \ie $\mathcal{C}$ is contained in the core $X'$). Let $X$ be embedded in the plane with respect to $L=\ldots l_2l_1\neq 0^{\infty}l_{n}\ldots l_{1}$ for every $n\in\N$. The case when $\mathcal{E}$-embedding is equivalent to $L=0^{\infty}1$ (the Brucks-Diamond embedding from \cite{BrDi}) will be studied in Section~\ref{sec:BD}.

\begin{remark}
	When we study $X$ (\ie including the \black{ray} $\mathcal{C}$), there exist cylinders $[a_n\ldots a_1]$ where \black{$a_n\ldots a_1\prec_L c_2\ldots c_{1+n}$}, but there is $k\in\{1, \ldots, n-1\}$ such that $a_k\ldots a_1$ is admissible, $a_k=1$ and $a_n\ldots a_{k+1}=0^{n-k}$. In that case, $[a_n\ldots a_1]$ contains only one basic arc, that is $[a_n\ldots a_1]=\{ 0^{\infty}a_n\ldots a_1\}$ and $L_{a_n\ldots a_1}=S_{a_n\ldots a_1}=0^{\infty}a_n\ldots a_1.$ 
\end{remark}

\begin{remark}\label{rem:Cisolated}
	The \black{ray} $\mathcal{C}$ is isolated (when $X$ is not the Knaster continuum), and thus it is fully accessible in any $\mathcal{E}$-embedding of $X$. \black{In the circle of prime ends, $\mathcal{C}$ corresponds to an open interval with $\overline 0$ in the center.}
\end{remark}

\begin{proposition}\label{prop:Cacc}
	Take an admissible left-infinite sequence $\ovl{a}=\ldots a_2a_1$ such that $A(\ovl{a})\not\subset \mathcal{C}$ and $a_n\neq l_n$ for infinitely many $n\in\N$. Then there exist sequences $(\ovl{s_i})_{i\in\N}$ and $(\ovl{t_i})_{i\in\N}$ such that $A(\ovl{s_i}), A(\ovl{t_i})\subset\mathcal{C}$, $\ovl{s_i}, \ovl{t_i}\to\ovl a$ as $i\to\infty$ and $\ovl{s_i}\prec_L\ovl{a}\prec_L\ovl{t_i}$.
\end{proposition}
\begin{proof}
	First note that the assumption $A(\ovl{a})\not\subset \mathcal{C}$ is indeed needed since by Remark~\ref{rem:Cisolated}, $\mathcal{C}$ is isolated and thus the statement of the proposition does not hold for basic arcs from $\mathcal{C}$; thus assume $A(\ovl{a})\not\subset \mathcal{C}$.\\
	 Let $(N_i)_{i\in\N}$ be the sequence of natural numbers such that $a_n\neq l_n$ for $n\in\{N_i: i\in\N\}$.
	 Denote by
	$$\hspace{0.6cm}\ovl{t_i}:=0^{\infty}a^*_{N_{2i-1}}a_{N_{2i-1}-1}\ldots a_1$$
	$$\ovl{s_i}:=0^{\infty}a^*_{N_{2i}}a_{N_{2i}-1}\ldots a_1$$
	for every $i\in\N$. By contradiction, if a sequence $\ovl{t_i}$ is not admissible it holds that $\ldots 1a_{N_{2i-1}-1}\ldots a_1\succ_L \nu$. Thus, $a_{N_{2i-1}-1}\ldots a_1\ldots \prec \ovr{c_2}$ which is a contradiction with $a_{N_{2i-1}-1}\ldots a_1$ being an admissible word. Thus $\ovl{t_i}$ is admissible sequence and proof goes analogously for $\ovl{s_i}$. Note that $A(\ovl{t_i}), A(\ovl{s_i})\subset\mathcal{C}$ for every $i\in\N$.

	 Since $\#_1(a_{N_{2i-1}-1}\ldots a_1)$  and $\#_1(l_{N_{2i-1}-1}\ldots l_1)$ are of the same parity (the sequences differ on even number of entries) and
	 $\#_1(a_{N_{2i}-1}\ldots a_1)$ and $\#_1(l_{N_{2i}-1}\ldots l_1)$ are of different parity (the sequences differ on odd number of entries), it holds that $\ovl{s_i}\prec_L\ovl{a}\prec_L\ovl{t_i}$  for every $i\in\N$. 
\end{proof}

Combining Proposition~\ref{prop:fully} with Proposition~\ref{prop:Cacc} we obtain that only basic arcs from $\mathcal{U}_L$ or $\mathcal{C}$ can be tops or bottoms of cylinders of $\mathcal{E}$-embeddings of $X$.  \black{Using Proposition~\ref{prop:fourth}}  we obtain the following corollary.

\begin{corollary}\label{cor:withC} 
	If $A(L)$ is not a spiral point \black{and $X$ is not Knaster, \ie $X\neq X'$,} \black{then the embedding} $\phi_L(X)$ has exactly two non-degenerate fully accessible arc-components, namely $\mathcal{U}_L$ and $\mathcal{C}$ (however in  the embedding by Brucks-Diamond it holds that $\mathcal{C}=\mathcal{U}_L$). If $A(L)$ is non-degenerate, there are two remaining points on the circle of prime ends and they correspond either to an infinite canal in $X$ or to an \black{accessible} folding point. If $A(L)$ is degenerate then there are no infinite canals in $X$.
\end{corollary}
\black{
\begin{remark}\label{rem:new}
		It is easy to check that for Knaster continuum $X=X'$, only basic arcs from $\mathcal{U}_L$ can be extrema od cylinders in $\mathcal{E}$-embedding $\phi_L$, \ie $\mathcal{C}$ is not accessible, except for possibly its endpoint $\overline 0=(\ldots,0,0,0)$. Actually, in $\mathcal{E}$-embeddings the endpoint $\overline 0$ will always be accessible, see Remark~\ref{rem:endptKnaster}. So, there is a point in the circle of prime ends corresponding to an accessible $\overline 0$ and an interval corresponding to a fully-accessible line $\mathcal{U}_L$. Specially, there are no simple dense canals.
\end{remark}
We return to the embeddings of $X'$ and until the end of this section give grounds for further study.}
The following statements are going to be used often throughout the paper to determine \black{when} an arc-component is fully accessible.

\begin{definition}\label{def:neigh}
	Let $\ovl s=\ldots s_2s_1$ be an admissible left-infinite sequence. If $\tau_R(\ovl{s})<\infty$, the tail $\ovl {r(s)}=\ldots s_{\tau_R(\ovl s)+1}s^*_{\tau_R(\ovl s)}s_{\tau_R(\ovl s)-1}\ldots s_1$ is called \emph{the right neighbour of $\ovl s$} and if $\tau_L(\ovl{s})<\infty$, the tail $\ovl {l(s)}=\ldots s_{\tau_L(\ovl s)+1}s^*_{\tau_L(\ovl s)}s_{\tau_L(\ovl s)-1}\ldots s_1$ is called \emph{the left neighbour of $\ovl s$}. 
\end{definition}

\begin{proposition}\label{prop:wiggles}
	Embed $X'$ in the plane with respect to $L$. Assume $\ovl s$ is at the bottom (top) of some cylinder. If $\ovl{r(s)}$ is not at the top (bottom) of any cylinder, then $A(\ovl{r(s)})$ contains an accessible folding point, see Figure~\ref{fig:foldincyl1}. Analogous statement holds for $\ovl{l(s)}$.  
\end{proposition}
\begin{proof}
	If $\ovl{r(s)}$ is not the top (bottom) of any cylinder, then there exist left-infinite admissible sequences $\ovl{x_i}\succ_L\ovl{r(s)}$ ($\ovl{x_i}\prec_L\ovl{r(s)}$) such that $\ovl{x_i}\to\ovl{r(s)}$ as $i\to\infty$. If $\tau_R(\ovl{x_i})=\infty$ for infinitely many $i\in\N$, we have found a folding point in $A(\ovl{r(s)})$. So assume without loss of generality that $\tau_R(\ovl{x_i})<\infty$ for all $i\in\N$. If $\ovl{s}\succ_L \ovl{r(x_i)}$ ($\ovl{s}\prec_L \ovl{r(x_i)}$) for infinitely many $i\in\N$ we get a contradiction with $\ovl{s}$ being the top (bottom) of some cylinder. But then $\ovl{r(x_i)}\prec_L\ovl{r(s)}$ ($\ovl{r(x_i)}\succ_L\ovl{r(s)}$) for all but finitely many $i\in\N$ which gives a folding point in $A(\ovl{r(s)})$ again.
\end{proof}

\begin{figure}[!ht]
	\centering
	\begin{tikzpicture}[scale=1.5]
	
	\draw (0,1.6)--(2,1.6);
	\draw (0,1.5)--(2,1.5);
	\draw (0,1.38)--(2,1.38);
	\draw (0,1.3)--(2,1.3);
	\draw (0,1.16)--(2,1.16);
	\draw (0,1.1)--(2,1.1);
	\draw[domain=270:450] plot ({2+0.05*cos(\x)}, {1.55+0.05*sin(\x)});
	\draw[domain=270:450] plot ({2+0.04*cos(\x)}, {1.34+0.04*sin(\x)});
	\draw[domain=270:450] plot ({2+0.03*cos(\x)}, {1.13+0.03*sin(\x)});
	
	\draw[thick] (0, 1)--(3,1);
	\draw[thick] (0, 0)--(3,0);
	\draw[domain=270:450] plot ({3+0.5*cos(\x)}, {0.5+0.5*sin(\x)});
	\node[circle,fill, inner sep=1] at (2,1){};
	\node at (2.15,1.1){\small $p$};
	\node at (2.6,1.2){\small $\overleftarrow{r(s)}$};
	\node at (1.5,-0.2){\small $\overleftarrow{s}$};
	
	\draw[thin] (0,0.09)--(3,0.09);
	\draw[thin] (0,0.05)--(3,0.05);
	\draw[thin] (0,0.03)--(3,0.03);
	\draw[thin] (0,0.98)--(3,0.97);
	\draw[thin] (0,0.96)--(3,0.95);
	\draw[thin] (0,0.92)--(3,0.91);
	\draw[domain=270:450] plot ({3+0.47*cos(\x)}, {0.5+0.47*sin(\x)});
	\draw[domain=270:450] plot ({3+0.45*cos(\x)}, {0.5+0.45*sin(\x)});
	\draw[domain=270:450] plot ({3+0.41*cos(\x)}, {0.5+0.41*sin(\x)});
	\end{tikzpicture}
	\caption{Setup of Proposition~\ref{prop:wiggles}, where $p$ is a folding point.}
	\label{fig:foldincyl1}
\end{figure}
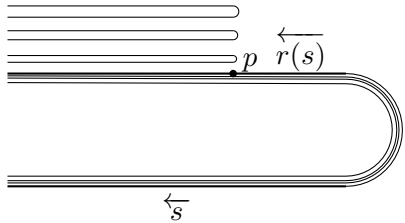

The following corollary follows directly from Proposition~\ref{prop:wiggles}.

\begin{corollary}\label{cor:fully}
	Let $\mathcal{U}\subset X'$ be an arc-component which contains no folding points and let $X'$ be $\mathcal{E}$-embedded. If there exists a basic arc from $\mathcal{U}$ that is fully accessible, then $\mathcal{U}$ is fully accessible. 
\end{corollary}

\begin{remark}
	When we embed only the core $X'$, there can exist accessible points in $X'\setminus\mathcal{U}_L$, see \eg Example~\ref{ex:1} and Example~\ref{ex:2}. In these two examples $\mathcal{U}_S\neq \mathcal{U}_L$ and points from $A(S)$ are accessible. In some cases $\mathcal{U}_S$ is fully accessible (see Lemma~\ref{lem:horror} from Section~\ref{sec:fullyacc}), but that is not always the case. In Section~\ref{subsec:accFP} we explicitly construct examples in which the arc-component $\mathcal{U}_S$ is only partially accessible.
\end{remark}

From Lemma~\ref{lem:cyl} it follows that the points at the top or bottom of cylinders are accessible. If a point which is not at the top or bottom of any cylinder has a neighbourhood homeomorphic to the Cantor set of arcs, we can conclude that is not accessible. However, the accessibility of folding points needs to be studied separately, since it is not straightforward to determine if they are accessible or not in a given embedding, see for example Figure~\ref{fig:folding}. Thus we need to do a detailed study on conditions for a folding point to be accessible.
For instance, in special embeddings of the Knaster continuum in \cite{Sch} the endpoint is always accessible. 

\begin{remark}\label{rem:fp}
	  When the orbit of $c$ is finite, with (pre)period $n\in\N$, there exist exactly $n$ folding points (see \cite{Br2}). They are contained in different arc-components which are permuted by the shift homeomorphism. If the orbit of $c$ is periodic, the folding points are endpoints (see \cite{BaMa}). \black{If $c$ is strictly preperiodic, the folding points are not endpoints.}
\end{remark}

\begin{figure}[!ht]
	\centering
	\begin{tikzpicture}[scale=1.2]
	\draw[dashed] (2,1) circle (1);
	\draw (0,1.6)--(2,1.6);
	\draw (0,1.5)--(2,1.5);
	\draw (0,1.38)--(2,1.38);
	\draw (0,1.3)--(2,1.3);
	\draw (0,1.16)--(2,1.16);
	\draw (0,1.1)--(2,1.1);
	\draw (2,0.4)--(4,0.4);
	\draw (2,0.5)--(4,0.5);
	\draw (2,0.62)--(4,0.62);
	\draw (2,0.7)--(4,0.7);
	\draw (2,0.84)--(4,0.84);
	\draw (2,0.9)--(4,0.9);
	\draw[domain=270:450] plot ({2+0.05*cos(\x)}, {1.55+0.05*sin(\x)});
	\draw[domain=270:450] plot ({2+0.04*cos(\x)}, {1.34+0.04*sin(\x)});
	\draw[domain=270:450] plot ({2+0.03*cos(\x)}, {1.13+0.03*sin(\x)});
	\draw[domain=90:270] plot ({2+0.05*cos(\x)}, {0.45+0.05*sin(\x)});
	\draw[domain=90:270] plot ({2+0.04*cos(\x)}, {0.66+0.04*sin(\x)});
	\draw[domain=90:270] plot ({2+0.03*cos(\x)}, {0.87+0.03*sin(\x)});
	\draw[thick] (0, 1)--(4,1);
	\draw (0, 1.07)--(4,1.07);
	\draw (0, 1.24)--(4,1.24);
	\draw (0, 1.44)--(4,1.44);
	\node[circle,fill, inner sep=1] at (2,1){};
	\node at (2,-0.3) {($a$)};
	
	\draw[dashed] (7,1) circle (1);
	\draw (5,1.6)--(7,1.6);
	\draw (5,1.5)--(7,1.5);
	\draw (5,1.38)--(7,1.38);
	\draw (5,1.3)--(7,1.3);
	\draw (5,1.16)--(7,1.16);
	\draw (5,1.1)--(7,1.1);
	\draw (7,0.4)--(9,0.4);
	\draw (7,0.5)--(9,0.5);
	\draw (7,0.62)--(9,0.62);
	\draw (7,0.7)--(9,0.7);
	\draw (7,0.84)--(9,0.84);
	\draw (7,0.9)--(9,0.9);
	\draw[domain=270:450] plot ({7+0.05*cos(\x)}, {1.55+0.05*sin(\x)});
	\draw[domain=270:450] plot ({7+0.04*cos(\x)}, {1.34+0.04*sin(\x)});
	\draw[domain=270:450] plot ({7+0.03*cos(\x)}, {1.13+0.03*sin(\x)});
	\draw[domain=90:270] plot ({7+0.05*cos(\x)}, {0.45+0.05*sin(\x)});
	\draw[domain=90:270] plot ({7+0.04*cos(\x)}, {0.66+0.04*sin(\x)});
	\draw[domain=90:270] plot ({7+0.03*cos(\x)}, {0.87+0.03*sin(\x)});
	\draw[thick] (5, 1)--(9,1);
	\node[circle,fill, inner sep=1] at (7,1){};
	\draw (5,1.445)--(9,1.445);
	\draw (5,1.23)--(9,1.23);
	\draw (5,1.05)--(9,1.05);
	\draw (5,0.95)--(9,0.95);
	\draw (5,0.77)--(9,0.77);
	\draw (5,0.56)--(9,0.56);
	\node at (7,-0.3) {($b$)};
	
		\draw[dashed] (4.5,-1) circle (1);
		\draw (2.5,-0.795)--(4.5,-0.795);
		\draw (2.5,-0.735)--(4.5,-0.735);
		\draw (2.5,-0.59)--(4.5,-0.59);
		\draw (2.5,-0.52)--(4.5,-0.52);
		\draw (2.5,-0.35)--(4.5,-0.35);
		\draw (2.5,-0.25)--(4.5,-0.25);
		\draw (4.5,-0.38)--(6.5,-0.38);
		\draw (4.5,-0.48)--(6.5,-0.48);
		\draw (4.5,-0.62)--(6.5,-0.62);
		\draw (4.5,-0.7)--(6.5,-0.7);
		\draw (4.5,-0.84)--(6.5,-0.84);
		\draw (4.5,-0.9)--(6.5,-0.9);
		\draw[domain=270:450] plot ({4.5+0.05*cos(\x)}, {-0.3+0.05*sin(\x)});
		\draw[domain=270:450] plot ({4.5+0.035*cos(\x)}, {-0.555+0.035*sin(\x)});
		\draw[domain=270:450] plot ({4.5+0.03*cos(\x)}, {-0.765+0.03*sin(\x)});
		\draw[domain=90:270] plot ({4.5+0.05*cos(\x)}, {-0.43+0.05*sin(\x)});
		\draw[domain=90:270] plot ({4.5+0.04*cos(\x)}, {-0.66+0.04*sin(\x)});
		\draw[domain=90:270] plot ({4.5+0.03*cos(\x)}, {-0.87+0.03*sin(\x)});
		\draw[thick] (2.5, -1)--(6.5,-1);
		\node[circle,fill, inner sep=1] at (4.5,-1){};
		\draw (2.5,-1.5)--(6.5,-1.5);
		\draw (2.5,-1.33)--(6.5,-1.33);
		\draw (2.5,-1.05)--(6.5,-1.05);
		\draw (2.5,-1.1)--(6.5,-1.1);
		\draw (2.5,-1.2)--(6.5,-1.2);
		\node at (4.5,-2.3) {($c$)};
	
	\end{tikzpicture}
	\caption{Neighbourhoods of folding points. In Case (a) and (c) folding point is accessible, while in Case (b) it is not.}
	\label{fig:folding}
\end{figure}
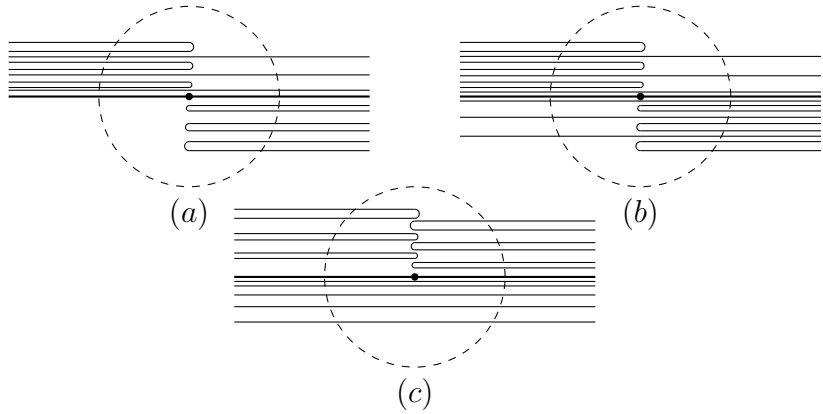

\iffalse
\begin{figure}[!ht]
	\centering
	\begin{tikzpicture}[scale=1.2]
	\draw[dashed] (2,1) circle (1);
	\draw (0,0.795)--(2,0.795);
	\draw (0,0.735)--(2,0.735);
	\draw (0,0.59)--(2,0.59);
	\draw (0,0.52)--(2,0.52);
	\draw (0,0.35)--(2,0.35);
	\draw (0,0.25)--(2,0.25);
	\draw (2,0.38)--(4,0.38);
	\draw (2,0.48)--(4,0.48);
	\draw (2,0.62)--(4,0.62);
	\draw (2,0.7)--(4,0.7);
	\draw (2,0.84)--(4,0.84);
	\draw (2,0.9)--(4,0.9);
	\draw[domain=270:450] plot ({2+0.05*cos(\x)}, {0.3+0.05*sin(\x)});
	\draw[domain=270:450] plot ({2+0.035*cos(\x)}, {0.555+0.035*sin(\x)});
	\draw[domain=270:450] plot ({2+0.03*cos(\x)}, {0.765+0.03*sin(\x)});
	\draw[domain=90:270] plot ({2+0.05*cos(\x)}, {0.43+0.05*sin(\x)});
	\draw[domain=90:270] plot ({2+0.04*cos(\x)}, {0.66+0.04*sin(\x)});
	\draw[domain=90:270] plot ({2+0.03*cos(\x)}, {0.87+0.03*sin(\x)});
	\draw[thick] (0, 1)--(4,1);
	\node[circle,fill, inner sep=1] at (2,1){};
	\draw (0,1.3)--(4,1.3);
	\draw (0,1.23)--(4,1.23);
	\draw (0,1.05)--(4,1.05);
	\draw (0,1.1)--(4,1.1);
	\draw (0,1.15)--(4,1.15);
	\node at (2,-0.3) {($c$)};
	
	\end{tikzpicture}
	\caption{Neighbourhoods of folding points.}
	\label{fig:folding2}
\end{figure}
\fi

\section{Tops/bottoms of finite cylinders}\label{sec:tops/bottoms}

In this section we study the symbolics of tops/bottoms of cylinders depending on an $\mathcal{E}$-embedding of $X'$ and we restrict to cases where $L\neq 0^{\infty}l_{n}\ldots l_{1}$ for all $n\in \N$.

For $t\in\{0, 1\}$, we denote by $t^*=1-t$. For $A=a_1\ldots a_n\in\{0,1\}^{n}$ denote by $\astl=a^*_1a_2\ldots a_n$, $\astr=a_1\ldots a_{n-1}a^*_n$ and $\ast=a^*_1a_2\ldots a_{n-1}a^*_n$.

\begin{definition}
	Let $\nu$ be a kneading sequence. We say that a finite word $a_1\ldots a_n\in\{0, 1\}^n$ is \emph{irreducibly non-admissible} if it is not admissible and $a_2\ldots a_n$ is admissible.
\end{definition}
\begin{definition}\label{def:dependence}
	Fix a kneading sequence $\nu$. We say that a finite cylinder $B=[b_n\ldots b_1]$ of length $n\in\N$ \emph{alters} $L=\ldots l_{2}l_{1}$, if there exist words $(A_i)_{i\in\N}$ such that $\ldots A_3A_2A_1=\ldots l_{n+2}l_{n+1}$ and the words
	$A_1B$ and $\astr_{i} \ast_{i-1}\ldots \ast_2 \astl_1B$ are irreducibly non-admissible for every $i\geq 2$. 
\end{definition}

\begin{proposition}
	If a finite cylinder $B$ alters the admissible sequence $L$ then $L_B$ or $S_B$ has different tail than $L$.
\end{proposition}
\begin{proof}
	Assume $B$ alters $L$ with words $A_i$ as in the definition. If $\#_1(B)-\#_1(l_{n}\ldots l_{1})$ is even, then $L_B=\ldots\ast_{i} \ast_{i-1}\ldots \ast_2 \astl_1B$. The sequence $L_{B}$ differs from $L$ on infinitely many places. If $\#_1(B)-\#_1(l_{n}\ldots l_{1})$ is odd, then $S_B=\ldots\ast_{i} \ast_{i-1}\ldots \ast_2 \astl_1B$.
\end{proof}

The following example shows that there exist $\mathcal{E}$-embeddings of $X'$ such that none of the extrema of certain cylinders are contained in $\mathcal{U}_L$.

\begin{example}\label{ex:3}
 Let $\nu=(100111011)^{\infty}$ and $L=(001)^{\infty}11$. Note that $S_{10}=(100)^{\infty}(101)10\\
 \subset\mathcal{U}_{L_{10}}$ and $L_{10}=(100)^{\infty}10\subset\mathcal{U}_{L_{10}}$. Therefore, $L_{10}, S_{10}\not\subset\mathcal{U}_L$. 
\end{example} 

In Example~\ref{ex:3} both extrema belong to the same arc-component. This is not necessarily always the case, see \eg Example~\ref{ex:4} below.

\begin{proposition}
	If a finite cylinder $B$ is such that $L_B\not\subset\mathcal{U}_L$ or $S_B\not\subset\mathcal{U}_L$, then there exists a finite cylinder $B'$ such that $B'$ alters $L$.
\end{proposition}
\begin{proof}
	Assume $\#_1(B)-\#_1(l_{n}\ldots l_{1})$ is even and $L_B\not\subset\mathcal{U}_L$. Then obviously $B'=B$ alters $L$. Similarly, if $\#_1(B)-\#_1(l_{n}\ldots l_{1})$ is odd and $S_B\not\subset\mathcal{U}_L$. So assume $\#_1(B)-\#_1(l_{n}\ldots l_{1})$ is even and $S_B\not\subset\mathcal{U}_L$. Then $l^*_{n+1}B$ alters $L$, if $l^*_{n+1}B$ is admissible. If $l^*_{n+1}B$ is not admissible, there exists $i\in \N$ such that $l^{*}_{n+i}\ldots l_{n}B$ is admissible, since otherwise $S_B=L_B$, which is a contradiction. The word $l^{*}_{n+i}\ldots l_{n}B$ alters $L$. Analogously if  $\#_1(B)-\#_1(l_{n}\ldots l_{1})$ is odd and $L_B\not\subset\mathcal{U}_L$.
\end{proof}

\begin{example}\label{ex:4}
Let $\nu=1001(101)^{\infty}$ and $L=((001)(001101))^{\infty}$. Then $S=S_{0}=((100)(101100))^{\infty}\not\subset\mathcal{U}_L$. So $B=0$ alters $L$ and words $A_i$ are divided by brackets.
\end{example}

Next we show there exist $\mathcal{E}$-embeddings with more than two accessible arc-components.

\begin{proposition}\label{prop:third}
	Assume that $\nu$ starts with some finite words $\nu=1B\ldots=1ABA\ldots$, where $B^*$ and $ABA^*$ are irreducibly non-admissible. The embedding of $X'$ with respect to $L=(BA)^{\infty}$ contains at least three tails which are extrema of cylinders.
\end{proposition}
\begin{proof}
	Note that $S=({^*\!\!B^*}{^*\!\!A}BA^*)^{\infty}$. Take any admissible word $D$ such that $|D|=|A|$ and such that $\#_1(D)-\#_1(A)$ is even. Then $S_{D}=({^*\!\!A}B{A^*}{^*\!\!B}^*)^{\infty}D$  and therefore we found three different tails which are extrema of cylinders.
\end{proof}

The following example shows that it is indeed possible to satisfy the conditions of Proposition~\ref{prop:third}.

\begin{example}\label{ex:5}
	  Take $\nu=1001100100111\ldots$, $B=001$, $A=0011$ and $L=(BA)^{\infty}$ which is easily checked to be admissible. For $D$ take \eg $D=1111$. Note that $S=({^*\!\!B^*}{^*\!\!A}BA^*)^{\infty}$ and $S_{D}=({^*\!\!A}B{A^*}{^*\!\!B}^*)^{\infty}D$ and thus we obtain
three accessible basic arcs with different tails. If we take \eg $\nu=(10011001001111)^{\infty}$, since by Remark~\ref{rem:fp} the only folding points are endpoints and there are no spiral points in $X'$, it follows by \black{Proposition~\ref{prop:wiggles}} that there are three fully accessible non-degenerate dense arc-components. Moreover, none of those arc-components contains an endpoint so they are all lines. We will return to this particular example in Section~\ref{sec:FP}, Example~\ref{ex:9}.
\end{example}

\section{Accessible folding points}\label{sec:FP}
In this section we study accessibility of folding points which are not at the top or the bottom of any cylinder. 

\subsection{Accessible endpoints}

Let us fix $X'$ and the $\mathcal{E}$-embedding depending on $L$.
Recall that we denote by $\mathcal{U}_{L}$ the arc-component of $x\in A(L)\subset X'$. 
By Proposition~\ref{prop:fully}, every point with the same symbolic tail as $L$ is accessible. 

The following remark is a direct consequence of Proposition~\ref{prop:endpt}. 

\begin{remark}
	If $e\in X'$ is an endpoint \black{with a single itinerary $\ovl e$}, then there exists a strictly increasing sequence $(n_i)_{i\in \N}$ such that 
	$\bar{e}=\ldots e_{n_{i}+1}c_1\ldots c_{n_i}.c_{n_{i}+1}\ldots=\ldots e_{n_{i}+1}\nu$ for every $i\in \N$. \black{If an endpoint $e\in X$ has two itineraries, then one of them will have the properties above. That itinerary of $e$ will always be denoted by $\bar e$. So specially by $\ovl e$ we always mean the left-itinerary of $\bar e$. Note that, since the two itineraries are identified by $\sim$ when constructing embeddings of $X$ this will not affect accessibility of $e$ and thus we may work with either of them.}
\end{remark}

In this section we work with the concept of an endpoint being capped which is defined below.  See Figure~\ref{fig:cap}.

\begin{definition}\label{def:cap} 
	Let $e\in X'$ be an endpoint with $\tau_{L}(\overleftarrow{e})=\infty$ ($\tau_{R}(\overleftarrow{e})=\infty$).
	We say that a point $e$ is \emph{capped from the left (right)}, if there exist sequences of admissible itineraries $(\overleftarrow{y}^{i})_{i\in \N},
	(\overleftarrow{w}^{i})_{i\in \N}
	\subset \{0,1\}^{\infty}$ such that 
	$\overleftarrow{y}^{i},\overleftarrow{w}^{i}\rightarrow \overleftarrow{e}$ as $i\rightarrow \infty$,
	$\overleftarrow{y}^{i}\Ls \overleftarrow{e}\Ls \overleftarrow{w}^{i}$ for every $i\in \N$ and  
	arcs $A(\overleftarrow{y}^{i})$ and $A(\overleftarrow{w}^{i})$ are joined on the left (right).
\end{definition}

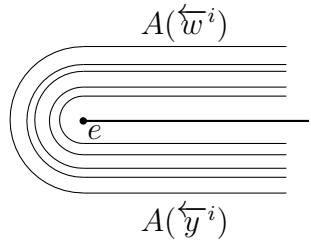
\begin{figure}[!ht]
	\centering
	\begin{tikzpicture}[scale=1.5, rotate=180]
	\draw (0.2,0.34)--(2,0.34);
	\draw (0.2,0.5)--(2,0.5);
	\draw (0.2,0.56)--(2,0.56);
	\draw (0.2,0.7)--(2,0.7);
	\draw (0.2,0.78)--(2,0.78);
	\draw (0.2,1.64)--(2,1.64);
	\draw (0.2,1.5)--(2,1.5);
	\draw (0.2,1.42)--(2,1.42);
	\draw (0.2,1.3)--(2,1.3);
	\draw (0.2,1.2)--(2,1.2);
	\draw[domain=90:-90] plot ({2+0.649*cos(\x)}, {0.99+0.649*sin(\x)});
	\draw[domain=90:-90] plot ({2+0.43*cos(\x)}, {0.99+0.43*sin(\x)});
	\draw[domain=90:-90] plot ({2+0.5*cos(\x)}, {1+0.5*sin(\x)});
	\draw[domain=90:-90] plot ({2+0.21*cos(\x)}, {0.99+0.21*sin(\x)});
	\draw[domain=90:-90] plot ({2+0.3*cos(\x)}, {1+0.3*sin(\x)});
	\draw[thick] (0, 1)--(2,1);
	\node[circle,fill, inner sep=1] at (2,1){};
	\node at (1.9,1.1) {$e$};
	\node at (1.1,1.9) {$A(\overleftarrow{y}^{i})$};
	\node at (1.1,0.15) {$A(\overleftarrow{w}^{i})$};
	\end{tikzpicture}
	\caption{Endpoint $e$ is capped from the left.} 
	\label{fig:cap}
\end{figure}

\begin{remark}\label{rem:notcapped}
If $e\in X'$ is a right (left) endpoint which is not capped from the right (left), then $e$ is accessible by a horizontal \black{segment} in the plane.
	Note that if $\ovl{e}$  lies on an extremum of a cylinder (which holds if \eg $e$ has the same symbolic tail as $L$), then $e$ is not capped. 
\end{remark}

\begin{remark}\label{rem:endptKnaster}
	Let $\nu=10^{\infty}$, \ie $X=X'$ is a Knaster continuum and let $L$ be arbitrary. Note that any two points $x,y\in X'$ that are $\eps>0$ close to the point $\bar{0}$ and are identified have the form $x_{k}x_{k-1}\ldots x_{1}=y_{k}y_{k-1}\ldots y_{1}=10^{k-1}$ for some $k\in\N$. It follows that either $\overleftarrow{x}, \overleftarrow{y}\Ls\overleftarrow{0}$ or $\overleftarrow{x}, \overleftarrow{y}\Ll\overleftarrow{0}$, depending on the parity of $\#_1(l_{k-1}\ldots l_1)$. Therefore, the endpoint $\bar{0}\in X'$ is not capped and thus always accessible in $\mathcal{E}$-embeddings of the Knaster continuum, see Figure~\ref{fig:endpoint_zero}. 
\end{remark}

From now on we assume in this subsection that $X'$ is not the Knaster continuum and thus $\nu\neq 10^{\infty}$.

\begin{figure}[!ht]
	\centering
	\begin{tikzpicture}[scale=1.5]
	\draw (0,1.6)--(2,1.6);
	\draw (0,1.5)--(2,1.5);
	\draw (0,1.38)--(2,1.38);
	\draw (0,1.3)--(2,1.3);
	\draw (0,1.16)--(2,1.16);
	\draw (0,1.1)--(2,1.1);
	\draw (0,0.4)--(2,0.4);
	\draw (0,0.5)--(2,0.5);
	\draw (0,0.62)--(2,0.62);
	\draw (0,0.7)--(2,0.7);
	\draw (0,0.84)--(2,0.84);
	\draw (0,0.9)--(2,0.9);
	\draw[domain=90:270] plot ({0.05*cos(\x)}, {1.55+0.05*sin(\x)});
	\draw[domain=90:270] plot ({0.04*cos(\x)}, {1.34+0.04*sin(\x)});
	\draw[domain=90:270] plot ({0.03*cos(\x)}, {1.13+0.03*sin(\x)});
	\draw[domain=90:270] plot ({0.05*cos(\x)}, {0.45+0.05*sin(\x)});
	\draw[domain=90:270] plot ({0.04*cos(\x)}, {0.66+0.04*sin(\x)});
	\draw[domain=90:270] plot ({0.03*cos(\x)}, {0.87+0.03*sin(\x)});
	\draw[thick] (0, 1)--(2,1);
	\node[circle,fill, inner sep=1] at (0,1){};
	\node at (-0.2,1) {$\bar{0}$};
	
	\end{tikzpicture}
	\caption{Neighbourhood of the end-point $\bar{0}$ of the Knaster continuum ($\nu=10^{\infty}$) in an $\mathcal{E}$-embedding.}
	\label{fig:endpoint_zero}
\end{figure}
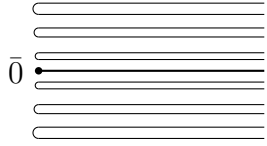

It is well known (see \eg \cite{BaMa}) that $X'$ contains endpoints if and only if the critical point $c$ of map $T$ is recurrent (\ie $T^n(c)$ get arbitrary close to $c$ as $n\to \infty$).

\begin{definition}\label{def:complete}
Fix a kneading sequence $\nu$ and let $e\in X'$ be an endpoint and thus $\tau_L(\overleftarrow{e})=\infty$ ($\tau_R(\overleftarrow{e})=\infty$).  
A sequence $(m_i)_{i\in\N}\subset \N$ is called the \emph{complete sequence for $e$}, if for every $n\in \N$ such that $e_{n}\ldots e_{1}=c_1c_2\ldots c_{n}$ and $\#_1(c_1c_2\ldots c_{n})$ is odd (even) there exist $i\in \N$ such that $m_i=n$.
\end{definition}

From $\tau_{L}(\ovl{e})=\infty$ (or $\tau_{R}(\ovl{e})=\infty$) it follows that the sequence $(m_i)_{i\in\N}$ indeed exists.\\
The main result in this subsection is that every endpoint of $X'$ (where $X'$ is not the Knaster continuum) which is not contained in $\mathcal{U}_L$ is capped in an $\mathcal{E}$-embedding of $X'$ which is not equivalent to Brucks-Diamond embedding from \cite{BrDi}.
In the proof of Theorem~\ref{thm:capped} we construct an increasing subsequence $(n_i)_{i\in\N}\subset (m_i)_{i\in \N}$ and basic arcs  $A(\ovl{x}^{O(i)}),A(\ovl{x}^{I(i)})\subset \mathcal{R}\subset  X'$ such that
\begin{equation}\label{eq:semicirc}
\ovl{x}^{O(i)}=1^{\infty}a^i_k\ldots a^i_1 0c_1 c_2\ldots c_{n_i} \hspace{0,5cm}
\ovl{x}^{I(i)}=1^{\infty}a^i_k\ldots a^i_1 1c_1 c_2\ldots c_{n_i}.
\end{equation}
and $\overleftarrow{x}^{O(i)}\Ls\overleftarrow{e}\Ls\overleftarrow{x}^{I(i)}$ or $\overleftarrow{x}^{I(i)}\Ls\overleftarrow{e}\Ls\overleftarrow{x}^{O(i)}$ for some admissible word $a^i_k\ldots a^i_1\in\{0,1\}^{k}$. Note that the arcs  $A(\ovl{x}^{O(i)})$ and $A(\ovl{x}^{I(i)})$ are joined by left (right) semi-circle. Here $\mathcal{R}$ denotes the arc-component of the point $\bar{1}$ which is a dense line in $X'$ independently on the choice of $\nu$ (see Proposition 1 in \cite{BB}).

\begin{remark}\label{rem:per1}
	Let $e\in X'$ be an endpoint and thus $\tau_L(\overleftarrow{e})=\infty$ ($\tau_R(\overleftarrow{e})=\infty$). Then $\#_{1}(c_1\ldots c_{m_{i}})$ is odd (even) and $\#_{1}(c_{1}\ldots c_{m_{i+1}-m_{i}})$ is even (even) for every $i\in \N$. \black{This follows from the fact that subtracting both odd from odd number and even from even number results in an even number.}
\end{remark}

\begin{remark}
	Fix the kneading sequence $\nu$. Assume that $a_{n-1}\ldots a_1\in\{0,1\}^{n-1}$ is admissible but $a_n\ldots a_1\in\{0,1\}^{n}$ is not. Then  $a_n\ldots a_1\prec c_2\ldots c_{n+1}$.
\end{remark}

\begin{lemma}\label{lem:odd}
	Let $\nu$ be an admissible kneading sequence. A word $c_2\ldots c^*_n$ is not admissible if and only if either $\#_1(c_2\ldots c_n)$ is odd or there exists $k\in\{3, \ldots, n\}$ such that $c_k\ldots c_n=c_2\ldots c_{n-k+2}$ and $\#_1(c_k\ldots c_n)$ is odd.
\end{lemma}
\begin{proof}
	Assume that $c_2\ldots c^{*}_n$ is not admissible, so there exists $i\in\{2, \ldots, n\}$ such that $c_i\ldots c^{*}_n$ is not admissible. Take the largest such index $i$ and note that $c_i\ldots c_n=c_2\ldots c_{n-i+2}$ and $c_2\ldots c_{n-i+2}^*\prec c_2\ldots c_{n-i+2}$. Let us assume by contradiction that $\#_1(c_2\ldots c_{n-i+2})$ is even. If $c_{n-i+2}=0$ $(c_{n-i+2}=1)$ it follows that $\#_1(c_2\ldots c_{n-i+1})$ is even (odd) and in both cases $c_2\ldots c^{*}_{n-i+2}\succ c_2\ldots c_{n-i+2}$ and thus $c_2\ldots c^{*}_{n-i+2}$ is admissible, a contradiction.
\end{proof}

\begin{lemma}\label{lem:per2}
	Let $\nu$ be an admissible kneading sequence and let $(m_i)_{i\in\N}$ be the complete sequence for an endpoint $e\in X'$. Then for every natural number $k\geq 3$ and $j\in\{0,\ldots, m_i\}$, the word $c_{k}\ldots c^{*}_{m_{i+1}-{m_i}}c_1c_2\ldots c_j$ is admissible for every $i\in\N$. Specifically, if $j=0$, we set $c_1\ldots c_j=\emptyset$.
\end{lemma}

\begin{proof}
	Assume by contradiction that there exists $k\geq 3$ and $j\in\N_0$ such that the word $c_{k}\ldots c^{*}_{m_{i+1}-{m_i}}c_1c_2\ldots c_j$ is not admissible and assume that $k$ is the largest and $j$ is the smallest such index. By the choice of $k$ and $j$ every proper subword of $c_{k}\ldots c^{*}_{m_{i+1}-{m_i}}c_1c_2\ldots c_j$ is admissible. Thus $c_{k}\ldots
	c^{*}_{m_{i+1}-{m_i}}c_1c_2\ldots c_j=c_2c_3\ldots
	 c_{m_{i+1}-m_i-k}\\
	 c_{m_{i+1}-m_i-k+1}\ldots c^{*}_{m_{i+1}-{m_i}-k+j+1}$ and $\#_1(c_2c_3\ldots c^{*}_{m_{i+1}-{m_i}-k+j+1})$ is even by Lemma~\ref{lem:odd}.  Furthermore, Lemma~\ref{lem:odd} implies that $\#_1(c_{k}\ldots
	c^*_{m_{i+1}-{m_i}}=c_2\ldots c_{m_{i+1}-m_{i}-k-1})$ is even. 
	
	If $j=1$, then both $\#_1(c_{k}\ldots c^*_{m_{i+1}-{m_i}})$ and $\#_1(c_k\ldots c^*_{m_{i+1}-{m_i}}c_1)$ are even, which is impossible.
	
	If $j\geq 2$, it follows by Lemma~\ref{lem:odd} that $\#_1(c_2\ldots c_{j})$ is odd. Thus $c_2\ldots c^{*}_{j}=c_{m_{i+1}-m_i-k+1}\\
	\ldots c_{m_{i+1}-{m_i}-k+j+1}$ is not admissible, which is a contradiction, since $c_2c_3\ldots c^{*}_j$ is a subword of $\nu$.
	 
	Let $c_1\ldots c_j$ be an empty word. Then $c_{k}\ldots
	c_{m_{i+1}-{m_i}}=c_2c_3\ldots c_{m_{i+1}-m_i-k}c_{m_{i+1}-m_i-k+1}$ and $\#_1(c_2c_3\ldots c_{m_{i+1}-m_i-k}c_{m_{i+1}-m_i-k+1})$ is odd. Let $l$ be the maximal natural number such that $c_{m_{i+1}-m_i+1}\ldots
	 c_{m_{i+1}-m_i+l}=c_{m_{i+1}-m_i-k+2}\ldots c_{m_{i+1}-m_i-k+l+1}$, \ie 
	$$c_{k}\ldots c_{m_{i+1}-m_i+l}=c_2c_3\ldots c_{m_{i+1}-m_i-k+l+1}$$
	and $c_{m_{i+1}-m_i+l+1}\neq c_{m_{i+1}-m_i-k+l+2}$. Such $l$ indeed exists since $(m_i)$ is complete.
	  Note that $c_{m_{i+1}-m_i-k+2}\ldots c_{m_{i+1}-m_i-k+l+1}=c_1\ldots c_{l}$ and $\#_1(c_1\ldots c_{l+1})$ is odd by Lemma~\ref{lem:odd}. Thus $\#_1(c_1\ldots c_{l}c^{*}_{l+1})$ is even and we conclude that $\#_1(c_2\ldots c_{m_{i+1}-m_i-k+l+2})$ is odd. Since $c_2\ldots c^{*}_{m_{i+1}-m_i-k+l+2}=c_{k}\ldots c_{m_{i+1}-m_i+l+1}$ is admissible, we get a contradiction.
\end{proof}

The main idea of the proof of the following theorem is illustrated in Example~\ref{ex:6}.

\begin{theorem}\label{thm:capped}
	Let $e\in X'$ be an endpoint such that $\tau_R(\ovl{e})=\infty$ ($\tau_L(\ovl{e})=\infty$) and let $L=\ldots l_2l_1\neq 0^{\infty}l_{n}\ldots l_{1}$ be admissible and $\nu\neq 10^{\infty}$. If $L$ and $\ovl e$ have different tails, then $e$ is capped from the right (left).
\end{theorem}

\begin{proof}
 Let $(m_i)_{i\in\N}\subset \N$ be the complete sequence for an endpoint $e$ where $\tau_R(\ovl{e})=\infty$. The proof works analogously if $\tau_L(\ovl{e})=\infty$.
	We will find infinitely many $i\in\N$ such that   $\overleftarrow{x}^{O(i)}\Ll\overleftarrow{e}
	\Ll\overleftarrow{x}^{I(i)}$ (or with reversed inequalities) and arcs $\ovl{x}^{O(i)}$ and  $\ovl{x}^{I(i)}$ are joined by a semi-circle on the right.
	
	Fix some $i\in\N$ and let $M_1(i)>m_i+1$ be the smallest natural number such that $l_{M_1(i)}=e^{*}_{M_1(i)}$. Note that such $M_1(i)$ exists, otherwise $\ovl e$ and $L$ have the same tail.
	
	Assume that $M_1(i)=m_{i+1}$. Note that then $e_{M_1(i)-1}\ldots e_{m_i+1}=c_2\ldots c_k=0^{\kappa}1c_{\kappa+2}\ldots c_k$ ($c_{\kappa+2}\ldots c_k$ can be empty) and $e_{M_1(i)}=1$. Then $l_{M_1(i)}\ldots l_{m_i+1}=0^{\kappa+1}1c_{\kappa+3}\ldots c_k$, which is not admissible.
	
	Assume that $M_1(i)=m_{i+1}+1$. By the paragraph above $M_1(i+1)\neq m_{i+2}$. If $M_1(i+1)=m_{i+2}+1$, then $l_{m_{i+2}}\ldots l_{m_{i+1}+2}l_{m_{i+1}+1}=c_1\ldots c_{m_{i+2}-m_{i+1}+1}c_{m_{i+2}-m_{i+1}}^*$ which is not admissible since $c_1\ldots c_{m_{i+2}-m_{i+1}}$ is even by Remark~\ref{rem:per1}.
	So either $M_1(i)\in\{m_{i}+2, \ldots m_{i+1}-1\}$ or there is $k\in\N$ such that $M_1(i+k)=M_1(i)$. Note that there is infinitely many $i\in\N$ such that $M_1(i)\in\{m_{i}+2, \ldots m_{i+1}-1\}$ and from now on we work with such $i\in\N$.

	If both of the following sequences are admissible, we set:
	$$
	\overleftarrow{x}^{O(i)}=1^{\infty}e^{*}_{M_1(i)}e_{M_1(i)-1}\ldots e_{m_i+2}0e_{m_i}\ldots e_{1},
	$$
	$$
	\overleftarrow{x}^{I(i)}=1^{\infty}e^{*}_{M_1(i)}e_{M_1(i)-1}\ldots e_{m_i+2}1e_{m_i}\ldots e_{1}.
	$$
	\textbf{a)} Assume that $\#_1(e_{m_i}\ldots e_{1})$ and $\#_1(l_{m_i}\ldots l_{1})$ have the same parity and $e_{m_i+1}\\=l_{m_i+1}=0$.\newline
	Then it follows that $\overleftarrow{e} \Ll \overleftarrow{x}^{I(i)}$. Because $l_{M_1(i)-1}\ldots l_{m_{i}+2}=e_{M_1(i)-1}\ldots e_{m_{i}+2}$ the parities of $\#_1(e_{M_1(i)-1}\ldots e_{1})$ and $\#_1(l_{M_1(i)-1}\ldots l_{1})$ are the same and because $l_{M_1(i)}=e^{*}_{M_1(i)}$ it follows that $\overleftarrow{x}^{O(i)}\Ll \overleftarrow{e}$. 
	
	\textbf{b)} Assume that $\#_1(e_{m_i}\ldots e_{1})$ and $\#_1(l_{m_i}\ldots l_{1})$ have the same parity and $e_{m_i+1}=l_{m_i+1}=1$.\newline
	Then it follows that $\overleftarrow{e} \Ll \overleftarrow{x}^{O(i)}$. Because $l_{M_1(i)-1}\ldots l_{m_{i}+2}=e_{M_1(i)-1}\ldots e_{m_{i}+2}$ the parities of $\#_1(e_{M_1(i)-1}\ldots e_{1})$ and $\#_1(l_{M_1(i)-1}\ldots l_{1})$ are the same and because $l_{M_1(i)}=e^{*}_{M_1(i)}$ it follows that $\overleftarrow{x}^{I(i)}\Ll \overleftarrow{e}$.
	
	\textbf{c)} Assume that $\#_1(e_{m_i}\ldots e_{1})$ and $\#_1(l_{m_i}\ldots l_{1})$ have the same parity and $e_{m_i+1}=1\neq 0=l_{m_i+1}$.\\
	Then $\ovl x^{O(i)}\Ll\ovl e$. Since the parities of $\#_1(e_{M_1(i)-1}\ldots e_{1})$ and $\#_1(l_{M_1(i)-1}\ldots l_{1})$ are different and $l_{M_1(i)}=e^*_{M_1(i)}$, it follows that $\ovl e\Ll\ovl x^{I(i)}$. 
	
	\textbf{d)} Assume that $\#_1(e_{m_i}\ldots e_{1})$ and $\#_1(l_{m_i}\ldots l_{1})$ have the same parity and $e_{m_i+1}=0\neq 1=l_{m_i+1}$.\\
	Then $\ovl x^{I(i)}\Ll\ovl e$. Since the parities of $\#_1(e_{M_1(i)-1}\ldots e_{1})$ and $\#_1(l_{M_1(i)-1}\ldots l_{1})$ are different and $l_{M_1(i)}=e^*_{M_1(i)}$, it follows that $\ovl e\Ll\ovl x^{O(i)}$. 
	
	Note that if $\#_1(e_{m_i}\ldots e_{1})$ and $\#_1(l_{m_i}\ldots l_{1})$ are of different parities, then all the inequalities in cases a), b), c) and d) are reversed and we use analogous arguments to conclude that either  $\overleftarrow{x}^{O(i)}\Ls \overleftarrow{e} \Ls \overleftarrow{x}^{I(i)}$ or $\overleftarrow{x}^{I(i)}\Ls \overleftarrow{e} \Ls \overleftarrow{x}^{O(i)}$.
	
	\iffalse
	If one assumes that $l_{m_i+1}=e^{*}_{m_i+1}$ one can observe as in the paragraph above that finding $M_1(i)>m_{i}+1$ such that $l_{M_1(i)}=e^{*}_{M_1(i)}$ and setting $
	\overleftarrow{x}^{O(i)}=1^{\infty}e^{*}_{M_1(i)}e_{M_1(i)-1}\ldots\\ e_{m_i+2}0e_{m_i}\ldots e_{1}., 
	\overleftarrow{x}^{I(i)}=1^{\infty}e^{*}_{M_1(i)}e_{M_1(i)-1}\ldots e_{m_i+2}1e_{m_i}\ldots e_{1}.
	$, if both admissible, results in $\overleftarrow{x}^{O(i)}\Ll\overleftarrow{e}
	\Ll\overleftarrow{x}^{I(i)}$ (or with reversed inequalities).
	
	Let us study what are the possible values for indices $M_1(i)$.\\
	Say that $M_1(i)\geq m_{i+1}$. Because $\#_{1}(c_1\ldots c_{m_{i+1}-m_{i}})$ is even (the sequence $(m_i)_{i\in\N}$ is complete) it follows that one of $\overleftarrow{x}^{O(i)},\overleftarrow{x}^{I(i)}$ contains a subword $c_2\ldots c^{*}_{m_{i+1}-m_{i}}$ which is not admissible.\\
	Assume that $M_1(i)=m_{i+1}$. Note that $e_{M_1(i)-1}\ldots e_{m_i+1}=c_2\ldots c_k=0^{\kappa}1c_{2+\kappa+1}\ldots c_k$ ($c_{2+\kappa+1}\ldots c_k$ can be empty) and $e_{M_1(i)}=1$. Then $l_{M_1(i)}\ldots l_{m_i+1}=0^{\kappa+1}1c_{2+\kappa+1}\ldots c_k$, which is not admissible.\\
	Thus we can assume from now onwards that $m_{i}+1<M_1(i)<m_{i+1}$.
	\fi
	
	Now assume that one of $e^{*}_{M_1(i)}e_{M_1(i)-1}\ldots e_{m_i+2}e^{(*)}_{m_i+1}\ldots e_1$ is not admissible (where $s^{(*)}$ means $s^{*}$ or $s$). Then we set $x^{O(i)}_{M_1(i)}=x^{I(i)}_{M_1(i)}=e_{M_1(i)}$. If $e_{M_1(i)+1}=l_{M_1(i)+1}$, then we set $x^{O(i)}_{M_1(i)+1}=x^{I(i)}_{M_1(i)+1}=e^{*}_{M_1(i)+1}$ and we argue that $e^{*}_{M_1(i)+1}e_{M_1(i)}\ldots e_{m_i+2}e^{(*)}_{m_i+1}\ldots e_1=e^{*}_{M_1(i)+1}
	 10^{\kappa-1}1\ldots e_1$ are admissible words.  Indeed, word  $e_{M_1(i)}\ldots e_{m_i+2}
	 e^{(*)}_{m_i+1}\ldots e_1$ is admissible  by  Lemma~\ref{lem:per2}. If $e_{M_1(i)+1}^{*}10^{\kappa-1}1\ldots$ were not admissible, then $T^3(c)>T^4(c)$ which is a contradiction with $T$ being non-renormalizable.
	 So the following sequences are admissible:
	 	$$
	 	\overleftarrow{x}^{O(i)}=1^{\infty}e^{*}_{M_1(i)+1}e_{M_1(i)}\ldots e_{m_i+2}0e_{m_i}\ldots e_{1},
	 	$$
	 	$$\hspace{0.1cm}
	 	\overleftarrow{x}^{I(i)}=1^{\infty}e^{*}_{M_1(i)+1}e_{M_1(i)}\ldots e_{m_i+2}1e_{m_i}\ldots e_{1},
	 	$$
	and $\overleftarrow{x}^{O(i)}\Ll\ovl e\Ll\overleftarrow{x}^{I(i)}$ or $\overleftarrow{x}^{I(i)}\Ll\ovl e\Ll\overleftarrow{x}^{O(i)}$.
	
	Assume that $e^{*}_{M_1(i)+1}=l_{M_1(i)+1}$. Set $x^{O(i)}_{M_1(i)+1}=x^{I(i)}_{M_1(i)+1}=e_{M_1(i)+1}$. Then the words $e_{M_1(i)+1}e_{M_1(i)}\ldots e^{(*)}_{m_i+1}e_{m_i}\ldots e_1$ are admissible by  Lemma~\ref{lem:per2}, if $M_1(i)+1\neq m_{i+1}-1$.\\
	\iffalse
	Assume that there exists an index $m_{i+1}-2>P(i)>M_1(i)+1$ such that $e^{*}_{P(i)}=l_{P(i)}$ and $e^{*}_{P(i)}e^{P(i)-1}\ldots e_{m_i+2}e^{(*)}_{m_i+1}\ldots e_1=c_2c_3\ldots c_{P(i)-m_{i}+1}\ldots $ is not admissible. However this is not possible since the word $e^{*}$	
	Note that because $e^{*}_{M_1(i)}e_{M_1(i)-1}\ldots e_{m_i+2}e^{*}_{m_i+1}=c_2\ldots c^{*}_{M_1(i)-m_i}$ is not admissible it follows that $l_{M_1(i)+1}l_{M_1(i)}\ldots l_{m_i+1}=c_1c_2\ldots c_{M_1(i)-m_i+1}$ This means that $l_{P(i)}\ldots l^{*}_{m_i+1}=e^{*}\ldots e^{*}_{m_i+1}=c_2\ldots c_{P(i)-m_i}$.... 
	\fi
	Now say that $M_1(i)=m_{i+1}-2$. 
	By the assumption in the beginning of this paragraph, at least one of the words $e^{*}_{m_{i+1}-2}e_{m_{i+1}-3}
	\ldots e_{m_i+1}e^{(*)}_{m_i+1}\ldots e_1$ is not admissible. \\
	\textbf{a)} Say $\nu=10^{\kappa}1\ldots$, where $\kappa>1$. By  Lemma~\ref{lem:per2}, $e^{*}_{m_{i+1}-2}e_{m_{i+1}-3}\ldots=c^{*}_3c_4c_5\ldots=10^{\kappa-2}1\ldots$ is always admissible, a contradiction.\\
	\textbf{b)} Say that $\nu=10(11)^{n}0\ldots$. Then $e^{*}_{m_{i+1}-2}e_{m_{i+1}-3}\ldots=0(11)^{n-1}10
	\ldots$ is again always admissible, because $\#_1(0(11)^{n-1}1)$ is odd, a contradiction.\\
	
	Thus caps have been constructed except in the following case:\\
	(one of) $e^{*}_{M_1(i)}e_{M_1(i)-1}\ldots e_{m_i+2}e^{(*)}_{m_i+1}\ldots e_1$ is not admissible and  $e^{*}_{M_1(i)+1}=l_{M_1(i)+1}$.
	
	For $j>1$ denote by $M_j(i)$ the smallest $k\in\N$ such that $k>M_{j-1}(i)$ and $e^*_{k}=l_{k}$. By the previous paragraph, it follows that $M_2(i)< m_{i+1}-1$. Take the largest $N\in\N$ such that $M_N(i)<m_{i+1}-1$.  
	Note that for odd $j\in\{1, \ldots N\}$ and
		$$
		\overleftarrow{x}^{O(i)}=1^{\infty}e^{*}_{M_j(i)}e_{M_j(i)-1}\ldots e_{m_i+2}0e_{m_i}\ldots e_{1},
		$$
		$$\hspace{0.1cm}
		\overleftarrow{x}^{I(i)}=1^{\infty}e^{*}_{M_j(i)}e_{M_j(i)-1}\ldots e_{m_i+2}1e_{m_i}\ldots e_{1},
		$$
	if follows that $\overleftarrow{x}^{O(i)}\Ll\ovl e\Ll\overleftarrow{x}^{I(i)}$ or $\overleftarrow{x}^{I(i)}\Ll\ovl e\Ll\overleftarrow{x}^{O(i)}$. The conclusion follows from the fact that $\#_1(l_{M_j(i)-1}\ldots l_{m_i+2})$ and $\#_1(e_{M_j(i)-1}\ldots e_{m_i+2})$ are of the same parity since $j$ is odd.
	
	Assume that for every odd $j\in\{1, \ldots, N\}$ we have that $1^{\infty}e^{*}_{M_j(i)}
	e_{M_j(i)-1}\ldots e_{m_i+2}e_{m_i+1}^{(*)}\\
	e_{m_i}\ldots e_{1}$ are not admissible. If $M_{j+1}(i)>M_j(i)+1$, we set:
	$$
	\overleftarrow{x}^{O(i)}=1^{\infty}e^{*}_{M_j(i)+1}e_{M_j(i)}\ldots e_{m_i+2}0e_{m_i}\ldots e_{1},
	$$
	$$\hspace{0.1cm}
	\overleftarrow{x}^{I(i)}=1^{\infty}e^{*}_{M_j(i)+1}e_{M_j(i)}\ldots e_{m_i+2}1e_{m_i}\ldots e_{1},
	$$
	and argue that both are admissible as in preceding paragraphs. Calculations as above give $\overleftarrow{x}^{O(i)}\Ll\ovl e\Ll\overleftarrow{x}^{I(i)}$ or $\overleftarrow{x}^{I(i)}\Ll\ovl e\Ll\overleftarrow{x}^{O(i)}$. 
	
	\iffalse
	Note that there can exist finitely many $M_j(i)\in \{M_{j-1}(i)+1,\ldots m_{i+1}-2\}$ for some $j\in\{2\ldots N\}$ where $N\in\N$ such that $e^{*}_{M_j(i)}=l_{M_j(i)}$. We do the arguments analogous to ones we made in previous paragraphs for $M_1(i)$. 
	\fi
	
	The situation left to consider is when  $1^{\infty}e^{*}_{M_j(i)}e_{M_j(i)-1}\ldots e_{m_i+2}e_{m_i+1}^{(*)}e_{m_i}\ldots e_{1}$ are not admissible and $M_{j+1}(i)=M_j(i)+1$ for every odd $j\in\{1, \ldots, N\}$. Note that $N$ must be even. Otherwise $1^{\infty}e^{*}_{m_{i+1}-2}e_{m_{i+1}-3}\ldots e_{m_i+2}e_{m_i+1}^{(*)}e_{m_i}\ldots e_{1}$ are not admissible and we have already argued that this is not possible.
	
	Thus we conclude that $L$ is of the form:
	$$
	\ldots  e_{M_N(i)+1}e^{*}_{M_N(i)}e^{*}_{M_N(i)-1}e_{M_N(i)-2}\ldots e_{M_1(i)+2}e^{*}_{M_1(i)+1}e^{*}_{M_1(i)}e_{M_1(i)-1}\ldots e_{m_i+2}l_{m_{i}+1}\ldots l_{1}.
	$$

	\iffalse
	and for every $k\in \{m_{i}+1\ldots m_{i+1}-2\}$ either at least one of the words $A^{(*)}:=e^{*}_{k}\ldots e^{(*)}_{m_{i}+1}e_{m_{i}}\ldots e_{1}$ is not admissible or setting $\overleftarrow{x}^{O(i)}=\overleftarrow{x}^{I(i)}=1^{\infty}A^{(*)}$ implies that $\overleftarrow{x}^{O(i)},\overleftarrow{x}^{I(i)}\Ll\overleftarrow{e}$ (or with reversed inequalities).
	\fi
	
	Note that $\#_1(e^{*}_{M_N(i)}e^{*}_{M_N(i)-1}e_{M_N(i)-2}\ldots e_{M_1(i)+2}
	e^{*}_{M_1(i)+1}e^{*}_{M_1(i)}e_{M_1(i)-1}\ldots e_{m_{i}+2})$ is of the same parity as $\#_1(e_{M_N(i)}e_{M_N(i)-1}e_{M_N(i)-2}\ldots e_{M_1(i)+2}
	e_{M_1(i)+1}e_{M_1(i)}e_{M_1(i)-1}\ldots e_{m_{i}+2})$, because changes in $L$ compared with $\overleftarrow{e}$ always appear in pairs (as two consecutive letters). We set
	$$
	\overleftarrow{x}^{O(i)}=1^{\infty}e^{*}_{m_{i+1}-1}e_{m_{i+1}-2}\ldots e_{m_i+2}0e_{m_i}\ldots e_{1},
	$$
	$$\hspace{0.1cm}
	\overleftarrow{x}^{I(i)}=1^{\infty}e^{*}_{m_{i+1}-1}e_{m_{i+1}-2}\ldots e_{m_i+2}1e_{m_i}\ldots e_{1},
	$$
	and note that

	$\overleftarrow{x}^{O(i)}\Ll\overleftarrow{e}
	\Ll\overleftarrow{x}^{I(i)}$ (or with reversed inequalities). Also note that $\overleftarrow{x}^{I(i)}$ and $\overleftarrow{x}^{O(i)}$ set in such a way are always admissible by  Lemma~\ref{lem:per2} and since $e_{m_{i+1}-1}=c_2=0$.
	
	We have constructed the sequence corresponding to basic arc with the following properties: $\overleftarrow{x}^{O(i)}\Ls\overleftarrow{e}\Ls\overleftarrow{x}^{I(i)}$ or $\overleftarrow{x}^{I(i)}\Ls\overleftarrow{e}\Ls\overleftarrow{x}^{O(i)}$, $\overleftarrow{x}^{O(i)}$ and $\overleftarrow{x}^{I(i)}$ are joined on the right and $\overleftarrow{x}^{O(i)}, \overleftarrow{x}^{I(i)}\to \ovl e$ as $i\to\infty$. Since that can be done for infinitely many $i\in\N$, this concludes the proof.
\end{proof}

\begin{example}\label{ex:6}
Let $X'$ be the inverse limit space with the corresponding kneading sequence $\nu=(100111101011010111)^{\infty}$. Let us study the cappedness of the endpoint $e\in X'$ with the itinerary $\bar{e}=
(100111101011010
111)^{\infty}.
(100111101011010111)^{\infty}$ in an embedding determined by $L=(010111110011100
111)^{\infty}$. It follows that $\overleftarrow{x}^{O(i)}\Ls \overleftarrow{e}$, because $\#_1(100111101011010111)$ and $\#_1(0101111
10011100111)$ are both even. Note that  $M_1(i):=m_{i}+5$ is the smallest index strictly greater than $m_{i}+1$ such that $e^{*}_{M_1(i)}=l_{M_1(i)}$. We obtain the following situation:

\begin{center}
$$
\hspace{0.2cm}\ldots (1\mathbf{0}01111\mathbf{01}011\mathbf{01}0111)(100111101011010111)^i=\overleftarrow{e}
$$
$$
\ldots (0\mathbf{1}01111\mathbf{10}011\mathbf{10}0111)(010111110011100111)^i=L
$$
$$
\hspace{0.8cm}1^{\infty}(1\mathbf{1}01111\mathbf{01}011\mathbf{01}0110)(100111101011010111)^i=\overleftarrow{x}^{O(i)}
$$
$$  \hspace{0.8cm}1^{\infty}(1\mathbf{1}01111\mathbf{01}011\mathbf{01}0111)(100111101011010111)^i=\overleftarrow{x}^{I(i)}
$$ 
\end{center}

where we denoted with bold the letters of $\overleftarrow{e}$ and $L$ which differ for indices larger than $m_{i}$.
Note that $M_{3}(i)=m_{i}+10$ but the word $00110=e^{*}_{m_{i}+10}e_{m_{i}+9}\ldots e_{m_{i}+6}$ is not admissible and thus we need to set $x^{O(i)}_{M_{3}(i)}=x^{I(i)}_{M_{3}(i)}=e_{M_{3}(i)}=1$. Note that $M_{5}(i)=m_{i}+17=m_{i+1}-1$. Thus we set $x^{O(i)}_{M_{5}(i)}=x^{I(i)}_{M_{5}(i)}=e^{*}_{M_{5}(i)}$. Because $\#_{1}(e_{m_i+16}\ldots e_{1})$ and $\#_1(l_{m_i+16}\ldots l_1)$ are of the same parity we obtain that $\overleftarrow{e}\Ls\overleftarrow{x}^{I(i)}$.  Lemma~\ref{lem:per2} again ensures that every subword of $\ovl{x}^{O(i)}$ is admissible. Therefore points $x^{O(i)}, x^{I(i)}\in X'$ cap the point $e$ from the right. 
\end{example}

If an endpoint $e$ is capped, we still cannot conclude that it is not accessible, see \eg Figure~\ref{fig:endpoint_ngbhd}. However, if we know that the length of basic arcs arbitrary close to $\ovl e$ has a lower bound, the conclusion follows. Thus we introduce the notion of long-branchness in the following definition.

\begin{definition}
	Let $T\colon [0,1]\to [0,1]$ be a continuous map. The \emph{lap} of $T$ is a maximal interval of monotonicity of $T$ and a \emph{branch} of $T$ is an image of a lap. We say that $T$ is \emph{long-branched}, if there exists $\delta>0$ such that the length of all branches of $T^n$ is larger than $\delta$ for all $n\in\N$.
\end{definition}

\begin{remark}
	Note that if the critical point of $T$ is periodic, then $T$ is long-branched.
\end{remark}

\begin{corollary}\label{cor:endpoints_tau}
Assume $T\neq T_2$ is long-branched and let $e\in X'$ be an endpoint of $X'$. Assume $X'$ is embedded in the plane with respect to $L$ where $A(L)\not\subset\mathcal{C}$. If $\ovl e$ and $L$ have different tails, then $e$ is not accessible. 
\end{corollary}
\begin{proof}
	\black{By Theorem~\ref{thm:capped},  $e\in X'$ is capped. Since $T$ is long-branched it holds that there exists $\delta>0$ such that lengths of $A(\ovl y^i)$ and $A(\ovl w^i)$ (see Definition~\ref{def:cap}) are greater than $\delta$. It follows that $e$ is not accessible.}
\end{proof}

\begin{figure}[!ht]
	\centering
	\begin{tikzpicture}[scale=1.5]
	\draw (0,0.34)--(2,0.34);
	\draw (0,0.5)--(2,0.5);
	\draw (0,0.56)--(2,0.56);
	\draw (0,0.7)--(2,0.7);
	\draw (0,0.78)--(2,0.78);
	\draw (0,0.9)--(2,0.9);
	\draw (1,1.64)--(2,1.64);
	\draw (1,1.5)--(2,1.5);
	\draw (1.6,1.42)--(2,1.42);
	\draw (1.6,1.3)--(2,1.3);
	\draw (1.9,1.2)--(2,1.2);
	\draw (1.9,1.1)--(2,1.1);
	\draw[domain=90:270] plot ({1+0.07*cos(\x)}, {1.57+0.07*sin(\x)});
	\draw[domain=90:270] plot ({1.6+0.06*cos(\x)}, {1.36+0.06*sin(\x)});
	\draw[domain=90:270] plot ({1.9+0.05*cos(\x)}, {1.15+0.05*sin(\x)});
	\draw[domain=90:-90] plot ({2+0.649*cos(\x)}, {0.99+0.649*sin(\x)});
	\draw[domain=90:-90] plot ({2+0.43*cos(\x)}, {0.99+0.43*sin(\x)});
	\draw[domain=90:-90] plot ({2+0.5*cos(\x)}, {1+0.5*sin(\x)});
	\draw[domain=90:-90] plot ({2+0.21*cos(\x)}, {0.99+0.21*sin(\x)});
	\draw[domain=90:-90] plot ({2+0.3*cos(\x)}, {1+0.3*sin(\x)});
	\draw[domain=90:-90] plot ({2+0.1*cos(\x)}, {1+0.1*sin(\x)});
	\draw[thick] (0, 1)--(2,1);
	\node[circle,fill, inner sep=1] at (2,1){};
	\node at (1.72,1.1) {$e$};
	
	\end{tikzpicture}
	\caption{Neighbourhood of an endpoint $e$. Note that $e$ is capped but also accessible.} 
	\label{fig:endpoint_ngbhd}
\end{figure}
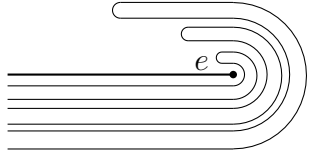

\iffalse
\begin{corollary}
	Let $X'$ be such that the corresponding $\nu$ is periodic. If $X'$ is $\mathcal{E}$-embedded in the plane such that there are $n$ fully accessible arc-components which contain no endpoints, then there are $n$ simple dense canals.
\end{corollary}
\fi

We merge the knowledge from this and the preceding section and give some interesting examples of embeddings of some $X'$.

\begin{example}\label{ex:7}
	 Let $\nu=(101)^{\infty}$ and let $L=(01^k)^{\infty}$ for any $k\geq 2$. Take an admissible $B=a_n\ldots a_1\in\{0,1\}^n$ for some $n\in\N$. If $l_{n+1}B$ is not admissible, then $\ldots l_{n+3}l^{*}_{n+2}l^*_{n+1}B$ is admissible by the choice of $k$ and since every non-admissible word for $\nu=(101)^{\infty}$ contains $00$. Tail $L$ is thus not altered by $B$ for every finite admissible word $B$ (recall Definition~\ref{def:dependence}). Therefore, it follows that $L_{a_n\ldots a_1}=S_{a_n\ldots a_1}\subset\mathcal{U}_L$. \black{Since $B$ is an arbitrary finite admissible word}, we conclude that $\mathcal{U}_L$ is fully accessible and it is the only non-degenerate accessible set. By Corollary~\ref{cor:endpoints_tau}, endpoints of $X'$ are not accessible. The remaining point on the circle of prime ends corresponds to the simple dense canal.
\end{example}

\begin{example}\label{ex:8} 
	Let $\nu=(101)^{\infty}$ and let $L=(01)^{\infty}$. Note that $S=(10)^{\infty}\not\subset\mathcal{U}_L$ and $S=S_{0}$. Thus, $B=0$ alters $L$ (recall Definition~\ref{def:dependence}; here $A_1=0$, $A_i=01$ for all $i\geq 2$). Since $\nu$ is periodic, it follows from Corollary~\ref{cor:fully} that both $\mathcal{U}_L$ and $\mathcal{U}_S$ are fully accessible. As in the example above we can show that no other point from $X'$ is accessible. We conclude that there are two simple dense canals with shores $\mathcal{U}_L$ and $\mathcal{U}_S$.
\end{example}

\begin{example}\label{ex:9}
 Take $\nu=(10011001001111)^{\infty}$, $B=001$, $A=0011$, $C=1111$ and $L=(BA)^{\infty}$ as in Example~\ref{ex:5}. Recall that at least three arc-components (which are dense lines) are fully accessible. Further calculations show that no other tail can be the top or the bottom of a cylinder. By Corollary~\ref{cor:endpoints_tau} endpoints of $X'$ are not accessible. Therefore, the remaining three points on the circle of prime ends correspond to three simple dense canals with shores from pairwise different fully accessible arc-components which are lines. In comparison, %the kneading sequence from this example has height 2/7 (see the Definition~\ref{def:height}) and belongs to the rational interior case, so 
 the Brucks-Diamond embedding of $X'$ contains 7 fully accessible arc-components which are shores of 7 simple dense canals (see Section~\ref{sec:BD} in this paper or \cite{3G}).
\end{example}

\subsection{Accessible folding points when $\nu$ is preperiodic}\label{subsec:accFP}	
\black{
First we state a characterization of folding points which we will use implicitly in this section very often.
Let $\omega(c)$ denote the set of all accumulation points of the forward orbit of the critical point $c$ by the map $T$.
\begin{proposition}\label{prop:foldpts}\cite[Theorem 2.2]{Raines}
	A point $x\in X$ is a folding point if and only if  $\pi_n(x)\in \omega(c)$ for every $n\in\N$.
\end{proposition}
}
In this subsection we assume that $\nu=c_1\ldots c_k(c_{k+1}\ldots c_{k+n})^{\infty}$ and that $c_k\neq c_{k+n}$, since otherwise also $\nu=c_1\ldots c_{k-1}
(c_k\ldots c_{k+n-1})^{\infty}$. By Remark~\ref{rem:fp} the space $X'$ contains $n$ folding points which are not endpoints with symbolic descriptions: 
$$\sigma^i((c_{k+1}\ldots c_{k+n})^{\infty}.(c_{k+1}\ldots c_{k+n})^{\infty})$$
for $i\in\{1, \ldots n\}$. In this subsection we study the accessibility of folding points that are not contained in extrema of cylinders in $\mathcal{E}$-embeddings of $X'$  when $\nu$ is preperiodic.

Let $Q\subset \R^2$ be an arc. From now onwards let $\mathrm{Int}(Q)$ denote the points from $Q$, which are not endpoints of $Q$.

\begin{remark}\label{rem:preper1FP}
	Let $\nu=c_1\ldots c_k(c_{k+1}\ldots c_{k+n})^{\infty}$ and let $p\in X'$ be a folding point. Then an arc-component of $p$ \black{in $X'$} can contain at most one folding point. Also, since $c_k\neq c_{n+k}$ it holds that $p\in \mathrm{Int}(A(\ovl p))$. \black{In specific, every folding point $p\in X'$ has a unique two sided infinite itinerary of zeros and ones assigned to it.}
\end{remark}

The following lemma restricts the search for accessible folding points which are not tops/bottoms of cylinders to the case where $\nu=10(c_3\ldots c_{n+2})^{\infty}$, \ie $k=2$.

\begin{proposition}\label{prop:restriction}
	Assume $c$ is preperiodic and such that $T^3(c)$ is not periodic. Embed $X'$ in the plane with respect to $L\neq 0^{\infty}l_{n}\ldots l_1$. A folding point $p\in X'$ is accessible if and only if the basic arc $A(\overleftarrow{p})$ is top or bottom of a finite cylinder.
\end{proposition}
\begin{proof}
	Note that $\nu=c_1\ldots c_k(c_{k+1}\ldots c_{k+n})^{\infty}$ where $k>2$.
	Take a folding point $p\in X'$ with the symbolic description
	$$\bar p=(c_{k+1}\ldots c_{k+n})^{\infty}c_{k+1}\ldots c_{k+i}.c_{k+i+1}\ldots c_{k+n}(c_{k+1}\ldots c_{k+n})^{\infty}$$
	and assume it is not on the top or bottom of any cylinder in $X'$.  Denote $\pi_0(A(\ovl p))=:[T^l(c), T^r(c)]$. By Remark~\ref{rem:preper1FP} it holds that $\pi_{0}(p)\in (T^l(c), T^r(c))$.\\
	\black{For $M\geq 0$} denote by $p^{M}\in X'$ \black{any} point with the symbolic description
	$$\bar p^{M}:=\black{\ldots} c_1\ldots c_k(c_{k+1}\ldots c_{k+n})^Mc_{k+1}\ldots c_{k+i}.c_{k+i+1}\ldots c_{k+n}(c_{k+1}\ldots c_{k+n})^{\infty}$$
	Note that points $p^{M}$ converge to $p$ as $M\to\infty$ and the corresponding basic arcs $A(\overleftarrow{p}^{M})$ project to $[T^l(c), T^{k+i+1}(c)]$ (we refer to them as \emph{left}) or $[T^{k+i+1}(c), T^r(c)]$ (referred to as \emph{right}) depending on the parity of $M$.  We will find long basic arcs (\ie arcs projecting with $\pi_0$ also to $[T^l(c), T^r(c)]$) converging to $A(\ovl p)$ from both sides. Since $c$ is preperiodic there exists a neighbourhood $U$ of $A(\ovl p)$ which contains only basic arcs which project to $[T^l(c), T^r(c)]$, $[T^l(c), T^{k+i+1}(c)]$ or $[T^{k+i+1}(c), T^r(c)]$ (\ie only long or left/right arcs).
	
	Assume that all but finitely many long arcs in $U$ are greater than $A(\ovl p)$. Since $k>2$, note that for every $M>0$ basic arcs $1^{\infty}c_{k}(c_{k+1}
	\ldots c_{k+n})^{M}c_{k+1}\ldots c_{k+i}$ are long.
	Since $c_k\neq c_{k+n}$ it holds that both $1^{\infty}c_{k}(c_{k+1}
	\ldots c_{k+n})^{M}c_{k+1}\ldots c_{k+i}\Ll\ovl p$ and $\ovl p^M\Ll \ovl p$.
	 Thus, it follows that $A(\ovl p)$ is at the bottom of some cylinder, a contradiction. The proof goes analogously if all but finitely many long arcs are smaller than $A(\ovl p)$.
\end{proof}

\iffalse
Next corollary is a direct consequence of Proposition~\ref{prop:restriction} and Proposition~\ref{prop:fourth}. 

\begin{corollary}
	Let $c$ be preperiodic and such that $T^3(c)$ is not periodic and let $X'$ be $\mathcal{E}$-embedded in the plane with respect to $L\neq 0^{\infty}l_{n}\ldots l_1$. If there are $n\in \N$ fully accessible arc-components which do not contain folding points, then there are $n$ simple dense canals in this $\mathcal{E}$-embedding of $X'$.
\end{corollary}
\fi

Therefore, by Proposition~\ref{prop:restriction}, if we want to find accessible folding points which are not at the top/bottom of any cylinder it is enough to study cases $\nu=10(c_3\ldots c_{n+2})^{\infty}$ where $c_{n+2}=1$.

\begin{remark}\label{rem:types}
	Assume $c$ is preperiodic and $p$ is an accessible folding point of an embedding of $X'$. By Corollary~\ref{cor:class} and since every arc-component contains at most one folding point, only the following three cases can occur:
	\begin{itemize}
		\item[(1)] $\ovl p$ is the top or the bottom of some cylinder; then $\mathcal{U}_p$ is fully accessible.
		\item[(2)] $\ovl p$ is not the top or the bottom of any cylinder, but $\ovl{r(p)}$ or $\ovl{l(p)}$ is; then one component of $\mathcal{U}_p\setminus \{p\}$ is fully accessible, and the other component of $\mathcal{U}_p\setminus \{p\}$ is not accessible. See Figure~\ref{fig:foldincyl1}.
		\item[(3)] $\ovl p$, $\ovl{r(p)}$ and $\ovl{l(p)}$ are not extrema of any cylinder; then $c$ is order reversing and $p$ is the only accessible point of $\mathcal{U}_p$. See Figure~\ref{fig:folding}(c).
	\end{itemize} 
\end{remark}

\begin{definition}\label{def:FPtypes}
	We say that an accessible folding point $p$ is accessible of \emph{Type i} if it satisfies the condition $i$ from Remark~\ref{rem:types} for $i\in\{1, 2, 3\}$.
\end{definition}

As it turns out, all types of accessible folding points can occur in $\mathcal{E}$-embeddings. In the following subsections we describe how they can be constructed in preperiodic orbit case (when $T^3(c)$ is periodic) and give examples of such constructions. We will see that the standard Brucks-Diamond embedding does not allow Type 3 folding points for any $X'$ (see Section~\ref{sec:BD}).

\subsubsection{Type 2}
 First we give examples of $X'$ which cannot be $\mathcal{E}$-embedded with Type 2 folding points.
Then we show in general how to construct a Type 2 accessible folding point and give an example of such construction in both the order preserving and the order reversing case.

\begin{lemma}\label{lem:nottype2}
	Let $\nu=10(c_3\ldots c_{n+2})^{\infty}$ and assume that $c^*_ic_{i+1}\ldots c_{n+2}(c_3\ldots c_{n+2})^{M}$ is admissible for all $i\in\{3, \ldots, n+1\}$ and for all but finitely many $M\in\N$.
	Then no folding point is Type 2 in any $\mathcal{E}$-embedding of $X'$ which is non-equivalent to the Brucks-Diamond ($L=0^{\infty}1$) embedding.
\end{lemma} 
\begin{proof}
	Take a folding point $p\in X'$ with symbolic description $\bar{p}=(c_3\ldots c_{n+2})^{\infty}.(c_3\ldots\\ c_{n+2})^{\infty}$. We will try to reconstruct $L$ which embeds $p$ as Type 2 and see that this is not possible. 
	
	Assume first that $\#_1(c_3\ldots c_{n+2})$ is odd and for some natural number $M$ we have (the following, possibly with reversed inequalities, needs to be satisfied in order for $p$ to be a Type 2 folding point, see Figure~\ref{fig:type2rev}):
	$$\ldots 0(c_3\ldots c_{n+2})^{M}\succ_L \ldots 1(c_3\ldots c_{n+2})^{M}$$
	$$\ldots c^*_ic_{i+1}\ldots c_{n+2}(c_3\ldots c_{n+2})^{M}\prec_L \ldots c_ic_{i+1}\ldots c_{n+2}(c_3\ldots c_{n+2})^{M}$$
	$$\ldots 0(c_3\ldots c_{n+2})^{M+k}\prec_L \ldots 1(c_3\ldots c_{n+2})^{M+k}$$
	$$\ldots c^*_ic_{i+1}\ldots c_{n+2}(c_3\ldots c_{n+2})^{M+k}\prec_L \ldots c_ic_{i+1}\ldots c_{n+2}(c_3\ldots c_{n+2})^{M+k}$$
	$$\ldots 0(c_3\ldots c_{n+2})^{M+N}\succ_L \ldots 1(c_3\ldots c_{n+2})^{M+N}$$
	$$\ldots c^{*}_{n+1}c_{n+2}(c_3\ldots c_{n+2})^{M+N}\prec_L \ldots c_{n+1}c_{n+2}(c_3\ldots c_{n+2})^{M+N}$$
	for all $i\in\{3, \ldots, n+1\}$ and all $k\in\{1, \ldots, N-1\}$, where natural number $N>1$ is even.\\
	If $\#_1((c_3\ldots c_{n+2})^{M})$ is of the same parity as $\#_1(l_{Mn}\ldots l_1)$, then $l_{Mn+1}=0$, and if $\#_1((c_3\ldots c_{n+2})^{M})$ is of different parity as $\#_1(l_{Mn}\ldots l_1)$, then $l_{Mn+1}=1$. In any case, $\#_1(c_{n+2}(c_3\ldots c_{n+2})^{M})$ is of different parity as $\#_1(l_{Mn+1}\ldots l_1)$ so $l_{Mn+2}=c^*_{n+1}$. So $\#_1(c_{n+1}c_{n+2}(c_3\ldots c_{n+2})^{M})$ is of the same parity as $\#_1(l_{Mn+2}l_{Mn+1}\ldots l_1)$ and thus $l_{Mn+3}=c_n$. Continuing further, we get 
	$$l_{(M+N)n+2}\ldots l_{Mn+2}=c^*_{n+1}c^*_{n+2}(c_3\ldots c_{n+2})^{N-1}c_3\ldots c_nc^*_{n+1}.$$
	Since $L\subset X'$, it follows that $c^*_{n+1}=1$, $\#_1(c_3\ldots c_n)$ is even and the word on the right side of the last equation above is equal to $10(c_3\ldots c_{n+2})^{N-1}c_3\ldots c_nc^*_{n+1}.$ Note that $\#_1(10(c_3\ldots c_{n+2})^{N-1}c_3\ldots c_nc_{n+1})$ is even and thus $10(c_3\ldots c_{n+2})^{N-1}c_3\ldots c_nc^*_{n+1}$ is not admissible by Lemma~\ref{lem:odd}, a contradiction.
	
	Assume that $\#_1(c_3\ldots c_{n+2})$ is even. Note that in this case $N$ is not necessarily even, but now the conclusion $c_{n+1}=0$ implies that $\#_1(c_3\ldots c_n)$ is odd. We continue with arguments as in the paragraphs above. Since $\#_1(c_3\ldots c_{n+2})$ is even the word $\#_1(10(c_3\ldots c_{n+2})^{N-1}c_3\ldots c_n
	c_{n+1})$ is even and thus $10(c_3\ldots c_{n+2})^{N-1}c_3\ldots c_nc^*_{n+1}$ is by Lemma~\ref{lem:odd} again not admissible, a contradiction. 
	
	Note that the proof works analogously for other folding points from the space $X'$. 
\end{proof}

Next we give examples of preperiodic $\nu$ where no folding point can be $\mathcal{E}$-embedded as Type 2, except possibly using the Brucks-Diamond embedding, see Section~\ref{sec:BD}, specially the rational endpoint case.

\begin{example}\label{ex:10}
The assumptions from Lemma~\ref{lem:nottype2} hold for \eg $\nu=10(0^{\alpha}1^{\beta})$ for all $\alpha, \beta\in\N$.
\end{example}

%The proof of the following lemma follows directly from the statement, see Figure~\ref{fig:foldincyl1}.

\begin{lemma}[Order preserving case]\label{lem:T2op}
	Let $\nu=10(c_3\ldots
	c_{n+2})^{\infty}$, $c_{n+2}=1$, and $\#_1(c_3\ldots \\
	c_{n+2})$ even. Let $\bar{p}=(c_3\ldots c_{n+2})^{\infty}c_3\ldots c_i.c_{i+1}\ldots c_{n+2}(c_3\ldots c_{n+2})^{\infty}$ be a symbolic description of a folding point $p\in X'$. Then $p$ is a Type 2 folding point if and only if there exists a natural number $M$ such that
	$$\ldots c_j^*c_{j+1}\ldots c_{n+2}(c_3\ldots c_{n+2})^{M+N}c_3\ldots c_i\succ_L \ldots c_jc_{j+1}\ldots c_{n+2}(c_3\ldots c_{n+2})^{M+N}c_3\ldots c_i,$$
	for all $N\in\N$ and all $j\in\{3, \ldots, 1+n\}$ for which $c_j^*c_{j+1}\ldots c_{n+2}(c_3\ldots c_{n+2})^{M+N}c_3\ldots c_i$ is admissible, and
	$$\ldots 0(c_3\ldots c_{n+2})^{M+N'}c_3\ldots c_i\prec_L \ldots1(c_3\ldots c_{n+2})^{M+N'}c_3\ldots c_i,$$
	for infinitely many $N'\in\N$, or \black{the whole statement} with reversed inequalities.
\end{lemma}
\black{
\begin{proof}
Assume that $\#_1(c_3\ldots c_i)$ is odd (even). Note that $$\ldots c_j^*c_{j+1}\ldots c_{n+2}(c_3\ldots c_{n+2})^{M+N}c_3\ldots c_i$$ are exactly itineraries of all ``long" basic arcs in a sufficiently small cylinder around $A(\ovl p)$ and $$\ldots 0(c_3\ldots c_{n+2})^{M+N'}c_3\ldots c_i$$ are itineraries of all ``left" (``right") basic arcs in the same cylinder, see Figure~\ref{fig:foldincyl1}. Here we use terms ``long" and ``left" (``right") as in the proof of Lemma~\ref{lem:nottype2}, \ie if we denote $\pi_0(A(\ovl p))=[T^l(c), T^r(c)]$, then long basic arcs are the ones which project to $[T^l(c), T^r(c)]$ and left (right) basic arcs are the ones which project to $[T^l(c), \pi_0(p)]$ ($[\pi_0(p), T^r(c)]$). Note that there exist a neighbourhood of $A(\ovl p)$ which contains only long and left (right) basic arcs. The condition in the statement forces that all long basic arcs are above $A(\ovl p)$ and infinitely many left (right) are below $A(\ovl p)$. The statement also holds true if all long basic arcs were forced to be under $A(\ovl p)$ and infinitely many left (right) are above $A(\ovl p)$ (which refers to ``the whole statement with reversed inequalities'' in the statement of this lemma).
\end{proof}
}
We give an example that satisfies the assumptions of Lemma~\ref{lem:T2op}.

\begin{example}[Type 2, order preserving case]\label{ex:11}
	Take $\nu=10(01101001)^{\infty}$, $L=(10100101\\
	11001001)^{\infty}$ and $$\bar{p}=(01101001)^{\infty}01.101001(01101001)^{\infty}.$$
Then $\ovl{r(p)}$ is the smallest left-infinite tail so it is the smallest in the cylinder $[0]$. As the calculations below show, all long basic arcs in small neighbourhood of $A(\ovl{p})$ are below $A(\ovl{p})$ and left arcs are both above and below $A(\ovl{p})$, depending on the parity of period which corresponds with $\ovl{p}$ in the left infinite description of basic arcs, see Figure~\ref{fig:type2}.
$$\ldots 0(01101001)^{2N}01\succ_L \ovl p,$$
$$\hspace{-0.4cm}\ldots 0(01101001)^{2N+1}01\prec_L \ovl p,$$
$$\hspace{-0.1cm}\ldots 11(01101001)^{N}01\prec_L\ovl p,$$
$$\hspace{-0.3cm}\ldots 101(01101001)^{N}01\prec_L\ovl p,$$
$$\hspace{-0.7cm}\ldots 11001(01101001)^{N}01\prec_L\ovl p,$$
$$\hspace{-0.9cm}\ldots 001001(01101001)^{N}01\prec_L\ovl p,$$
$$\hspace{-1.1cm}\ldots 0101001(01101001)^{N}01\prec_L\ovl p,$$
$$\hspace{-1.3cm}\ldots 11101001(01101001)^{N}01\prec_L\ovl p,$$
for all $N\in\N$. Further calculations show that only tails of $L$ and $S$ can appear as the extrema of cylinders. By Proposition~\ref{prop:fully},  the arc-component $\mathcal{U}_L$ is fully accessible and since $\mathcal{U}_L$ contains no folding points, it corresponds to an open interval on the circle of prime ends. The accessible part of  $\mathcal{U}_S$ corresponds to a half-open interval on the circle of prime ends, where the endpoint of the half-open interval corresponds to the accessible folding point $p$. By further calculations we obtain that other folding points are not accessible, so the remaining point on the circle of prime ends corresponds to a simple dense canal with shores being $\mathcal{U}_L$ and $\mathcal{U}_S$.
\end{example}

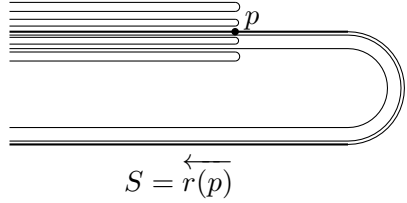
\begin{figure}[!ht]
	\centering
	\begin{tikzpicture}[scale=1.5]
	
	\draw (0,1.26)--(2,1.26);
	\draw (0,1.18)--(2,1.18);
	\draw (0,1.11)--(2,1.11);
	\draw (0,1.05)--(2,1.05);
	\draw[domain=270:450] plot ({2+0.04*cos(\x)}, {1.22+0.04*sin(\x)});
	\draw[domain=270:450] plot ({2+0.03*cos(\x)}, {1.08+0.03*sin(\x)});
	
	\draw (0,0.89)--(2,0.89);
	\draw (0,0.95)--(2,0.95);
	\draw[domain=270:450] plot ({2+0.03*cos(\x)}, {0.92+0.03*sin(\x)});
	\draw (0,0.74)--(2,0.74);
	\draw (0,0.82)--(2,0.82);
	\draw[domain=270:450] plot ({2+0.04*cos(\x)}, {0.78+0.04*sin(\x)});
	
	\draw[thick] (0, 1)--(3,1);
	\draw[thick] (0, 0)--(3,0);
	\draw[domain=270:450] plot ({3+0.5*cos(\x)}, {0.5+0.5*sin(\x)});
	\node[circle,fill, inner sep=1] at (2,1){};
	\node at (2.15,1.1){\small $p$};
	\node at (1.5,-0.3){\small $S=\ovl{r(p)}$};
	
	%\draw[thin] (0,0.09)--(3,0.09);
	\draw[thin] (0,0.15)--(3,0.15);
	\draw[thin] (0,0.03)--(3,0.03);
	\draw[thin] (0,0.97)--(3,0.97);
	\draw[thin] (0,0.85)--(3,0.85);
	%\draw[thin] (0,0.92)--(3,0.91);
	\draw[domain=270:450] plot ({3+0.47*cos(\x)}, {0.5+0.47*sin(\x)});
	\draw[domain=270:450] plot ({3+0.35*cos(\x)}, {0.5+0.35*sin(\x)});
	%\draw[domain=270:450] plot ({3+0.41*cos(\x)}, {0.5+0.41*sin(\x)});
	\end{tikzpicture}
	\caption{Type 2 folding point from Example~\ref{ex:11}.}
	\label{fig:type2}
\end{figure}

%The proof of the following lemma follows directly from its assumptions (see Figure~\ref{fig:type2rev}).

\begin{lemma}[Order reversing case]\label{lem:T2or}
	Let $\nu=10(c_3\ldots
	c_{n+2})^{\infty}$, $c_{n+2}=1$, and $\#_1(c_3\ldots \\c_{n+2})$ odd. Let $\bar{p}=(c_3\ldots c_{n+2})^{\infty}c_3\ldots c_i.c_{i+1}\ldots c_{n+2}(c_3\ldots c_{n+2})^{\infty}$ be a symbolic description of a folding point $p\in X'$. Then $p$ is a Type 2 folding point if and only if there exists a natural number $M$ such that
	$$\ldots c_j^*c_{j+1}\ldots c_{n+2}(c_3\ldots c_{n+2})^{M+N}c_3\ldots c_i\prec_L \ldots c_jc_{j+1}\ldots c_{n+2}(c_3\ldots c_{n+2})^{M+N}c_3\ldots c_i,$$
	for all $N\in\N$ and all $j\in\{3, \ldots, 1+n\}$ for which $c_j^*c_{j+1}\ldots c_{n+2}(c_3\ldots c_{n+2})^{M+N}c_3\ldots c_i$ is admissible, and
	$$\ldots 0(c_3\ldots c_{n+2})^{M+2N'}c_3\ldots c_i\prec_L \ldots 1(c_3\ldots c_{n+2})^{M+N'}c_3\ldots c_i,$$ and
	$$\ldots 0(c_3\ldots c_{n+2})^{M+2N''+1}c_3\ldots c_i\succ_L \ldots 1(c_3\ldots c_{n+2})^{M+N''}c_3\ldots c_i,$$
	for infinitely many $N'\in\N$ and all but finitely many $N''\in\N$\black{, or either only last two equalities are reversed, all the inequalities are reversed or only first one is reversed.}
\end{lemma}
\black{
	\begin{proof}
		Assume that $\#_1(c_3\ldots c_i)$ is odd and $M$ is even. As in the proof of Proposition~\ref{prop:restriction} we just need to note that: 
		$$\ldots c_j^*c_{j+1}\ldots c_{n+2}(c_3\ldots c_{n+2})^{M+N}c_3\ldots c_i$$
		are itineraries of all long basic arcs,
		$$\ldots 0(c_3\ldots c_{n+2})^{M+2N'}c_3\ldots c_i$$
		are itineraries of all left basic arcs, and
		$$\ldots 0(c_3\ldots c_{n+2})^{M+2N''+1}c_3\ldots c_i$$
		are itineraries of all right basic arcs (recall the notation long/left/right from the proof of Proposition~\ref{prop:restriction}), and those are all basic arcs in a sufficiently small neighbourhood of $A(\ovl p)$. The conditions from the statement thus force all but finitely many long and left basic arcs below $A(\ovl p)$ and infinitely many right basic arcs above $A(\ovl p)$. See Figure~\ref{fig:type2rev}. Note that different parities of $\#_1(c_3\ldots c_i)$ and $M$ change which itineraries are right/left basic arcs and that the statement holds true if the equalities are reversed as described at the end of the statement of the lemma.
	\end{proof}
}
We give an example that satisfies the assumptions of Lemma~\ref{lem:T2or}.

\begin{example}[Type 2, order reversing case]\label{ex:12}
	 Take $\nu=10(011101001)^{\infty}$, $L=(011101\\
	 001011110010)^{\infty}$ and $\ovl p=(011101001)^{\infty}$. What follows is an easy computation:
$$\ldots 0(011101001)^{2M+1}\prec_L\ovl p,$$
$$\hspace{0.4cm}\ldots 0(011101001)^{2M}\succ_L \ovl p,$$
$$\hspace{0.3cm}\ldots 11(011101001)^M\prec_L \ovl p,$$
$$\hspace{0.05cm}\ldots 101(011101001)^M\prec_L \ovl p,$$
$$\hspace{-0.4cm}\ldots 11001(011101001)^M\prec_L \ovl p,$$
$$\hspace{-0.6cm}\ldots 001001(011101001)^M\prec_L\ovl p,$$
$$\hspace{-0.8cm}\ldots 0101001(011101001)^M\prec_L\ovl p,$$
$$\hspace{-1cm}\ldots 01101001(011101001)^M\prec_L\ovl p,$$
$$\hspace{-1.2cm}\ldots 111101001(011101001)^M\prec_L\ovl p,$$
for every $M\in\N$.
So $p$ is accessible folding point of Type 2. Note that $\ovl {l(p)}=(010010111)^{\infty}01011=S_{1011}$, see Figure~\ref{fig:type2rev}. By further symbolic calculations we again conclude that there is one simple dense canal for this embedding of $X'$.
\end{example}

\begin{figure}[!ht]
	\centering
	\begin{tikzpicture}[scale=1.5]
	\begin{scope}[yscale=1,xscale=-1]
	
	\draw (0,1.26)--(2,1.26);
	\draw (0,1.18)--(2,1.18);
	\draw (0,1.11)--(2,1.11);
	\draw (0,1.05)--(2,1.05);
	\draw[domain=270:450] plot ({2+0.04*cos(\x)}, {1.22+0.04*sin(\x)});
	\draw[domain=270:450] plot ({2+0.03*cos(\x)}, {1.08+0.03*sin(\x)});
	
	\draw (2,0.89)--(3,0.89);
	\draw (2,0.95)--(3,0.95);
	\draw[domain=90:270] plot ({2+0.03*cos(\x)}, {0.92+0.03*sin(\x)});
	\draw (2,0.74)--(3,0.74);
	\draw (2,0.82)--(3,0.82);
	\draw[domain=90:270] plot ({2+0.04*cos(\x)}, {0.78+0.04*sin(\x)});
	
	\draw[thick] (0, 1)--(3,1);
	\draw[thick] (0, 0)--(3,0);
	\draw[domain=270:450] plot ({3+0.5*cos(\x)}, {0.5+0.5*sin(\x)});
	\node[circle,fill, inner sep=1] at (2,1){};
	\node at (2.15,1.1){\small $p$};
	\node at (1.5,-0.3){\small $\ovl{l(p)}=S_{1011}$};
	
	\draw[thin] (0,0.26)--(3,0.26);
	\draw[thin] (0,0.15)--(3,0.15);
	\draw[thin] (0,0.03)--(3,0.03);
	\draw[thin] (0,0.97)--(3,0.97);
	\draw[thin] (0,0.85)--(3,0.85);
	\draw[thin] (0,0.18)--(3,0.18);
	\draw[thin] (0,0.11)--(3,0.11);
	\draw[thin] (0,0.05)--(3,0.05);
	\draw[domain=270:450] plot ({3+0.47*cos(\x)}, {0.5+0.47*sin(\x)});
	\draw[domain=270:450] plot ({3+0.35*cos(\x)}, {0.5+0.35*sin(\x)});
	\draw[domain=270:450] plot ({3+0.32*cos(\x)}, {0.5+0.32*sin(\x)});
	\draw[domain=270:450] plot ({3+0.24*cos(\x)}, {0.5+0.24*sin(\x)});
	\draw[domain=270:450] plot ({3+0.39*cos(\x)}, {0.5+0.39*sin(\x)});
	\draw[domain=270:450] plot ({3+0.45*cos(\x)}, {0.5+0.45*sin(\x)});
	\end{scope}
	\end{tikzpicture}
	\caption{Type 2 folding point from Example~\ref{ex:12}.}
	\label{fig:type2rev}
\end{figure}
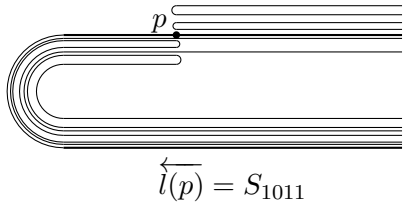

\subsubsection{Type 3}

From now onwards we study folding points of Type 3, see Figure~\ref{fig:preper}.

\begin{remark}\label{rem:T3noteven}
	Let $\nu=10(c_3\ldots c_{n+2})^{\infty}$ be such that $\#_1(c_3\ldots c_{n+2})$ is even and $c_{n+2}=1$. Then $X'$ does not contain folding points of Type 3. 
\end{remark}

The following lemma gives necessary and sufficient symbolic conditions for a folding point to be $\mathcal{E}$-embedded as Type 3.

\begin{lemma}[Type 3]\label{lem:type3sy}
	Let $\nu=10(c_3\ldots c_{n+2})^{\infty}$, $c_{n+2}=1$, and $\#_1(c_3\ldots c_{n+2})$ odd. Let $\bar{p}=(c_3\ldots c_{n+2})^{\infty}c_3\ldots c_i.c_{i+1}\ldots c_{n+2}(c_3\ldots c_{n+2})^{\infty}$ be the symbolic description of a  folding point $p\in X'$. Then $p$ is a Type 3 folding point if and only if there exists $M>0$ such that
	$$\ldots c_j^*c_{j+1}\ldots c_{n+2}(c_3\ldots c_{n+2})^{M+N}c_3\ldots c_i\prec_L \ovl p,$$
	for all $N\in\N$ and all $j\in\{3, \ldots, n+1\}$ for which $c_j^*c_{j+1}\ldots c_{n+2}(c_3\ldots c_{n+2})^{M+N}c_3\ldots c_i$ is admissible, and
	$$\ldots 0(c_3\ldots c_{n+2})^{M+N'}c_3\ldots c_i\succ_L \ovl p,$$
	for infinitely many \black{of both even and odd} $N'\in\N$, or with reversed inequalities. See Figure~\ref{fig:preper}.
\end{lemma}
\black{\begin{proof}
		Note that $$\ldots c_j^*c_{j+1}\ldots c_{n+2}(c_3\ldots c_{n+2})^{M+N}c_3\ldots c_i$$ are long basic arcs and $$\ldots 0(c_3\ldots c_{n+2})^{M+N'}c_3\ldots c_i$$ are right or left, depending on the parity of $M$ and $N'$. In any case, the conditions force all long basic arcs below $A(\ovl p)$ and infinitely many right and left basic arcs above $A(\ovl p)$. 
	\end{proof}
}

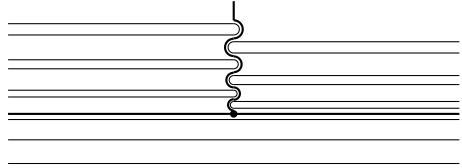
\begin{figure}[!ht]
	\centering
	\begin{tikzpicture}[scale=1.5]
	\draw (5,1.8)--(7,1.8);
	\draw (5,1.7)--(7,1.7);
	\draw (5,1.48)--(7,1.48);
	\draw (5,1.4)--(7,1.4);
	\draw (5,1.21)--(7,1.21);
	\draw (5,1.15)--(7,1.15);
	\draw (7,1.64)--(9,1.64);
	\draw (7,1.54)--(9,1.54);
	\draw (7,1.34)--(9,1.34);
	\draw (7,1.26)--(9,1.26);
	\draw (7,1.11)--(9,1.11);
	\draw (7,1.05)--(9,1.05);
	\draw[domain=270:450] plot ({7+0.05*cos(\x)}, {1.75+0.05*sin(\x)});
	\draw[thick][domain=270:450] plot ({7+0.08*cos(\x)}, {1.75+0.08*sin(\x)});
	\draw[domain=270:450] plot ({7+0.04*cos(\x)}, {1.44+0.04*sin(\x)});
	\draw[thick][domain=270:450] plot ({7+0.07*cos(\x)}, {1.44+0.07*sin(\x)});
	\draw[domain=270:450] plot ({7+0.03*cos(\x)}, {1.18+0.03*sin(\x)});
	\draw[thick][domain=270:450] plot ({7+0.055*cos(\x)}, {1.18+0.055*sin(\x)});
	\draw[domain=90:270] plot ({7+0.05*cos(\x)}, {1.59+0.05*sin(\x)});
	\draw[thick][domain=90:270] plot ({7+0.077*cos(\x)}, {1.59+0.077*sin(\x)});
	\draw[domain=90:270] plot ({7+0.04*cos(\x)}, {1.30+0.04*sin(\x)});
	\draw[thick][domain=90:270] plot ({7+0.065*cos(\x)}, {1.30+0.065*sin(\x)});
	\draw[domain=90:270] plot ({7+0.03*cos(\x)}, {1.08+0.03*sin(\x)});
	\draw[thick][domain=90:270] plot ({7+0.05*cos(\x)}, {1.08+0.05*sin(\x)});
	\draw[thick] (5, 1)--(9,1);
	\node[circle,fill, inner sep=1] at (7,1){};
	\draw (5,0.95)--(9,0.95);
	\draw (5,0.77)--(9,0.77);
	\draw (5,0.56)--(9,0.56);
	\put(295,35){\small $p$};
	\draw[thick] (7,1.83)--(7,2);
	\put(300,80){\small $R$};
	%\put(302,65){\small $R$}
	\end{tikzpicture}
	\caption{Type 3 folding point. Folding point $p$ is accessible from the complement by an arc $R\cup\{p\}\subset \R^{2}$, where $R$ is a ray.}
	\label{fig:preper}
\end{figure}

The following lemma gives conditions on preperiodic, order reversing $\nu$ such that no folding point can be $\mathcal{E}$-embedded  as Type 3 folding point (except possibly with the Brucks-Diamond embedding studied in detail in Section~\ref{sec:BD}).

\begin{lemma}\label{lem:notType3}
	Let $\nu=10(c_3\ldots c_{n+2})^{\infty}$ be such that $\#_1(c_3\ldots c_{n+2})$ is odd, $c_{n+2}=1$ and let $c_j^*c_{j+1}\ldots c_{n+2}(c_3\ldots c_{n+2})^Mc_3
	\ldots c_i$
	be admissible for every $j\in\{3, \ldots 1+n\}$ and all $M\in\N$.
	If $c_{n+1}=1$ then there exists no $L$ such that folding point $p\in X'$ is of Type 3.
\end{lemma}

\begin{proof}
	Take a folding point $p\in X'$ with the symbolic description
	$$\bar{p}=(c_{3}\ldots c_{n+2})^{\infty}c_{3}\ldots c_{i}.c_{i+1}\ldots c_{n+2}(c_{3}\ldots c_{n+2})^{\infty}$$ for some $i\in \{3,\ldots n+2\}$
	and assume that $A(\ovl p)$ is not at the top or bottom of any cylinder in $X'$. 
	Since
	$c_j^*c_{j+1}\ldots c_{n+2}(c_3\ldots c_{n+2})^Mc_3
	\ldots c_i$
	is admissible for every $j\in\{3, \ldots n+1\}$ and all $M\in\N$, the same calculations as in the proof of Lemma~\ref{lem:nottype2} imply that the only $L$ which satisfies all the conditions from Lemma~\ref{lem:type3sy} is
	$$L=(c_3\ldots c_n00)^{\infty}l_k\ldots l_1,$$
	for some $l_k\ldots l_1$.  However, the word $00c_3\ldots c_n$ is not admissible, a contradiction.
\end{proof}

\begin{example}[No Type 3 folding point]\label{ex:13}
	Note that $\nu=10(0^{\alpha}1^{\beta})^{\infty}$ for $\beta\geq 2$ satisfies the assumptions of Lemma~\ref{lem:notType3}. Thus no folding point from the corresponding $X'$ can be embedded as Type 3 folding point using $\mathcal{E}$-embeddings (except maybe Brucks-Diamond). Note that this example also satisfies the assumptions of Lemma~\ref{lem:nottype2}, so no folding point can be $\mathcal{E}$-embedded as Type 2 either. Thus in these cases a point from $X'$ is accessible if and only if it is on the top or the bottom of some cylinder. So there are $m\in\N$ simple dense canals in $\mathcal{E}$-embeddings of such $X'$, where $m$ is the number of fully accessible arc-components.
\end{example}

The following lemma gives sufficient symbolic conditions on a preperiodic $\nu$ such that every folding point can be $\mathcal{E}$-embedded as accessible folding point of Type 3.

\begin{lemma}\label{lem:type3}
	Let $\nu=10(c_3\ldots c_{n+2})^{\infty}$ be such that $\#_1(c_3\ldots c_{n+2})$ is odd and  $c_{n+2}=1$.
	Assume that $c_{n+1}=0$ and the tail $(10c_3\ldots c_n)^{\infty}$ is admissible. For every folding point $p\in X'$ there exists $L$ such that $p$ is of Type 3 in $\phi_L(X')$. 
\end{lemma}
\begin{proof}
	Take a folding point $p\in X'$ with the symbolic description
	$$\bar{p}=(c_{3}\ldots c_{n+2})^{\infty}c_{3}\ldots c_{i}.c_{i+1}\ldots c_{n+2}(c_{3}\ldots c_{n+2})^{\infty}$$ for some $i\in \{3,\ldots n+2\}$. Denote by $\pi_0(A(\ovl p))=:[T^l(c), T^r(c)]$ for some $l,r\in\N$.
	
	Let $L=(c_3\ldots c_nc^*_{n+1}c^*_{n+2})^{\infty}c_3\ldots c_{i}$.  Then
	$$\ldots 0(c_3\ldots c_{n+2})^m c_3\ldots c_{i}\succ_L \ldots 1(c_3\ldots c_{n+2})^m c_3\ldots c_{i},$$
	$$\ldots c^*_jc_{j+1}\ldots c_{n+2}(c_3\ldots c_{n+2})^m c_3\ldots c_{i}\prec_L \ldots c_jc_{j+1}\ldots c_{n+2}(c_3\ldots c_{n+2})^m c_3\ldots c_{i},$$
	for every $m\in\N$, every $j\in\{3, \ldots n+1\}$
	and all admissible $c^*_jc_{j+1}\ldots c_{n+2}(c_3\ldots c_{n+2})^m c_3\\
	\ldots c_{i}$, see Figure~\ref{fig:preper} to visualize the construction. By the assumptions we conclude that $L=(10c_3\ldots c_{n})^{\infty}10c_3\ldots c_i$ is indeed admissible. Since $\#_1(c_3\ldots c_{n+2})$ is odd we get pairs of basic arcs joined at a point which projects with $\pi_{0}$ to $\pi_{0}(p)$, approaching to $A(\ovl p)$ from above from both left and right side of $p$, exactly as in Figure~\ref{fig:preper}.
\end{proof}

\begin{example}[Type 3 folding point]\label{ex:14} Take $\nu=10(01101)^{\infty}$. If we embed $X'$ with respect to admissible $L=(01110)^{\infty},$ then $\bar{p}=(01101)^{\infty}.(01101)^{\infty}$ is an accessible folding point of Type 3, since it satisfies the conditions of Lemma~\ref{lem:type3}. Note that only $\mathcal{U}_L$ can contain the extremum of a cylinder and it corresponds to the circle of prime ends minus a point. The remaining point is the second kind prime end corresponding to the accessible folding point $p$ of Type 3. Specifically, there are no simple dense canals.
\end{example}

\section{Extendability of the shift homeomorphism for $\mathcal{E}$-embeddings}\label{sec:ext}

Planar embeddings of $X$ equivalent to \black{those constructed in} \cite{Br1} ($L=1^{\infty}$) and \cite{BrDi} ($L=0^{\infty}1$) make $\mathcal{R}$, $\mathcal{C}$ and $\mathcal{C}$ respectively fully accessible as can be deduced from Proposition~\ref{prop:fully} and Remark~\ref{rem:Cisolated}. Additionally it can be deduced from Proposition~\ref{prop:Cacc} that only remaining accessible points of embeddings of $X$ (if existent) need to be folding points. \black{We} denote the two special embeddings from now onwards by $\phi_{\mathcal{R}}$ and $\phi_{\mathcal{C}}$ respectively (recall that $\phi_L$ denotes the planar embedding determined by the left infinite sequence $L$) and refer to them as \emph{standard embeddings}, for the reasons below. 

\black{Barge and Martin show in \cite{BM} that every $X$ (actually every interval inverse limit with a single bonding map) can be embedded in the plane as an attractor of a planar homeomorphism which is conjugate to $\sigma$ on $X$. For unimodal inverse limits $X$, there are two ways to preform that construction, in an orientation-preserving or orientation-reversing way. Those embeddings are equivalent to $\phi_{\mathcal{C}}$ and $\phi_{\mathcal{R}}$, respectively.} Specially,  for $\phi_{\mathcal{C}}(X)$ and $\phi_{\mathcal{R}}(X)$, the homeomorphism $\sigma$ is extendable to $\R^2$ \black{(by extendable to $\R^2$ we always mean extendable to a planar homeomorphism)}.  Bruin directly showed in \cite{Br1} that the shift homeomorphism can be extended to the plane for embeddings $\phi_{\mathcal{R}}$. Now we show that except for the two mentioned standard embeddings, $\sigma$ is not extendable for any $\mathcal{E}$-embedding of $X'$.

Note that if $\sigma:\phi_L(X)\to \phi_L(X)$ is extendable to $\R^2$, then $\sigma|_{\phi_L(X')}: \phi_L(X')\to \phi_L(X')$ is also extendable to $\R^2$.

 \black{Let us again note that when it is clear from the context that we refer to the basic arc $A(\ovl{s})$ we often abbreviate notation and write only $\ovl{s}$.}

The following theorem \black{partially} answers the question whether for non-standard $\mathcal{E}$-embeddings the shift homeomorphism is extendable to the whole plane which was posed by Boyland, de Carvalho and Hall in \cite{3G}.

\begin{theorem}\label{thm:shuffle}
	If $X'$ is embedded in the plane with respect to $L$, where $A(L)\not\subset \mathcal{C}, \mathcal{R}$, then the shift homeomorphism $\sigma\colon \phi_L(X')\to \phi_L(X')$ cannot be extended to a homeomorphism of the plane.
\end{theorem}

\begin{proof}
	Let $\nu=c_1c_2\ldots$ be a kneading sequence and $A(L)\not\subset \mathcal{C}, \mathcal{R}$ and assume by contradiction that $\sigma\colon \phi_{L}(X')\to \phi_L(X')$ is extendable to $\mathbb{R}^2$. Let $(n_i)_{i\in\N}$ be an increasing sequence in $\N$ such that $l_{n_i+3}l_{n_i+2}=01$. Since $A(L)\not\subset\mathcal{C}, \mathcal{R}$, the sequence $(n_i)_{i\in\N}$ is indeed well defined.
	For $i\in\N$ define admissible tails
	$$\ovl {x_i}=1^{\infty}1011^{n_i},$$
	$$\ovl {y_i}=1^{\infty}0111^{n_i},$$
	$$\ovl {w_i}=1^{\infty}1101^{n_i}.$$
	Note that $\ovl{x_i}$ is between $\ovl{y_i}$ and $\ovl{w_i}$ \black{in the ordering $\preceq_L$} and $\ovl{x_i}1$ is the largest or the smallest \black{(again in $\preceq_L$)} among the admissible sequences $\ovl{x_i}1$, $\ovl{y_i}1$ and $\ovl{w_i}1$ because of the chosen $l_{n_i+3}l_{n_i+2}=01$.\\
	For $i$ large enough, note that $\pi_0(\ovl{x_i}1)=\f T^2(c), T(c)]$ so $A(\ovl{x_i}1)$ is a horizontal arc in the plane of length $|T(c)-T^2(c)|=:\delta>0$. Note also that $\pi_0(\ovl{x_i})=\pi_0(\ovl{y_i})=\pi_0(\ovl{w_i})=\f T^2(c),T(c)]$ for $i$ large enough. Let $\ovl{x_i}'=\pi^{-1}_0(\f c, T(c)])\cap\ovl{x_i}$, $\ovl{y_i}'=\pi^{-1}_0(\f c, T(c)])\cap\ovl{y_i}$ and $\ovl{w_i}'=\pi^{-1}_0(\f c, T(c)])\cap\ovl{w_i}$, see Figure~\ref{fig:shuffle}, left picture. Denote by $A_i\subset\R^2$ ($B_i\subset\R^2$) the vertical segment which joins the left (right) endpoints of $\ovl{y_i}'$ and $\ovl{w_i}'$. Note that $\diam(A_i), \diam(B_i)\to 0$ as $i\to\infty$. Also $D=A_i\cup\ovl{y_i}'\cup B_i\cup\ovl{w_i}'$ separates the plane, denote the bounded component of $\R^2\setminus D$ by $U\subset\R^2$. Note that $\Int\ovl{x_i}'\subset U$.\\
	Now note that $\sigma(\ovl{x_i}')=\ovl{x_i}1$ and similarly for $\ovl{y_i}', \ovl{w_i}'$. Since $\ovl{x_i}1$ is the smallest or the largest among $\ovl{x_i}1, \ovl{y_i}1, \ovl{w_i}1$ and $\sigma$ is extendable, at least one $\sigma({A_i})$ or $\sigma({B_i})$ has length greater than $\delta$, see Figure~\ref{fig:shuffle}. This contradicts the continuity of $\sigma$.
\end{proof}

\begin{figure}[!ht]
	\centering
	\begin{tikzpicture}[scale=1.5]
	\draw (0,0.5)--(2,0.5);
	\draw (0,1)--(2,1);
	\draw (0,1.5)--(2,1.5);
	\node at (-0.2,0.5) {\small $\ovl {w_i}$};
	\node at (-0.2,1) {\small $\ovl {x_i}$};
	\node at (-0.2,1.5) {\small $\ovl {y_i}$};
	\draw[->] (1,0.4)--(1,0.1);
	\node at (0.8,0.25) {\small $\pi_0$};
	\draw (0,0)--(2,0);
	\draw (0,-0.05)--(0,0.05);
	\node at (0,-0.2) {\scriptsize $T^2(c)$};
	\draw (0.7,-0.05)--(0.7,0.05);
	\node at (0.7,-0.2) {\scriptsize $c$};
	\draw (2,-0.05)--(2,0.05);
	\node at (2,-0.2) {\scriptsize $T(c)$};
	\draw[very thick] (0.7,0.5)--(0.7,1.5);
	\node at (0.52,0.75) {\small $A_i$};
	\draw[very thick] (2,0.5)--(2,1.5);
	\node at (2.2,0.75) {\small $B_i$};
	\path[fill=black,opacity=0.1]
	(0.7,0.5)--(0.7,1.5)--(2,1.5)--(2,0.5)--cycle;
	\node at (1.8,1.3) {\scriptsize $U$};
	\node at (1.3,0.6) {\tiny $\ovl{w_i}'$};
	\node at (1.3,1.1) {\tiny $\ovl{x_i}'$};
	\node at (1.3,1.6) {\tiny $\ovl{y_i}'$};
	
	\draw[->] (2.5,1)--(3.5,1);
	\node at (3,1.1) {\small $\sigma$};
	
	\draw (4,0.5)--(6,0.5);
	\draw (4,1)--(6,1);
	\draw (4,1.5)--(6,1.5);
	\node at (4.5,0.3) {\scriptsize $\ovl{w_i}1=\sigma(\ovl{w_i}')$};
	\node at (4.5,1.1) {\scriptsize $\ovl{y_i}1=\sigma(\ovl{y_i}')$};
	\node at (4.5,1.7) {\scriptsize $\ovl{x_i}1=\sigma(\ovl{x_i}')$};
	\draw[->] (5.7,0.4)--(5.7,0.1);
	\node at (5.5,0.25) {\small $\pi_0$};
	\draw (4,0)--(6.1,0);
	\draw (4,-0.05)--(4,0.05);
	\node at (4,-0.2) {\scriptsize $T^2(c)$};
	\draw (4.7,-0.05)--(4.7,0.05);
	\node at (4.7,-0.2) {\scriptsize $c$};
	\draw (6.1,-0.05)--(6.1,0.05);
	\node at (6.1,-0.2) {\scriptsize $T(c)$};
	\draw[very thick, domain=90:270] plot ({4+0.25*cos(\x)}, {0.75+0.25*sin(\x)});
	\node at (4.1,0.75) {\scriptsize $\sigma(A_i)$};
	\draw[very thick, domain=270:450] plot ({6+0.05*cos(\x)}, {1.05+0.05*sin(\x)});
	\draw[very thick] (6,1.1)--(4,1.4);
	\draw[very thick, domain=90:270] plot ({4+0.1*cos(\x)}, {1.5+0.1*sin(\x)});
	\draw[very thick] (4,1.6)--(6.1,1.6)--(6.1,0.5)--(6,0.5);
	\node at (5.5,1.3) {\scriptsize $\sigma(B_i)$};
	\path[fill=black,opacity=0.1]
	(4,0.5)-- plot[domain=90:270] ({4+0.25*cos(\x)}, {0.75+0.25*sin(\x)}) -- (4,1)--(6,1)-- plot[domain=270:450] ({6+0.05*cos(\x)}, {1.05+0.05*sin(\x)})--(6,1.1)--(4,1.4)--plot[domain=90:270] ({4+0.1*cos(\x)}, {1.5+0.1*sin(\x)})--(4,1.6)--(6.1,1.6)--(6.1,0.5)--(6,0.5)--cycle;
	\node at (5.8,0.65) {\scriptsize $\sigma(U)$};
	\end{tikzpicture}
	\caption{Shuffling of basic arcs from the proof of Theorem~\ref{thm:shuffle}.}
	\label{fig:shuffle}
\end{figure}
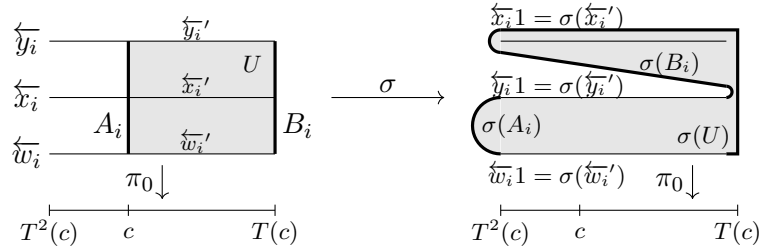

\section{$\mathcal{E}$-embeddings of $X'$ with more than one fully accessible arc-components}\label{sec:fullyacc}

In this section we study $\mathcal{E}$-embeddings of an arbitrary $X'$ that allow at least two fully accessible dense arc-components. 
\iffalse
As a consequence we observe that when $c$ has finite orbit there are two simple dense canals in these $\mathcal{E}$-embeddings of $X'$.
\fi

\begin{lemma}\label{lem:horror}
	Let $\nu=10^{\kappa}1\ldots$ and embed $X'$ with respect to $L=(0^{\kappa}1)^{\infty}$. The smallest left-infinite tail with respect to $\prec_L$ is $A(S)=A(S_0)=A((10^{\kappa})^{\infty})\not\subset\mathcal{U}_L.$ Moreover, both $\mathcal{U}_L$ and $\mathcal{U}_S$ are fully accessible and dense in $X'$.
\end{lemma}
\begin{proof}
	First, let us comment that $L=(0^{\kappa}1)^{\infty}$ is admissible. Note that there exists \black{a natural number} $0\leq \kappa_2<\kappa$ such that $\nu=10^{\kappa}10^{\kappa_2}1\ldots$, so the word $10^{\kappa}10^{\kappa}$ is indeed admissible.
	
	It is straightforward to calculate $S$, infinitely many changes occur because $0^{\kappa+1}$ is not admissible, \ie symbol $0$ alters $L$, see Definition~\ref{def:dependence}. 
	
	To prove that $\mathcal{U}_L$ and $\mathcal{U}_S$ are fully accessible, it is enough to show that every basic arc from $\mathcal{U}_L\cup\mathcal{U}_S$ is at the top or the bottom of some cylinder.\\
	Proposition~\ref{prop:fully} shows that $\mathcal{U}_L$ is fully accessible.
	Assume that $A(\overleftarrow{x})\subset\mathcal{U}_S$ and take $k\in\N$ such that $x_{k+i}= s_{k+i}$ for every $i\in\N$ and such that $\kappa +1$ divides $k$, where $S=\ldots s_2s_1$. Then $\overleftarrow{x}=\ldots 10^{\kappa}10^{\kappa}x_k\ldots x_1$. Note that if $\#_1(10^{\kappa}x_k\ldots x_1)$ and $\#_1(l_{k+\kappa+1}\ldots l_1)$ have the same parity, then $S_{10^{\kappa}x_k\ldots x_1}=\overleftarrow{x}$ and $L_{10^{\kappa}x_k\ldots x_1}=\overleftarrow{x}$ in the other case.
	
	To show that $\mathcal{U}_L$ and $\mathcal{U}_S$ are dense, fix a point $x\in X'$ with backward itinerary $\ovl x=\ldots x_2x_1$ and fix $n\in\N$.
	 Denote $\nu=10^{\kappa}10^{\kappa_2}10^{\kappa_3}10^{\kappa_4}1\ldots$, where $0\leq\kappa_2<\kappa$, $0\leq\kappa_3, \kappa_4\leq\kappa$.\\
	  If $\kappa_3>0$, then there exists $\gamma \geq 0$ so that  $A((0^{\kappa}1)^{\infty}0^{\kappa_2}1^{\gamma}x_n\ldots x_1)\subset \mathcal{U}_L$ is admissible.\\
	   Assume that $\kappa_3=0$. If $\kappa_4<\kappa$, then there exists $\gamma'\geq 0$ so that $A((0^{\kappa}1)^{\infty}0^{\kappa_2}110^{\kappa_4+1}1^{\gamma'}x_n\\
	   \ldots x_1)\subset \mathcal{U}_L$ is admissible. If $\kappa_4=\kappa$, then $\nu=10^{\kappa}10^{\kappa_2}110^{\kappa}10^{\kappa_{2}}0\ldots$. Therefore, there exist an appropriate $\gamma''\geq 0$ so that  $A((0^{\kappa}1)^{\infty}0^{\kappa_2}110^{\kappa}10^{\kappa_2}11^{\gamma''}x_n\ldots x_1)\subset \mathcal{U}_L$ is admissible. A proof for points from $\mathcal{U}_S$ is analogous. Therefore, there are points from both $\mathcal{U}_L$ and $\mathcal{U}_S$ which are arbitrary close to any $x\in X'$ and thus $\mathcal{U}_L$ and $\mathcal{U}_S$ are dense in $X'$.
\end{proof}

\begin{theorem}
	For every $X'$ there exists a planar embedding with \black{at least} two non-degenerate fully accessible dense arc-components. 
\end{theorem}
\begin{proof}
	Let $\nu=10^{\kappa}1\ldots$ and construct $\phi_L(X')$ with respect to $L=\ldots 0^{\kappa}10^{\kappa}10^{\kappa}1$. Using Lemma~\ref{lem:horror} we conclude that $\mathcal{U}_S$ and $\mathcal{U}_L$ are fully accessible and dense and the claim follows.
\end{proof}

In a special case when the orbit of $c$ is finite and only $\mathcal{U}_L$ and $\mathcal{U}_S$ are fully accessible we obtain the following corollary.

\begin{corollary}
	\black{Assume the orbit of the critical point is finite and $X'$ is embedded as in Lemma~\ref{lem:horror}. Moreover, assume that the set of accessible points consists of only  $\mathcal{U}_L$ and $\mathcal{U}_S$. Then there are exactly two simple dense canals.}
\end{corollary}
\begin{proof}
	Take the embedding constructed in Lemma~\ref{lem:horror}. \black{So, $X'$ with kneading sequence $\nu=10^{\kappa}1\ldots$ is embedded with respect to $L=(0^{\kappa}1)^{\infty}$.} Note that $\mathcal{U}_L$ and $\mathcal{U}_S$ do not contain endpoints for any chosen $\nu=10^{\kappa}1\ldots$ (since the kneading sequence $\nu=(10^{\kappa})^{\infty}$ does not appear as a kneading sequence in the tent map family) and are thus \black{dense} lines.  
	
	If $\nu$ is periodic, the endpoints of $X'$ are not accessible by Corollary~\ref{cor:endpoints_tau}. That in combination with Proposition~\ref{prop:fourth} gives two simple dense canals. If $\nu$ is preperiodic and $T^3(c)$ is not periodic, the conclusion again follows analogously as above. \black{What remains is to argue that when $T^3(c)$ is periodic,} Type 3 folding points do not exist for a chosen $L$. Since $L$ is periodic of period $\kappa+1$, it follows that $\sigma^{\kappa+1}:\phi_L(X')\to \sigma^{\kappa+1}(\phi_L(X'))$ is extendable to the whole plane.\\
	Assume that the point $p\in X'$ is a Type 3 folding point. Thus $\sigma^{\kappa+1}(p)$ is also Type 3 folding point. For $\nu=10(c_3\ldots c_{n+2})^{\infty}$, the itineraries of folding points are periodic of period $n\geq\kappa$. \black{Combining last two facts it follows that} $(\kappa+1)|n$. If $\kappa+1=n$, since $c_{n+2}=1$ it holds  that $c_3\ldots c_{n+2}=0^{\kappa-1}11$, which is even, a contradiction with Remark~\ref{rem:T3noteven}. From the circle of prime ends we get that there can be at most two Type 3 accessible folding points and thus $n=2(\kappa+1)$. Since $\#_1(c_3\ldots c_{n+2})$ is odd, it follows that $\ldots 0P^{2k+1}\succ_L\ldots 1P^{2k+1}$ and $\ldots 0P^{2k}\prec_L\ldots 1P^{2k}$ for all $k\in\N$, where $P=c_3\ldots c_{n+2}$. That is a contradiction with Lemma~\ref{lem:type3sy}, \black{and thus it follows that there are no accessible Type 3 folding points in these embeddings.}
\end{proof}

The following proposition shows that for $L$ as in Lemma~\ref{lem:horror} and $\nu$ of specific form there exist $\mathcal{E}$-embeddings of $X'$ that permit more than two fully accessible arc-components dense in $X'$. Specifically we improve the upper bound on the number of fully accessible non-degenerate arc-components \black{in the non-standard $\mathcal{E}$-embeddings} from three to four (compare to Example~\ref{ex:9}).

\begin{proposition}
Assume $\nu$ is of the form $\nu=10^{\kappa}10^{\kappa-1}110\ldots$ with $\kappa>1$. If $L=(0^{\kappa}1)^{\infty}$, then $\phi_L(X')$ has \black{at least} four fully-accessible dense arc-components.
\end{proposition}

\begin{proof}
 Note that for $L=(0^{\kappa}1)^{\infty}$ and chosen $\nu$ it holds that $S=(10^{\kappa})^{\infty}$ and note that for $\kappa$ even we have $L_{1^{\kappa+1}}=(1110^{\kappa-1}10^{\kappa-1})^{\infty}
 11^{\kappa+1}$ and $S_{ 01^{\kappa+1}}=(010^{\kappa-1}1110^{\kappa-2})^{\infty}01^{\kappa+1}$. For $\kappa$ odd we get $S_{1^{\kappa+1}}=(1110^{\kappa-1}10^{\kappa-1})^{\infty}
 11^{\kappa+1}$ and $L_{ 01^{\kappa+1}}=(010^{\kappa-1}1110^{\kappa-2})^{\infty}01^{\kappa+1}$. Thus we get at least four different accessible left infinite tails. For the rest of the proof we assume without the loss of generality that $\kappa$ is even.\\
 To see that $\mathcal{U}_{L_{1^{\kappa+1}}}$ is fully accessible take $\ovl x=\ldots x_2x_1\subset\mathcal{U}_{L_{1^{\kappa+1}}}$ and $n\in\N$ such that $\ldots x_{n+2}x_{n+1}=(1110^{\kappa-1}10^{\kappa-1})^{\infty}$. Note that then $\ovl x$ is either the largest or the smallest arc in the cylinder $[1110^{\kappa-1}10^{\kappa-1}x_n\ldots x_1]$, depending on the parity of $x_n\ldots x_1$. Similarly we show that $\mathcal{U}_{S_{01^{\kappa+1}}}$ is fully accessible.\\
 To see that $\mathcal{U}_{L_{1^{\kappa+1}}}$ is dense in $X'$, fix a point $x\in X'$ with backward itinerary $\ovl x=\ldots x_2x_1$ and fix $n\in\N$. Note that $0^{\kappa}\not\subset {L_{1^{\kappa+1}}}$ and therefore there exists $\gamma\in\N$ so that $(1110^{\kappa-1}10^{\kappa-1})^{\infty}1^{\gamma}x_n\ldots x_1
 \subset \mathcal{U}_{L_{1^{\kappa+1}}}$ is admissible. We analogously prove that $\mathcal{U}_{S_{01^{\kappa+1}}}$ is dense in $X'$. \black{Combining these facts with Lemma~\ref{lem:horror} we conclude the proof.}
\end{proof}

The characterization of fully accessible arc-components of $\mathcal{E}$-embeddings of $X'$ (excluding the standard embeddings, see Section~\ref{sec:Br} and Section~\ref{sec:BD}) is still outstanding. 

{\bf Question:}
Do there exist more than four fully accessible dense arc-components in non-standard (Section~\ref{sec:Br} and Section~\ref{sec:BD}) $\mathcal{E}$-embeddings of $X'$? Specifically, \black{what is the answer to the previous question if} $c$ is periodic? 

We lack the symbolic techniques to make a general construction that would answer the preceding question. Later in the paper we will see that for every $n\in\N$ there exists $X'$ such that the Brucks-Diamond embedding of $X'$ has $n$ fully-accessible dense arc-components.

\section{Bruin's embeddings $\Br(X')$}\label{sec:Br}
In this section we study the core $X'$ as a subset of the plane by Bruin's embedding constructed in \cite{Br1}, \ie for $L=1^{\infty}$. Recall that we denote these embeddings by $\phi_{\mathcal{R}}(X')$. If the slope $s=2$ and thus $X'=X$, \black{then the set of accessible points is exactly $\mathcal{R}$ and endpoint $\overline 0$ of $\mathcal{C}$, recall Remark~\ref{rem:new}. Specially, there are no simple dense canals.}

From now onwards we restrict to cases when $X\neq X'$ (\ie $s\neq 2$). Bruin showed in \cite{Br1} that $\sigma:\phi_{\mathcal{R}}(X')\to \phi_{\mathcal{R}}(X')$ is extendable to the plane and the extension is an orientation reversing planar homeomorphism.

\begin{theorem}\label{thm:henks}
	In embeddings $\Br(X')$ the arc-component $\mathcal{R}$ is fully accessible and no other point from $\Br(X')$ is accessible. There exists one simple dense canal for every $\Br(X')$.
\end{theorem}
\begin{proof}
		For embeddings given by Bruin in \cite{Br1} it holds that $L=1^{\infty}$ and thus $\mathcal{U}_L=\mathcal{R}$.\\
		We will explicitly calculate the top and bottom of an admissible cylinder $[a_n\ldots a_1]$ for $n\in\N$.
		
		\iffalse
		 If $1a_n\ldots a_1$ is not admissible, then it follows $1a_n\ldots a_1\succ c_1c_2\ldots c_{n+1}$ and therefore $a_n\ldots a_1\prec c_2\ldots c_{n+1}$, a contradiction. Since $1a_n\ldots a_1$ is always admissible, it holds that $1^{\infty}a_n\ldots a_1=L_{a_n\ldots a_1}(S_{a_n\ldots a_1})$ is also always admissible and thus $L_{a_n\ldots a_1}(S_{a_n\ldots a_1})=1^{\infty}a_n\ldots a_1$.\\
		 Now we determine $S_{a_{n}\ldots a_1} (L_{a_n\ldots a_1})$. Say that $0a_{n}\ldots a_1$ is not admissible. If $011a_{n}\ldots
		  a_1$ is admissible, then $S_{a_{n}\ldots a_1} (L_{a_n\ldots a_1})=1^{\infty}011a_{n}\ldots a_1$.
		If $011a_{n}\ldots a_1$ is not admissible, then it follows that $\kappa=1$ and since $T$ is non-renormalizable it holds that $\nu=10(11)^{k}0\ldots$ for some $k\in\N$.\\
		{\bf a.)} Say that $a_n=1$.\\
		Then $10(11)^{k}1\prec 10(11)^{k}0\ldots=\nu$ and thus admissible. Thus $S_{a_{n}\ldots a_1} (L_{a_n\ldots a_1})=1^{\infty}0(11)^{k}a_n\ldots a_1$.\\
		{\bf b.)} Say that $a_n=0$.\\
		If $0(11)^{k}0a_{n-1}\ldots a_1$ is admissible, then again $S_{a_{n}\ldots a_1} (L_{a_n\ldots a_1})=1^{\infty}0(11)^{k}a_n\ldots a_1$. If $0(11)^{k}0a_{n-1}\ldots a_1$ is not admissible, then it holds that $0(11)^{k+1}a_n\ldots a_1$ is admissible and $S_{a_{n}\ldots a_1} (L_{a_n\ldots a_1})=1^{\infty}0(11)^{k+1}a_n\ldots a_1$.\\
		Therefore, it holds that only basic arcs from $\mathcal{R}$ are accessible.
		\fi
		
		If $\#_1(a_n\ldots a_1)$ equals (does not equal) the parity of natural number $n$, then $L_{a_n\ldots a_1}=1^{\infty}a_n\ldots a_1$ ($S_{a_n\ldots a_1}=1^{\infty}a_n\ldots a_1$), since $1^{\infty}a_n\ldots a_1$ is always admissible by Lemma~\ref{lem:cylindersnonempty}. Also, $S_{a_n\ldots a_1}=1^{\infty}01^ka_n\ldots a_1$ ($L_{a_n\ldots a_1}=1^{\infty}01^ka_n\ldots a_1$), where $k\in\N_0$ is the smallest nonnegative integer such that $01^ka_n\ldots a_1$ is admissible.\\
		 Assume by contradiction that such $k$ does not exists. Then $01^ia_n\ldots a_1\prec c_2c_3\ldots\black{c_{n+i+2}}$ for every $i\in\N_0$. Since the word $01^i$ is always admissible, it follows that $c_2c_3\ldots\black{c_{i+2}}=01^i$ for every $i\in\N_0$, \ie $\nu=101^{\infty}$ and the unimodal interval map which corresponds to this kneading sequence $\nu$ is renormalizable, a contradiction \black{with $T$ being a tent map with slope $s\in (\sqrt{2},2]$}.
		 
		 Note that every $1^{\infty}a_n\ldots a_1$ is realized as an extremum of a cylinder, namely $1^{\infty}a_n\ldots 
		 a_1\\
		 =L_{a_n\ldots a_1}$ if $\#_1(a_n\ldots a_1)$ equals the parity of $n$ and $1^{\infty}a_n\ldots a_1=S_{a_n\ldots a_1}$ if $\#_1(a_n\ldots a_1)$ and $n$ are of different parity.
		
		 Note that if there was an accessible non-degenerate arc $Q\subset \Br(X')$ which is not the top or the bottom of any cylinder, then, since $\sigma$ is extendable, also every shift of $Q$ is accessible. But $\sigma$ \black{extends} arcs \black{in $X'$ (in the arc-length metric on $X'$)}, so there exists $i\in\N$ such that $\sigma^i(Q)$ contains a basic arc which is an extremum of a cylinder and thus $\sigma^i(Q)$ is a subset of $ \Br(\mathcal{R})$. Therefore, also $Q\subset \Br(\mathcal{R})$. We conclude that $\Br(\mathcal{R})$ corresponds to the circle of prime ends minus a point. The remaining prime end $P$ is either of the second, third, or fourth kind.\\
		 Assume first by contradiction that $P$ is of the second kind, \ie it corresponds to an accessible folding point. Since $\sigma:\phi_{\mathcal{R}}(X')\to \phi_{\mathcal{R}}(X')$ is extendable to the plane, it follows that $P$ needs to correspond to accessible point $\rho\black{=(\ldots,r,r,r),}$ \black{where $r$ is the non-zero fixed point of $T$} (since $\bar{\rho}=\ldots 11.11\ldots$ is the only $\sigma$-invariant itinerary of a point in $X'$). However, $A(1^{\infty})$ is the top or the bottom of a cylinder, so $\rho$ corresponds to a first kind prime end on the circle of prime ends, a contradiction.\\
		  Therefore, the remaining point $P$ on the circle of prime ends is either of the third or the fourth kind. Since $\mathcal{R}$ is dense in $X'$ (see Proposition 1 from \cite{BB}) and $\Br(\mathcal{R})$ bounds the canal in $\Br(X')$ it follows that $\Pi(P)=\Br(X')$ and thus $I(P)=\Pi(P)=\Br(X')$. Thus there exists one simple dense canal for every $\Br(X')$. 
\end{proof}

\section{Brucks-Diamond embeddings $\BD(X')$}\label{sec:BD}
In this section we study the core $X'$ as the subset of the plane by the Brucks-Diamond embedding $\BD$ constructed in \cite{BrDi}, \ie for $L=0^{\infty}1$. If the slope $s=2$, \ie $X=X'$ is the Knaster continuum, it follows from Corollary~\ref{cor:withC} and Remark~\ref{rem:endptKnaster} that $\mathcal{U}_L=\mathcal{C}$ is fully accessible and that no other point from $\BD(X')$ is accessible (observe the circle of prime ends). Specifically, there is no simple dense canal.

 Thus we restrict to cases when $X\neq X'$ (\ie $s\neq 2$).  Embeddings $\phi_{\mathcal{C}}(X')$ can be viewed as global attractors of orientation preserving planar homeomorphism \black{$h:\R^2\to \R^2$ so that $h|_{\phi_{\mathcal{C}(X')}}=\sigma$} and was described by Barge and Martin in \cite{BM}.
 \black{In other words}, $\sigma\colon \phi_{\mathcal{C}}(X')\to \phi_{\mathcal{C}}(X')$ can be extended to a planar homeomorphism \black{$h$}. For $\phi_{\mathcal{C}}(X)$ the set of accessible points is $\mathcal{C}$ and it forms an infinite canal which is dense in the core. However, if $\mathcal{C}$ is stripped off, the set of accessible points and the prime ends of $\phi_{\mathcal{C}}(X')$ become very interesting. Recently Boyland, de Carvalho and Hall gave in \cite{3G} a complete characterization of prime ends for embeddings $\varphi_{\mathcal{C}}$ of unimodal inverse limits satisfying certain regularity conditions which hold also for tent map inverse limits with indecomposable cores. In this section we obtain an analogous characterization of accessible points as in \cite{3G} using symbolic computations. What this sections adds to the results from \cite{3G} is the characterization of types of accessible folding points, specially in the irrational height case (see the definitions below). By knowing the exact symbolic description of points in $X'$ we can determine whether they are folding points or not, and if they are, whether they are endpoints of $X'$. The classification of accessible sets differentiates (as in \cite{3G}) according to the \emph{height} of the kneading sequence which we introduce shortly in this section (for more details see \cite{Ha}). Throughout this section the order $\prec_L$ corresponds with the standard parity-lexicographical order $\prec$.

We denote by $L'\black{\in\{0,1\}^{\infty}}$ the left infinite itinerary which is the largest admissible sequence in the embedding  
$X'$ for $L=0^{\infty}1$ (as in \cite{BrDi}) after $\mathcal{C}$ is removed. Therefore we need to find which basic arc of $X'$ is the closest to the basic arc $A(0^{\infty}1)$. This was calculated in \cite{BdCH}, \black{see Lemma~\ref{lem:BdCHL'}}.

\begin{definition}\label{def:height}
	Let $q\in (0, \frac{1}{2})$. For $i\in\N$ define
	$$\kappa_i(q)=
	\begin{cases}
	\lfloor \frac{1}{q}\rfloor -1, & \text{ if } i=1,\\
	\lfloor \frac{i}{q}\rfloor - \lfloor\frac{i-1}{q}\rfloor - 2, & \text{ if }  i\geq 2.
	\end{cases}
	$$
	If $q$ is irrational, we say that the kneading sequence $$\nu=10^{\kappa_1(q)}110^{\kappa_2(q)}110^{\kappa_3(q)}11\ldots$$ has \emph{height} $q$ or that it is of \emph{irrational type}.
	If $q=\frac{m}{n}$, where $m$ and $n$ are relatively prime, we define 
	$$c_q=10^{\kappa_1(q)}110^{\kappa_2(q)}11\ldots 110^{\kappa_m(q)}1,$$
	$$w_q=10^{\kappa_1(q)}110^{\kappa_2(q)}11\ldots 110^{\kappa_m(q)-1}.$$
	By $\hat a$ we denote the reverse of a word $a$, so $\hat w_q=0^{\kappa_m(q)-1}110^{\kappa_{m-1}(q)}11\ldots 110^{\kappa_1(q)}1$.
	We say that a kneading sequence has \emph{rational height} $q$ if $(w_q1)^{\infty}\preceq\nu\preceq 10(\hat w_q1)^{\infty}$. Denote by
	$\lhe(q):=(w_q1)^{\infty}$, $\rhe(q):=10(\hat w_q1)^{\infty}$. If $\lhe(q)\prec \nu\prec \rhe(q)$ we say that $\nu$ is of \emph{rational interior type}, and \emph{rational endpoint type} otherwise. Every kneading sequence that appears in the tent map family is either of rational endpoint, rational interior or irrational type, see Lemma 8 and Lemma 9 in \cite{BdCH} (for further information see also \cite{Ha}).
\end{definition}

\begin{remark}
	The values of $\kappa_i(q)$ can be obtained in the following way (see Lemma 2.5 in \cite{Ha} for details). Draw the graph $\Gamma_{\zeta}$ of the function $\zeta\colon\R\to\R$, $\zeta(\black{t})=q\black{t}$. Then $\kappa_i(q)=N_i-2$, where $N_i$ is the number of intersections of the graph $\Gamma_{\zeta}$ with vertical lines $\black{t}=N$, $N\in\N_0$ in the segment $\f i-1, i]$, see Figure~\ref{fig:line}. Note that it automatically follows that the word $\kappa_1(q)\kappa_2(q)\ldots\kappa_m(q)$ is a palindrome and thus $c_q$ is a palindrome. Furthermore, for every $i\in\N$ either $\ki=\kone$ or $\ki=\kone-1$.
\end{remark}

\begin{remark}\label{rem:palind}
	Assume $q=m/n$ is rational with $m$ and $n$ being relatively prime. Take $k\in\{1, \ldots, n-1\}$ such that $\lceil kq\rceil - kq$ obtains the smallest value; such $k$ is unique, since $m$ and $n$ are relatively prime. Denote by $K=\lceil kq\rceil$ and note that for every $i\in\{1, \ldots, k\}$ the line that joins $(0, 0)$ with $(k, K)$ intersects a vertical line in $\f i-1, i]$ if and only if $q\black{t}$ intersects a vertical line in $\f i-1, i]$. Thus $\kappa_1(q)\ldots\kappa_K(q)$ is a palindrome; it is the longest palindrome among $\kappa_1(q)\ldots \kappa_i(q)$ for $i<m$. By studying the line which joins $(k, K)$ with $(n, m)$ we conclude that $\kappa_{K+1}(q)\ldots\kappa_{m-1}(q)(\kappa_m(q)-1)$ is also a palindrome, see Figure~\ref{fig:line}. Thus for every rational $q$ there exist palindromes $Y, Z$ such that $c_q=Y1Z01$.
\end{remark}

\begin{remark}\label{rem:pal}
	Note that $\{\kappa_i(q)\}_{i\geq 1}$ is a Sturmian sequence for irrational $q$ and thus there exist infinitely many palindromic prefixes of increasing length (see \eg \cite{deL}, Theorem 5) which are of even parity.
	This can also be concluded by studying the rational approximations of $q$. Namely, if $k\in\N$ is such that $\lceil iq\rceil-iq$ achieves its minimum in $i=k$ for all $i\in\{1, \ldots, k\}$, then the word $\kappa_1(q)\ldots\kappa_k(q)$ is a palindrome. Note that $10^{\kappa_1(q)}11\ldots 110^{\kappa_k(q)}1$ is also a palindrome and it is an even word. By choosing better rational approximations of $q$ from above, we see that $k$ can be taken arbitrary large, and thus the beginning of $c_q$ consists of arbitrary long even palindromes.
\end{remark}

\begin{figure}[!ht]
	\centering
	\begin{tikzpicture}[scale=0.5]
	\draw[step=1,black] (0,0) grid (20,9);
	\draw (0,0) -- (20,9);
	\node at (-0.3,0.5) {\scriptsize $3$};
	\node at (-0.3,1.5) {\scriptsize $2$};
	\node at (-0.3,2.5) {\scriptsize $2$};
	\node at (-0.3,3.5) {\scriptsize $2$};
	\node at (-0.3,4.5) {\scriptsize $3$};
	\node at (-0.3,5.5) {\scriptsize $2$};
	\node at (-0.3,6.5) {\scriptsize $2$};
	\node at (-0.3,7.5) {\scriptsize $2$};
	\node at (-0.3,8.5) {\scriptsize $3$};
	\draw[ultra thick] (0,5)--(11,5)--(11,0);
	\node at (11,-0.5) {\scriptsize $k$};
	\node at (20.5,9.5) {\scriptsize $(20,9)$};
	\node at (-0.5,-0.5) {\scriptsize $(0,0)$};
	\end{tikzpicture}
	\caption{Calculating $\kappa_i(q)$ by counting the intersections of the line $q\black{t}$ with vertical lines over integers. The picture shows the values $N_i$ for $q=\frac{9}{20}$. It follows that $c_q=101111111101111111101=(101111111101)1(111111)01=Y1Z01$. The decomposition into palindromes $Y, Z$ follows since $\lceil \frac{9}{20}k\rceil - \frac{9}{20}k$ obtains its minimum for $k=11=\lfloor\frac{5}{q}\rfloor$ (bold line in the figure).}
	\label{fig:line}
\end{figure}
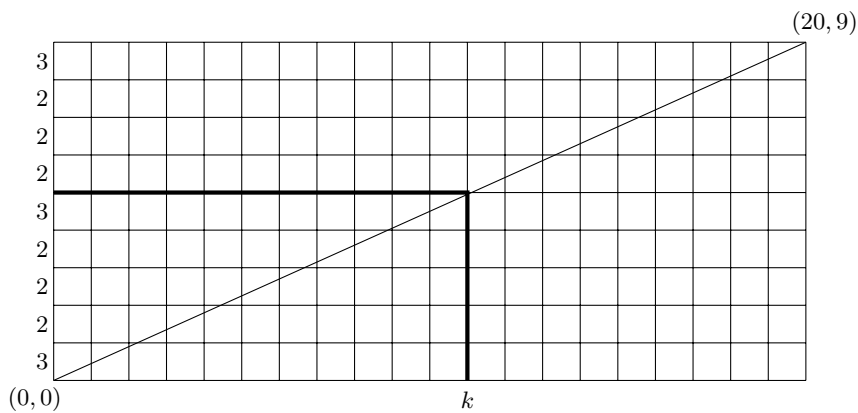

\begin{lemma}\label{lem:lhe/rhe}
	Let $q=\frac{m}{n}$. Then there exists $N\in\N$ such that $\sigma^N(\rhe(q))=\lhe(q)$.
\end{lemma}
\begin{proof}
	Recall that $\lhe(q)=(w_q1)^{\infty}$, $\rhe(q)=10(\hat w_q1)^{\infty}$, where $c_q=w_q01$. By Remark~\ref{rem:palind}, there exist palindromes $Y, Z$ such that $c_q=Y1Z01$, so $w_q=Y1Z$. It follows that $\lhe(q)=(Y1Z1)^{\infty}$ and $\rhe(q)=10(Z1Y1)^{\infty}$ which finishes the proof.
\end{proof}

\begin{remark}
	The height of a kneading sequence is the rotation number of the natural mapping on the circle of prime ends. We will only need symbolic representation of the height of a kneading sequence here; for a more detailed study of height see \cite{Ha}.
\end{remark} 

\black{We also remark that the notation $\ovl{x}$ is used in this section for substantially different purpose than it was used in the rest of the file. Namely: 
}
\begin{definition}
	Given an infinite sequence $\overrightarrow x=x_1x_2x_3\ldots$, we denote in this section its reverse by $\ovl x=\ldots x_3x_2x_1$.
\end{definition}

\begin{lemma}[\cite{BdCH}, Lemma 13]\label{lem:BdCHL'}
	Let $X'$ be embedded with $\phi_{\mathcal{C}}$. Denote by $L'$ the largest admissible basic arc in $X'$ and by $\nu$ the kneading sequence corresponding to $X'$. Then,
	$$L'=
	\begin{cases}
	\ovl{\rhe(q)}, & \text{ if $\lhe\black{(q)}\prec \nu\preceq\rhe(q)$,}\\
	\ovl{\nu}, & \text{ if $q$ is irrational or $\nu=\lhe(q)$}.
	\end{cases}$$
\end{lemma}

\subsection{Irrational height case}
Assume that $q$ is irrational and note that the map $T$ is then long-branched (since the kneading map is bounded, see \black{\eg}\cite{Br3}). Therefore, every proper subcontinuum is a point or an arc (see Proposition 3 in \cite{BB}) and consequently, every composant is an arc-component and thus either a line or a ray (\black{furthermore} every composant of $X'$ is dense in $X'$ so an arc cannot be a composant of $X'$). We will show that the basic arc $A(L')$ (which is fully accessible) contains an endpoint of $X'$. Furthermore, we will prove that the basic arc adjacent to $A(L')$ is not an extremum of a cylinder, and thus contains a folding point which is not an endpoint. \black{Moreover}, the ray $\mathcal{U}_{L'}$ is partially accessible; only a compact arc $Q\subset \mathcal{U}_{L'}$ is fully accessible and $\mathcal{U}_{L'}\setminus Q$ is not accessible. Since $\sigma$ is extendable, also $\sigma^i(Q)$ is accessible for every $i\in\Z$. Later in this subsection we show that no other non-degenerate arc except of $\sigma^i(Q)$ for every $i\in\Z$ is fully accessible. From the circle of prime ends we then see that there is still a Cantor set of points remaining to be associated to either accessible points or infinite canals of $\phi_{\mathcal{C}}(X')$. We prove that the remaining points on the circle of prime ends correspond to accessible endpoints of $\BD(X')$ and are thus second kind prime ends. Moreover, we prove that every endpoint from $\BD(X')$ is accessible. This is an extension of Theorem 4.46 from \cite{3G}. In this subsection the usage of variables $m$ and $n$ should not be confused with the values in the fraction $q=\frac{m}{n}$ which will be used in the rational height case later in the paper.

\begin{lemma}
	If $\nu$ is of irrational type, then $\tau_R(L')=\infty$ and $A(L')$ is non-degenerate.
\end{lemma}
\begin{proof}
If $\nu$ is of irrational type, then the bonding map $T$ is long-branched, so every basic arc in $X'$ is non-degenerate. To prove the first claim, first note that by Lemma~\ref{lem:BdCHL'} it holds that $L'=\ovl{\nu}$. Remark~\ref{rem:pal} implies that there exist infinitely many even palindromes of increasing length at the beginning of $\nu$. Thus there exists a strictly increasing sequence $(m_i)_{i\in\N}$ such that $l'_{m_i}\ldots l'_{1}=c_1\ldots c_{m_i}$ and $\#_1(c_1\ldots c_{m_i})$ is even for every $i$. Thus it follows that $\tau_{R}(L')=\infty$. 
\end{proof}

The following remark follows from Remark 15 in \cite{BdCH} and the fact that we restrict our study only on the tent map family.

\begin{remark}\label{rem:adm}
	 If $\nu$ is of irrational or rational endpoint type, it holds that $\ovl{t}\in\{0,1\}^{\infty}$ is admissible (\ie every subword of $\ovl{t}$ is admissible) if and only if $\ovr{t}$ is admissible (\ie every subword of $\ovr{t}$ is admissible).
\end{remark}

\begin{lemma}\label{lem:horror2}
Let $\nu$ be either of irrational or rational endpoint type and $X'$ embedded with $\BD$. Then every extremum of a cylinder of $\BD(X')$ belongs to $\sigma^i(L')$ for some $i\in\black{\N_0}$. 
\end{lemma}
\begin{proof}
	Take an admissible finite word $a_n\ldots a_1\in\{0,1\}^{n}$ and pick the smallest $k\in\{0, \ldots, n-1\}$ such that $a_{n}\ldots a_{k+1}=c_{n-k+1}\ldots c_2$. If there is no such $k$ we set $k=n$. 
	
	Assume first that $k>1$ and note that $a_k=1$, \black{otherwise either $0^{\kappa+1}\subset a_n\ldots a_1$ or $k$ is not the smallest such number.}\\
	 Assume that $\#_1(a_{k-1}\ldots a_1)$ is even and let us calculate $L_{a_n\ldots a_1}$. If admissible, the word $L'a_{k-1}\ldots a_1$ is the largest in the cylinder $[a_n\ldots a_1]$. Assume that $L'a_{k-1}\ldots a_1$ is not admissible. By Remark~\ref{rem:adm}, since both $L'$ and $a_{k-1}\ldots a_1$ are admissible, there exists $i\in\{1,\ldots,k-1\}$ such that $a_i\ldots a_{k-1} l'_1\ldots l'_j$ is not admissible for some $j\geq 1$. 
	If $j\leq n-k+1$, then $a_i\ldots a_{k-1} l'_1\ldots l'_j$ is a subword of $a_1\ldots a_n$ which is not admissible, a contradiction. Assume that $j>n-k+1$. In this case the word $a_i\ldots a_{k-1} l'_1\ldots l'_j\nsubseteq a_1\ldots a_n$ is not admissible, but then $a_i\ldots a_{n}=c_2\ldots c_{2+n-i}$ which is a contradiction with $k$ being the smallest such that $a_n\ldots a_{k+1}=c_{n-k}\ldots c_2$.
	 If $\#_1(a_{k-1}\ldots a_1)$ is odd we obtain that $S_{a_n\ldots a_1}=L'a_{k-1}\ldots a_1$ using analogous arguments as above.\\
	Now assume that $\#_1(a_{k-1}\ldots a_1)$ is odd and we calculate $L_{a_n\ldots a_1}$. Say that $\#_1(a_n\ldots a_k)$ is odd. Therefore, since we want to calculate the largest basic arc in the cylinder $[a_n\ldots a_1]$, we need to set $L_{a_n\ldots a_1}=\ldots1a_n\ldots a_1$, and note that $1a_n\ldots a_1$ is always admissible by Lemma~\ref{lem:cylindersnonempty}. Then, knowing that $\#_1(a_n\ldots a_k)$ is odd it follows from the special structure of $\nu$ in the irrational height case that the kneading sequence starts as $a_k\ldots a_n11$ or $a_k\ldots a_n0$ and thus the word $a_k\ldots a_n10$ is admissible. It follows that $L'a_n\ldots a_1$ is admissible and equals to $L_{a_n\ldots a_1}$. If $\#_1(a_n\ldots a_k)$ is even, it follows from the structure of $\nu$ (blocks of ones in $\nu$ are of even length) that $a_n=1$ and $a_k\ldots a_n$ ends in odd number of ones. The word $a_k\ldots a_n0^{\kone}$ is thus admissible and therefore $L_{a_n\ldots a_1}=L'a_{n-1}\ldots a_1$. Calculations for $S_{a_n\ldots a_1}$ when $\#_1(a_{k-1}\ldots a_1)$ is even follow analogously.
	
	Now say that $k=1$. Then $L_{a_n\ldots a_1}=L'$. 
	We conclude as in the preceding paragraph that if $\#_1(a_n\ldots a_1)$ is even, then $S_{a_n\ldots a_1}=L'a_{n-1}\ldots a_1$ and if $\#_1(a_n\ldots a_1)$ is odd, then $S_{a_n\ldots a_1}=L'a_{n}\ldots a_1$.
	
	If $k=0$, then $a_1\ldots a_n=c_2\ldots c_{n+1}$. So $S_{a_n\ldots a_1}=S=\ldots c_4c_3c_2$. To calculate $L_{a_n\ldots a_1}$, let $k'$ be the smallest natural number such that $a_{n}\ldots a_{k'}=c_{n-k'+1}\ldots c_1$. If $k'$ does not exist, set $k'=n+1$. From the structure of $\nu$ (blocks of ones in $\nu$ are of even length) it follows that $\#_1(a_{k'-1}\ldots a_1)$ is odd. The rest of the proof for this case follows the same as in the case for $k>1$.
\end{proof}
\begin{lemma}\label{lem:irrpartially}
	Assume $\nu$ is of irrational type and $X'$ embedded with $\BD$. Then the only basic arc from $\mathcal{U}_{L'}$ which is an extremum of a cylinder is $A(L')$. 
\end{lemma}
\begin{proof}
	Let $a_n\ldots a_1$ be an admissible word for some $n\in\N$.
	 If $n=1$, note that $L_1=L'\subset \mathcal{U}_{L'}$ and $L_0, S_0, S_1\not\subset \mathcal{U}_{L'}$, since $\nu$ is not (pre)periodic \black{(note that $L'=\ovl{\nu}$ and observe the structure of $\nu$)}.\\
	 Now assume that $n\geq 2$. Since $\nu$ is not (pre)periodic, the proof of Lemma~\ref{lem:horror2} gives that if $L_{a_n\ldots a_1}$ or $S_{a_n\ldots a_1}$ are contained in $\mathcal{U}_{L'}$, then $a_1\ldots a_n=c_1\ldots c_n$ (since otherwise $L_{a_n\ldots a_1}$ or $S_{a_n\ldots a_1}$ would be contained in $\sigma^{i}(\mathcal{U}_{L'})$ for some $i\in\Z\setminus \{0\}$). But then, following the proof of Lemma~\ref{lem:horror2} it holds that $L_{a_n\ldots a_1}=L'$ and $S_{a_n\ldots a_1}=L'a_n\ldots a_1$ or $S_{a_n\ldots a_1}=L'a_{n-1}\ldots a_1$, depending on the parity of $\#_1(a_n\ldots a_1)$. Since $L'a_n\ldots a_1\in \sigma^n(L')$ and $L'a_{n-1}\ldots a_1\in \sigma^{n-1}(L')$ the only extremum of a cylinder in $\mathcal{U}_{L'}$ is $A(L')$.
\end{proof}

\begin{remark}
	It follows from Lemma~\ref{lem:irrpartially} that when $\nu$ has irrational height, then $\mathcal{U}_{L'}$ is partially accessible. More precisely, from Proposition~\ref{prop:wiggles} it follows that $\ovl{l(L')}=\ldots 110^{\kthree}110^{\ktwo}110^{\kone-1}11$ contains a folding point $p$ and $A(L')\cup[a, p]$ is fully accessible, where $a$ denotes the left endpoint of $\ovl{l(L')}$. It follows from Corollary~\ref{cor:class} that no other point from $\mathcal{U}_{L'}$ (which is a ray) is accessible. Since $\sigma:\BD(X')\to \BD(X')$ is extendable to the plane, also $\sigma^i(A(L')\cup[a, p])$ is accessible for every $i\in\Z$. Moreover, those are the only accessible non-degenerate arcs, since $\sigma$ is extendable  \black{to a planar homeomorphism and furthermore extends every arc in $X'$} (see the discussion in the proof of Theorem~\ref{thm:henks}). In the lemmas to follow we prove that the remaining Cantor set of points on the circle of prime ends correspond to the endpoints of $\BD(X')$, and that all endpoints of $\BD(X')$ are accessible when $\nu$ is of irrational type.
\end{remark}

	The following lemma follows directly from the fact that $(\kappa_i(q))_{i\in\N}$ is Sturmian, but we prove it here for the sake of completeness. Say that $q\in(0, \frac{1}{2})$ is irrational. Denote by $\kappa=\kappa_1(q)$, so $\kappa_i(q)\in\{\kappa, \kappa-1\}$ for every $i\in\N$.

\begin{lemma}\label{lem:sturm}
	Let $q\in(0, \frac{1}{2})$ be irrational \black{with corresponding $(\kappa_i(q))_{i\in\N}$}. There exists $J\in\N$ such that if $\kappa_i(q)\kappa_{i+1}(q)
	\ldots\kappa_{i+N}(q)\kappa_{i+N+1}(q)=\kappa(\kappa-1)^{N}\kappa$, then $N\in\{J, J+1\}$.
\end{lemma}
\begin{proof}
	Let $J\in\N$ be such that $\kappa_2(q)= \ldots= \kappa_{J+1}(q)=\kappa-1$ and $\kappa_{J+2}(q)=\kappa$ \black{(such $J$ indeed exists since $q$ is irrational)}. So there exists a sequence of $J$ consecutive $(\kappa-1)$s. Denote by $H_n=\lfloor\frac nq\rfloor$ for $n\in\N$ and note that the function $g\colon\N\to\R$ given by $g(k)=\lceil kq\rceil-kq$ achieves its minimum on $\f 0, H_{J+2}]$ in $H_{J+2}$ (since $J+2$ is minimal index $a>1$ for which $\kappa_{a}=\kappa$). If we translate the graph of function $\zeta(t)=qt$ by $+\delta$  where $\delta\in(0,g(H_{J+2})]$, then the sequence of consecutive number of intersections with vertical lines over integers begins again with $(\kappa+2)(\kappa+1)^{J}(\kappa+2)$. Since $g$ restricted to $[0, H_{J+2})$ achieves its minimum in $H_1$, if $\delta\in(g(H_{J+2}), g(H_1))$, the sequence corresponding to the number of times the graph of $\zeta+\delta$ intersects vertical lines over integers begins with $(\kappa+2)(\kappa+1)^{J+1}(\kappa+2)$, see Figure~\ref{fig:transl}. Fix $i\geq 2$ such that $\kappa_i(q)=\kappa$. Note that then $g(H_{i-1}+1)<g(H_{1})$ since otherwise $qH_{i-1}>i-1$ which is a contradiction. So the graph of $\zeta$ on $[H_{i-1}+1,\infty)$ can be obtained from the graph of $\zeta$ on $[0, \infty)$ by translating it by $+\delta$ for $\delta\in(0,g(H_1))$ which finishes the proof.  
\end{proof}

\begin{figure}[!ht]
	\centering
	\begin{tikzpicture}[scale=0.5]
	\draw[step=1,black] (0,0) grid (20,9);
	\draw (0,0) -- (20,9-1/30);
	\node at (-0.3,0.5) {\scriptsize $3$};
	\node at (-0.3,1.5) {\scriptsize $2$};
	\node at (-0.3,2.5) {\scriptsize $2$};
	\node at (-0.3,3.5) {\scriptsize $2$};
	\node at (-0.3,4.5) {\scriptsize $3$};
	\node at (20.3,0.5) {\scriptsize $3$};
	\node at (20.3,1.5) {\scriptsize $2$};
	\node at (20.3,2.5) {\scriptsize $2$};
	\node at (20.3,3.5) {\scriptsize $2$};
	\node at (20.3,4.5) {\scriptsize $2$};
	\node at (20.3,5.5) {\scriptsize $3$};
	\draw[ultra thick] (0,5)--(11,5)--(11,0);
	\draw[ultra thick] (0,4)--(8,4)--(8,0);
	\draw[ultra thick] (1,0)--(1,0.44833);
	\node at (1.3,0.25) {\scriptsize $q$};
	\draw[dashed] (0,0.09)--(20,9-1/30+0.09);
	\node at (11,-0.5) {\tiny $H_{J+2}$};
	\node at (8,-0.5) {\tiny $H_{J+1}$};
	\node at (2,-0.5) {\tiny $H_{1}$};
	\node at (-0.5,-0.5) {\scriptsize $(0,0)$};
	\end{tikzpicture}
	\caption{The graph of $qt$ for $q\approx 0.4483\ldots$ with the number of intersections with vertical integer lines on the left. Dashed line represents the graph of $qt$ translated by $\delta\in(g(H_{J+2}), g(H_1))$. On the right we count the intersections of the translated graph with vertical integer lines.}
	\label{fig:transl}
\end{figure}
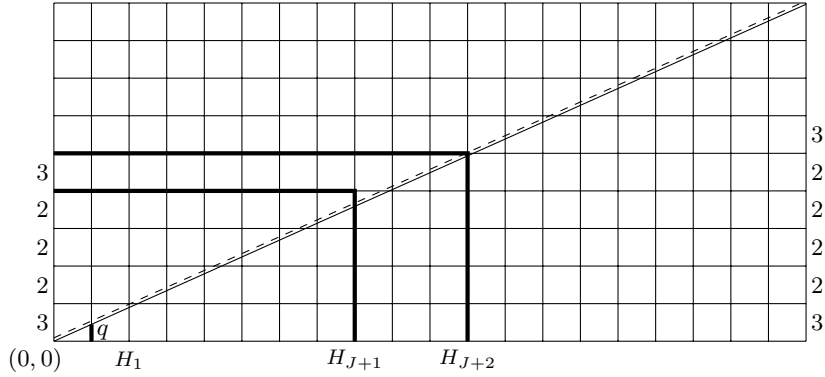

\begin{lemma}\label{lem:notexistant}
	Let $q\in(0, \frac{1}{2})$ be irrational \black{with corresponding $(\kappa_i(q))_{i\in\N}$} and $i, N\in\N$  such that $\kappa_{i+1}(q)\ldots\kappa_{i+N}(q)=\kappa_1(q)\ldots\kappa_N(q)$ and $\kappa_{i+N+1}(q)\neq\kappa_{N+1}(q)$. Then $\kappa_1(q)\ldots\kappa_{N+1}(q)$ is a palindrome. Moreover, $\kappa_{i+N+2}(q)=\kappa_1(q)$. If $K\in\N$ is such that $\kappa_{i+N+2}(q)\ldots\kappa_{i+N+K+1}(q)=\kappa_1(q)\ldots\kappa_{K}(q)$ and $\kappa_{i+N+K+2}(q)\neq\kappa_{K+1}(q)$, then $\kappa_{K+1}(q)\ldots\kappa_1(q)\kappa_{i+N+1}(q)\ldots\kappa_{i+1}(q)
	=\kappa_1(q)\ldots\kappa_{K+N+1}(q)$.
\end{lemma}
\begin{proof}
	For $i\in\N$ denote by $H_i=\lfloor\frac{i}{q}\rfloor$ and let $f\colon\N\to\R$ be given by $f(t)=tq-\lfloor tq\rfloor$. Note that the graph of $\zeta(t)=qt$ restricted to $\f H_i+1, \infty)$ is a translation of the graph of $\zeta$ on $\f 0, \infty)$ by some $\delta>0$ (see \eg Figure~\ref{fig:transl}). The conditions $\kappa_{i+1}(q)\ldots\kappa_{i+N}(q)=\kappa_1(q)\ldots\kappa_N(q)$ and $\kappa_{i+N+1}(q)\neq\kappa_{N+1}(q)$ imply that the global minimum of $f$ on $[H_i, H_{i+N+1}+1]$ is $H_{i+N+1}+1$. So the graph of $\zeta-f(H_{i+N+1}+1)$ on $[H_i, H_{i+N+1}+1]$ intersects vertical lines over integers the same number of times as $\zeta$ except for the point $(H_{i+N+1}+1, i+N+1)$. We conclude that $(\kappa_{i+N+1}(q)+1)\kappa_{i+N}(q)\ldots\kappa_{i+1}(q)=\kappa_1(q)\ldots\kappa_{N+1}(q)$ which concludes the first part of the proof. To see that  $\kappa_{i+N+2}(q)=\kappa_1(q)$ use Lemma~\ref{lem:sturm}.\\
	For the last part of the proof assume that $K\in\N$ is such that $\kappa_{i+N+2}(q)\ldots\kappa_{i+N+K+1}(q)\\=\kappa_1(q)\ldots\kappa_{K}(q)$ and $\kappa_{i+N+K+2}(q)\neq\kappa_{K+1}(q)$. That implies that the global minimum of $f$ on $[H_i, H_{i+N+K+2}+1]$ is $H_{i+N+K+2}+1$. Again by translating the graph of $\zeta$ on $[H_i, H_{i+N+K+2}+1]$ by $-f(H_{i+N+K+2}+1)$ we conclude the second part of the proof, see Figure~\ref{fig:transl2}. 
\end{proof}

\begin{figure}[!ht]
	\centering
	\begin{tikzpicture}[scale=0.5, yscale=1]
	\draw[step=1,black] (0,0) grid (25,11);
	\draw (0,0) -- (25,25*0.443);
	\node at (-0.3,0.5) {\scriptsize $3$};
	\node at (-0.3,1.5) {\scriptsize $2$};
	\node at (-0.3,2.5) {\scriptsize $2$};
	\node at (-1.9,3.5) {\scriptsize $\kappa_{i+1}+2=3$};
	\node at (-0.3,4.5) {\scriptsize $2$};
	\node at (-0.3,5.5) {\scriptsize $2$};
	\node at (-0.3,6.5) {\scriptsize $2$};
	
	\node at (25.3,10.5) {\scriptsize $3$};
	\node at (25.3,9.5) {\scriptsize $2$};
	\node at (25.3,8.5) {\scriptsize $2$};
	\node at (25.3,7.5) {\scriptsize $3$};
	\node at (25.3,6.5) {\scriptsize $2$};
	\node at (25.3,5.5) {\scriptsize $2$};
	\node at (25.3,4.5) {\scriptsize $2$};
	\node at (25.3,3.5) {\scriptsize $3$};
	
	\node at (-2.35,7.5) {\scriptsize $\kappa_{i+N+2}+2=3$};
	\node at (-0.3,8.5) {\scriptsize $2$};
	\node at (-2.75,9.5) {\scriptsize $\kappa_{i+N+K+1}+2=2$};
	\node at (-0.3,10.5) {\scriptsize $2$};
	\draw[ultra thick] (0,7)--(15,7)--(15,0);
	\draw[ultra thick] (0,3)--(6,3)--(6,0);
	\draw[ultra thick] (0,11)--(24,11)--(24,0);
	%\draw[ultra thick] (1,0)--(1,0.44833);
	%\node at (1.2,0.2) {\tiny $q$};
	\draw[dashed] (7,7*0.443-25*0.443+11)--(25,11);
	\node at (6,-0.5) {\tiny $H_{i}$};
	\node at (15,-0.5) {\tiny $H_{i+N+1}$};
	\node at (24,-0.5) {\tiny $H_{i+N+K+2}$};
	\node at (-0.5,-0.5) {\scriptsize $(0,0)$};
	\end{tikzpicture}
	\caption{Graphic representation of the proof of Lemma~\ref{lem:notexistant} for $q\approx 0.443\ldots$. Dashed line represents the graph of $\zeta(t)=qt$ on $\f H_i+1, H_{i+N+K+2}+1]$ translated by $-f(H_{i+N+K+2}+1)$. On the right side of the grid we count intersections of the dashed line with vertical integer lines.}
	\label{fig:transl2}
\end{figure}
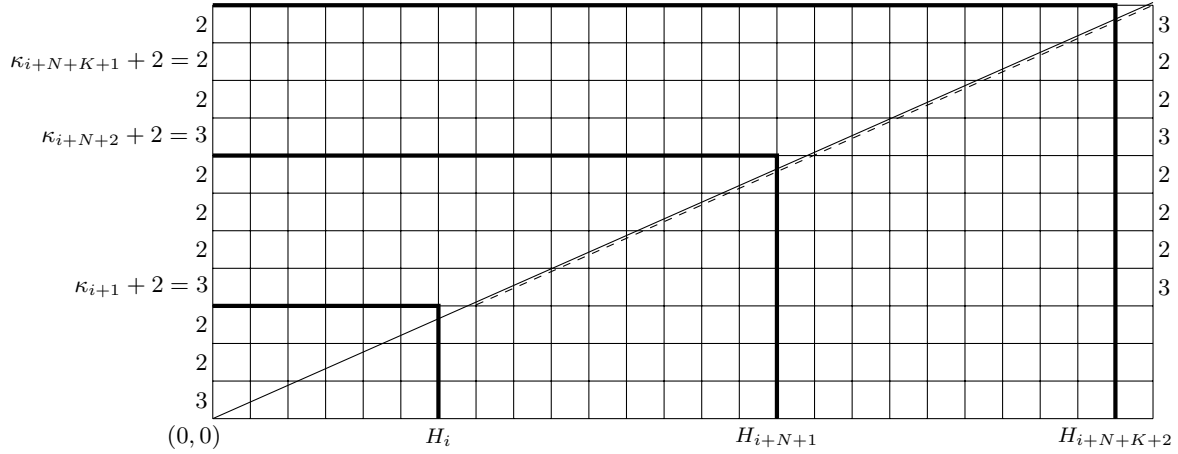

\begin{lemma}
	If $\nu$ is of irrational type or $\nu=lhe(q)$, then every endpoint of $\BD(X')$ is accessible.
\end{lemma}

\begin{proof}
	Let $e\in X'$ be an endpoint and let $\ovl{e}$ denote the left infinite symbolic description of $e$.\\
	Assume that $\tau_{R}(\ovl{e})=\infty$ and thus there exists a strictly increasing sequence $(m_i)_{i\in\N}$ such that $c_1\ldots c_{m_i}=e_{m_i}\ldots e_1$ and $\#_1(e_{m_i}\ldots e_1)$ is even. Assume $(m_i)_{i\in\N}$ is the complete sequence for $e$ (see Definition~\ref{def:complete}).\\
	Assume that for infinitely many $i\in\N$ there exist admissible left infinite itineraries $\ovl{x}^{O(i)}\prec_{L}\ovl{e}\prec_{L} x^{I(i)}$ so that $\ovl{x}^{O(i)}, \ovl{x}^{I(i)}\to \ovl{e}$ as $i\to \infty$, $\ovl{x}^{O(i)}, \ovl{x}^{I(i)}$ differ only at the index $m_i+1$ and equal $c_1\ldots c_{m_i}$ on the first $m_i$ places (if we are able to construct such $\ovl{x}^{O(i)}, \ovl{x}^{I(i)}$ the arcs will cap the endpoint $e$ which would thus be inaccessible - compare with the proof of Theorem~\ref{thm:capped}). So, $\ovl{x}^{O(i)}$ and $\ovl{x}^{I(i)}$ are of the form:
	$$\ovl{x}^{I(i)}=\ldots110^{\kone}110^{\ktwo}11\ldots110^{\kappa_j(q)}1.$$
	$$\hspace{15pt}\ovl{e}=\ldots110^{\kone}110^{\ktwo}11\ldots110^{\kappa_j(q)}1.$$
	$$\ovl{x}^{O(i)}=\ldots010^{\kone}110^{\ktwo}11\ldots110^{\kappa_j(q)}1.$$
	Note first that $0e_{m_i}\ldots e_1$ is indeed admissible.
	Since $\#_1(e_{m_i}\ldots e_{1})$ is even it holds that $\ovl{x}^{O(i)}\prec_{L}\ovl{e}$ for every $i\in\N$. Thus we need to find $\ovl{x}^{I(i)}\succ_{L}\ovl{e}$ in order to cap $e$.\\
	Denote by $J\in\N$ the smallest natural number such that 
	$$\ovl e=\ldots 110^{\kappa_J(q)-1}110^{\kappa_{J-1}(q)}11\ldots 110^{\kappa_2(q)}110^{\kappa_1(q)}110^{\kappa_2(q)}11\ldots 110^{\kappa_j(q)}1.$$
	By Lemma~\ref{lem:notexistant} it follows that $\kappa_J(q)\ldots\kappa_2(q)\kappa_1(q)$ is a palindrome and thus $10^{\kappa_J(q)}11\\
	0^{\kappa_{J-1}(q)}11\ldots 110^{\kappa_2(q)}11$ equals the beginning of $\nu$.
	
	We want to find $\ovl x^{I(i)}\succ\ovl e$. Note that none of $00^{\kappa_2(q)}110^{\kappa_1(q)}, \ldots, 00^{\kappa_{J-1}(q)}11\ldots 110^{\kappa_1(q)}$ are admissible. If we set 
	$$\ovl x^{I(i)}=\ldots 00^{\kappa_J(q)-1}110^{\kappa_{J-1}(q)}11\ldots 110^{\kappa_2(q)}110^{\kappa_1(q)}11\ldots 110^{\kappa_j(q)}1,$$
	then also
	$$\ovl x^{O(i)}=\ldots 00^{\kappa_J(q)-1}110^{\kappa_{J-1}(q)}11\ldots 110^{\kappa_2(q)}010^{\kappa_1(q)}11\ldots 110^{\kappa_j(q)}1.$$
	But since $100^{\kappa_J(q)-1}110^{\kappa_{J-1}(q)}11\ldots 110^{\kappa_2(q)}1$ equals the beginning of $\nu$, the word $00^{\kappa_J(q)-1}11\\
	0^{\kappa_{J-1}(q)}11
	\ldots 110^{\kappa_2(q)}0$ is not admissible, a contradiction.\\
	Thus we \black{have no other option but} to set 
	$$\ovl x^{I(i)}=\ldots 110^{\kappa_J(q)-1}110^{\kappa_{J-1}(q)}11\ldots 110^{\kappa_2(q)}110^{\kappa_1(q)}11\ldots 110^{\kappa_j(q)}1.$$
	By Lemma~\ref{lem:sturm} it follows that 
	$$\ovl e=\ldots 110^{\kappa_1(q)}110^{\kappa_J(q)-1}110^{\kappa_{J-1}(q)}11\ldots 110^{\kappa_2(q)}110^{\kappa_1(q)}110^{\kappa_2(q)}11\ldots 110^{\kappa_j(q)}1.$$
	Now take the smallest $K\in\N$ such that 
	$$\ovl e=\ldots 110^{\kappa_{K+1}(q)-1}110^{\kappa_K(q)}11\ldots 110^{\kappa_1(q)}110^{\kappa_J(q)-1}1$$
	$$	
	10^{\kappa_{J-1}(q)}11\ldots 110^{\kappa_2(q)}110^{\kappa_1(q)}110^{\kappa_2(q)}11\ldots 110^{\kappa_j(q)}1.$$
	By Lemma~\ref{lem:notexistant} it follows that $10^{\kappa_{K+1}(q)}11\ldots110^{\kappa_1(q)}110^{\kappa_{J}(q)-1}110^{\kappa_{J-1}(q)}11\ldots110^{\kappa_2(q)}11$ is the beginning of $\nu$. Thus we analogously argue that
	 $$\ovl x^{I(i)}=\ldots 110^{\kappa_{K+1}(q)-1}110^{\kappa_K(q)}11\ldots 110^{\kappa_1(q)}110^{\kappa_J(q)-1}1
	 $$
	 $$
	 10^{\kappa_{J-1}(q)}11\ldots 110^{\kappa_2(q)}110^{\kappa_1(q)}110^{\kappa_2(q)}11\ldots 110^{\kappa_j(q)}1,$$
	 which agrees with $\ovl{e}$.
	 Continuing inductively we conclude that $\ovl x^{I(i)}=\ovl e$. Thus $e$ is not capped \black{and by Remark~\ref{rem:notcapped} is accessible}.
\end{proof}

\begin{remark}
	\black{In this section we expand} the definition of Type 3 folding point introduced in the preperiodic orbit case. A point $p$ will be called a Type 3 folding point, if it is not an endpoint, it is accessible, and there is an arc $p\in Q\subset\mathcal{U}_p$ such that $Q\setminus\{p\}$ is not accessible, see Figure~\ref{fig:preper}.
\end{remark}

\begin{lemma}
If $\nu$ is of irrational type or rational endpoint type and $X'$ is embedded with $\BD$, then there are no Type 3 folding points.
\end{lemma}
\begin{proof}
	Assume by contradiction that there is a basic arc $\ovl x=\ldots x_2x_1$ and an accessible folding point $p\in A(\ovl x)$ of Type 3. Since $p$ is a folding point by Proposition~\ref{prop:foldpts} there exist blocks of symbols of $\nu$ of increasing length in $\ovl{x}$. 
	%$(m_i)$ and $(n_i)$ such that $x_{m_i}\ldots x_1=c_{n_i+1}\ldots c_{n_i+m_i}$ for $i\in\N$.
	
	We claim that if $c_{n}\ldots c_{n+k}=c_{m}\ldots c_{m+k}$ for some $m,n\in\N$ and there exists $i\in\{0,\ldots, k\}$ such that $c_{n+i}=0$, then \black{the parities of $\#_1(c_1\ldots c_{n+k})$ and $\#_1 (c_1\ldots c_{m+k})$ are the same} (then all the wiggles will accumulate on $A(\ovl{x})$ from exactly one side of $p$ as in Figure~\ref{fig:type2}).
	Indeed, take the largest such index $i$. Then it follows that $c_n\ldots c_{n+i-1}=1^{i}$. If $i$ is even (odd) it holds that $\#_1(c_1\ldots c_{n-1})$ is odd (even), which proves the claim.\\
	Therefore, if for $\ovl{x}=\ldots x_2x_1$ there exists $i\in\{0,\ldots k\}$ such that $c_{n+i}=0$ and $x_{\black{k+1}}\ldots x_1=c_n\ldots c_{n+k}$ it follows that $A(\ovl{x})$ contains no Type 3 folding point.
	
	Now assume that $\ovl{x}=1^{\infty}$. If $\kone>1$, then $\ldots 1101^{\alpha}\succ_L \ovl{x}\succ_L \ldots 1101^{\alpha+1}$ for every odd $\alpha\in\N$ and both $\ldots 1101^{\alpha}$ and $\ldots 1101^{\alpha+1}$ project to $[T^2(c),T(c)]$, which is again a contradiction with $p$ being a Type 3 folding point.\\
	If $\kone=1$, then $\nu=101^{\beta}0\ldots$ for some even $\beta\in\N$. Then, basic arcs with symbolic description $1^{\infty}01^{\gamma}$ for every $\gamma>\beta$ project to $[T^2(c),T(c)]$ and we get an analogous conclusion as in the preceding paragraph.   
\end{proof}

\begin{lemma}
	If $\nu$ is of irrational type, then there exist no third and \black{no} fourth kind prime ends corresponding to $\BD(X')$.
\end{lemma}

\begin{proof}
	Since the embedding $\BD(X')$ is realized as an alignment of basic arcs along vertically embedded Cantor set connected with semi-circles, we can study crosscuts which are vertical segments in the plane joining two adjacent cylinders, see Figure~\ref{fig:fourth}. Note that every infinite canal is  realized by such vertical crosscuts. Take two $n$-cylinders $A=\f a_n\ldots a_1]$ and $B=\f b_n\ldots b_1]$ for some $n\in\N$, such that $A\succ_L B$ and $A$ and $B$ are \emph{adjacent} $n$-cylinders, \ie there is no $n$-cylinder $D$ such that $A\succ_L D\succ_L B$. We will show that $S_A$ and $L_B$ have the same tail, \ie they both belong to $\sigma^i(L')$ for some $i\in\Z$. Since the accessible subsets of $\sigma^i(L')$ are arcs of finite length, it follows immediately that there cannot exist infinite canals for $\BD(X')$.
	
	Take $A$ and $B$ as above and let $m\in\{0, \ldots, n-1\}$ be the smallest nonnegative number such that $a_{m+1}\neq b_{m+1}$.\\
	First assume that $\#_1(a_m\ldots a_1)$ is odd. Then, $S_A=S_{0a_m\ldots a_1}$ and $L_B=L_{1a_m\ldots a_1}$, since $A\succ_L B$ are adjacent. Let $k\in \{1,\ldots, m-1\}$ be the smallest number such that $c_2\ldots c_{m-k+2}=a_{k+1}\ldots a_m1$, (compare with the proof of Lemma~\ref{lem:horror2}). Assume first that such $k$ indeed exists. Since also $c_2\ldots c^{*}_{m-k+2}\subset S_A$ is admissible, it follows that $\#_1(a_{k+1}\ldots a_m)$ is odd. Thus, $\#_1(a_k\ldots a_1)$ is even and since $a_k=1$ it holds that $\#_1(a_{k-1}\ldots a_1)$ is odd. As in the proof of Lemma~\ref{lem:horror2}, we conclude that $L_{1a_m\ldots a_1}=L'1a_{m}\ldots a_1$. \black{If $k$ does not exist we again have that $L_{1a_m\ldots a_1}=L'1a_{m}\ldots a_1$.} Note that $k=m$ is not possible. Furthermore, since $\#_1(a_k\ldots a_m)$ is odd, it follows from the specific form of $\nu$ that $S_{0a_m\ldots a_1}=L'0a_m\ldots a_1$,  which is always admissible.
	Therefore, $S_A$ and $L_B$ have the same left infinite tail.\\
	Now assume that $\#_1(a_m\ldots a_1)$ is even. Then $S_A=S_{1a_m\ldots a_1}$ and $L_B=L_{0a_m\ldots a_1}$, since $A\succ_L B$ are adjacent. Let $k\in \{1,\ldots, m-1\}$ again be the smallest number such that $c_2\ldots c_{m-k+1}=a_{k+1}\ldots a_m1$. By analogous arguments as in the preceding paragraph we obtain that $\#_1(a_{k-1}\ldots a_1)$ is even and thus as in the proof of Lemma~\ref{lem:horror2}, we conclude that $S_{1a_m\ldots a_1}=L'1a_{m}\ldots a_1$.  Furthermore, $L_{0a_m\ldots a_1}=L'0a_m\ldots a_1$ which is always admissible. Again, $S_A$ and $L_B$ have the same left infinite tail. Therefore, it holds that all the canals are finite, \ie there exist no third and fourth kind prime ends corresponding to $\BD(X')$.
\end{proof}

The following theorem follows directly from the preceding eight lemmas.

\begin{theorem}
	If $\nu$ is of irrational type and $X'$ is embedded with $\BD$, then there are countably infinitely many partially accessible rays of $\BD(X')$; these are the arc-components which are symbolically described by a tail which is a shift of $\ovl\nu$. Each of them contains an endpoint of $\BD(X')$ and accessible set is a compact arc which contains that endpoint. Furthermore, there exist uncountably many accessible arc-components which are accessible in a single point which is an endpoint of $\BD(X')$. All (uncountably many) endpoints of $\BD(X')$ are accessible. 
\end{theorem}

\subsection{Rational endpoint case} Let $q=\frac{m}{n}$.
In this subsection we study $\BD(X')$ when $\nu$ is either $\rhe(q)$ or $\lhe(q)$. We provide a symbolic proof of Theorem 4.66 from \cite{3G}.

When $\nu=\lhe(q)=(w_q1)^{\infty}$ it follows that $L'=\ovl{\lhe(q)}$. In Remark~\ref{rem:palind} we argued that there exist palindromes $Y, Z$ such that $\lhe(q)=(Y1Z1)^{\infty}$, thus $\ovl{\lhe(q)}=(1Z1Y)^{\infty}$. Note that both $Y$ and $Z$ are even, from which we conclude that $\tau_R(L')=\infty$. Thus the right endpoint of $A(L')$ is also an endpoint of $X'$ and since there are no other folding points on $\mathcal{U}_{L'}$ except of this endpoint, the ray $\mathcal{U}_{L'}$ is a fully accessible. Since $\sigma$ is extendable to the plane it follows that $\sigma^{i}(\mathcal{U}_{L'})$ are fully accessible for every $i\in\{0,1,\ldots, n-1\}$ (where $n$ is the period of $\lhe(q)$). Lemma~\ref{lem:horror2} assures that the union of $n$ rays is indeed the complete set of accessible points of $\BD(X')$ for $\nu=\lhe(q)$. Thus the circle of prime ends decomposes into $n$ half-open intervals, where the endpoints represent the endpoints of $X'$. Summarizing, we have the following theorem:

\begin{theorem}
	If $\nu=lhe(q)$ for some $q=\frac{m}{n}$, where $m$ and $n$ are relatively prime, then in $\BD(X')$ there exist $n$ fully accessible rays which are symbolically described  by a tail which is a shift of $\ovl{\nu}$ and no other point from $\BD(X')$ is accessible. Specifically, there exist no infinite canals in $\BD(X')$.  
\end{theorem}

When $\nu=\rhe(q)$ it holds by Lemma~\ref{lem:BdCHL'} that $L'=\ovl{\rhe(q)}=(1Y1Z)^{\infty}01$. Since $Y$ starts with $1$ it holds that there exists a folding point $p\in\mathcal{U}_{L'}$ on a basic arc with itinerary $\ovl{l(L')}=(1Y1Z)^{\infty}11$. Since $\rhe(q)$ is strictly preperiodic it follows that left tail of $\ovl{l(L')}$ always differs from \black{a positive shift of} $\ovl{\rhe(q)}$, so Lemma~\ref{lem:horror2} implies that $\ovl{l(L')}$ is not an extremum of any cylinder. Proposition~\ref{prop:wiggles} implies that $p$ is Type 2 folding point and consequently $\mathcal{U}_{L'}$ is partially accessible. Moreover, since $\mathcal{U}_{L'}$ contains no other folding points we conclude that one component of $\mathcal{U}_{L'}\setminus \{p\}$ is fully accessible and the other component of $\mathcal{U}_{L'}\setminus \{p\}$ is not accessible. Since $\sigma$ is extendable, $\sigma^i(\mathcal{U}_{L'})$ are also partially accessible. Lemma~\ref{lem:horror2} implies that the circle of prime ends decomposes into $n$ half-open intervals and their endpoints are representing the accessible folding points of Type 2. Thus we obtain the following theorem:

\begin{theorem}
	If $\nu=rhe(q)$ for some $q=\frac{m}{n}$, where $m$ and $n$ are relatively prime, then in $\BD(X')$ there exist $n$ partially accessible lines which are symbolically described  by a tail which is a shift of $\ovl{\nu}$ and no other point from $\BD(X')$ is accessible. Specifically, there exist no infinite canals in $\BD(X')$.  
\end{theorem}

\subsection{Rational interior case}
Assume $q=\frac{m}{n}$, where $m$ and $n$ are relatively prime. We will show that in the rational interior case there exist $n$ fully accessible arc-components which are dense lines in $X'$. We show that folding points which are not lying in the extrema of cylinders are not accessible, so the remaining $n$ points on the circle of prime ends are simple dense canals. That is an analogue of Theorem 4.64 from \cite{3G} for tent inverse limits.

\begin{lemma}[Theorem 16 in \cite{BdCH}]\label{lem:BdCH}
	Suppose that $\nu$ is of rational interior type for $q=m/n$, where $m$ and $n$ are relatively prime. Then a sequence $\ovl{t}\in\{0,1\}^{\infty}$ which does not belong to $\mathcal{C}$ is admissible if and only if
	\begin{itemize}
		\item[(a)] $\sigma^{i}(\ovl{t})\preceq \rhe(q)$ for all $i\in\N$,
		\item[(b)] $\sigma^{i}(\ovl{t})\preceq \lhe(q)$ for all $i\in\N$ for which $\sigma^{i}(\ovr{t})\succ \sigma^{n+1}(\nu)$.
	\end{itemize}
\end{lemma}
\black{
	\begin{remark}
	In the rest of this section, for simplicity, we introduce the following notation. Let $w\in\{0,1\}^{\infty}$ be an infinite sequence. We denote by $w[j]$ the first $j\in\N$ coordinates of $w$.
\end{remark}}

\begin{lemma}\label{lem:horrorgram}
	Say that $q=m/n$, where $m$ and $n$ are relatively prime.
	If $\lhe(q)\prec\nu\prec\rhe(q)$, then all the extrema of cylinders of $\BD(X')$ have tails in $\sigma^i(L')$ for some $i\in\Z$.
\end{lemma}
\begin{proof} Fix an arbitrary admissible word $b_j\ldots b_1\in\{0,1\}^{j}$ for some $j\in\N$.
	
	We will calculate the top/bottom of the cylinder $[b_j\ldots b_1]$.
	Assume that $b_j\ldots b_1\succ\sigma^{n+1}(\nu)\black{[j]}$ and $\#_1(b_j\ldots b_1)$ is even (odd). We first show that if $\ovl{\lhe(q)}b_j\ldots b_1$ is admissible, then it equals $L_{b_j\ldots b_1} (S_{b_j\ldots b_1})$. Assume by contradiction that there exists an admissible $\ldots x_2x_1b_j\ldots b_1\succ\ovl{\lhe(q)}b_j\ldots b_1$ ($\ldots x_2x_1b_j\ldots b_1\prec\ovl{\lhe(q)}b_j\ldots b_1$). Then $\ldots x_2x_1\succ\ovl{\lhe(q)}$ ($\ldots x_2x_1\succ\ovl{\lhe(q)}$). But that combined with $b_j\ldots b_1\succ\sigma^{n+1}(\nu)\black{[j]}$ gives by (b) from Lemma~\ref{lem:BdCH} that $\ldots x_2x_1b_j\ldots b_1\succ\ovl{\lhe(q)}b_j\ldots b_1$ is not admissible, a contradiction. Similarly, we show that if $b_j\ldots b_1\preceq\sigma^{n+1}(\nu)\black{[j]}$, $\#_1(b_j\ldots b_1)$ is even (odd) and $\ovl{\rhe(q)}b_j\ldots b_1$ is admissible, then it equals $L_{b_j\ldots b_1} (S_{b_j\ldots b_1})$.
	
	In the next two paragraphs we prove that 
	the sequences of the form $\ovl{\rhe(q)}b_j\ldots b_1$ and $\ovl{\lhe(q)}b_j\ldots b_1$ in special case to which we restrict later in the proof satisfy conditions $(a)$ and $(b)$ from Lemma~\ref{lem:BdCH} and are thus admissible.
	
	If $b_{i+1}\ldots b_j$ does not equal the beginning of $\rhe(q)$ for any $i\in\{0, \ldots, j-1\}$, then the sequences $\ovl{\rhe(q)}b_j\ldots b_1$ and $\ovl{\lhe(q)}b_j\ldots b_1$ satisfy $(a)$ from Lemma~\ref{lem:BdCH}. Assume there is an index $i\in\{0, \ldots j-1\}$ such that $b_{i+1}\ldots b_j$ is the beginning of $\rhe(q)$ and take the smallest such $i\in\{0, \ldots, j-1\}$. Assume $\#_1(b_{i+1}\ldots b_j)$ is odd (later in the proof we need only this special case). If $b_{\alpha+1}\ldots b_j$ is also the beginning of $\rhe(q)$ for some $\alpha\in\{0, \ldots, j-1\}$, where $\alpha\geq i$, then $\#_1(b_{\alpha+1}\ldots b_j)$ is also odd. Note that $b_{\alpha+1}\ldots b_j10\prec \rhe(q)\black{[j-\alpha+1]}$ for every such $\alpha$. Thus $\ovl{\rhe(q)}b_{j}\ldots b_1$ and $\ovl{\lhe(q)}b_j\ldots b_1$ satisfy condition $(a)$ from Lemma~\ref{lem:BdCH}.

	If for every $i\in\{1, \ldots, j\}$ either $b_{i}\ldots b_1\preceq\sigma^{n+1}(\nu)\black{[i]}$ or $b_{i+1}\ldots b_j$ is not the beginning of $\lhe(q)$, then $\ovl{\rhe(q)}b_j\ldots b_1$ and $\ovl{\lhe(q)}b_j\ldots b_1$ satisfy $(b)$ from Lemma~\ref{lem:BdCH}. Assume there is $i<j$ such that $b_i\ldots b_1\succ\sigma^{n+1}(\nu)\black{[i]}$ and $b_{i+1}\ldots b_j$ is the beginning of $\lhe(q)$ and take the smallest such index $i$. If $\#_1(b_{i+1}\ldots b_j)$ is odd (as in the paragraph above, later in the proof we need only this special case) and there is $\beta\in\{i, \ldots, j-1\}$ such that $b_{\beta+1}\ldots b_j$ is also the beginning of $\lhe(q)$, then $\#_1(b_{\beta+1}\ldots b_j)$ is also odd and thus $b_{\beta+1}\ldots b_j10\prec \lhe(q)\black{[j-\beta+1]}$ for every such $\beta$. We conclude that $\ovl{\rhe(q)}b_j\ldots b_1$ and $\ovl{\lhe(q)}b_j\ldots b_1$ satisfy condition $(b)$ from Lemma~\ref{lem:BdCH}.
	
	Recall that $L'=\ovl{\rhe(q)}=(1w_q)^{\infty}01$. 
	
	Fix an admissible word $a_N\ldots a_1\in \{0,1\}^{N}$ for some $N\in\N$. Let $k\in\{1, \ldots, N\}$ be, if existent, the smallest index such that $a_k\ldots a_1\succ\sigma^{n+1}(\nu)\black{[k]}$ and $a_{k+1}\ldots a_N$ is the beginning of $\lhe(q)$. We set $k=N$ when $a_N\ldots a_1\succ\sigma^{n+1}(\nu)\black{[N]}$ (then $a_{k+1}\ldots a_N=\emptyset$ is the beginning of $\lhe(q)$). Let $k'\in\{0, 1, \ldots, N-1\}$ be the smallest index such that $a_{k'+1}\ldots a_N$ equals the beginning of $\rhe(q)$. Note that if $a_i=1$ for some $i\in\{1, \ldots, N\}$, then such $k'$ exists. If $a_N\ldots a_1=0^N$, then $L_{a_N\ldots a_1}=\ovl{\rhe(q)}0^N$ and $S_{a_N\ldots a_1}=S=(1w_q)^{\infty}0$.

	If $a_i=1$ for some $i\in\{1, \ldots, N\}$, the diagram in Figure~\ref{fig:horrorgram} provides an algorithm to calculate $L_{a_N\ldots a_1}$ ($S_{a_N\ldots a_1}$).
	
	   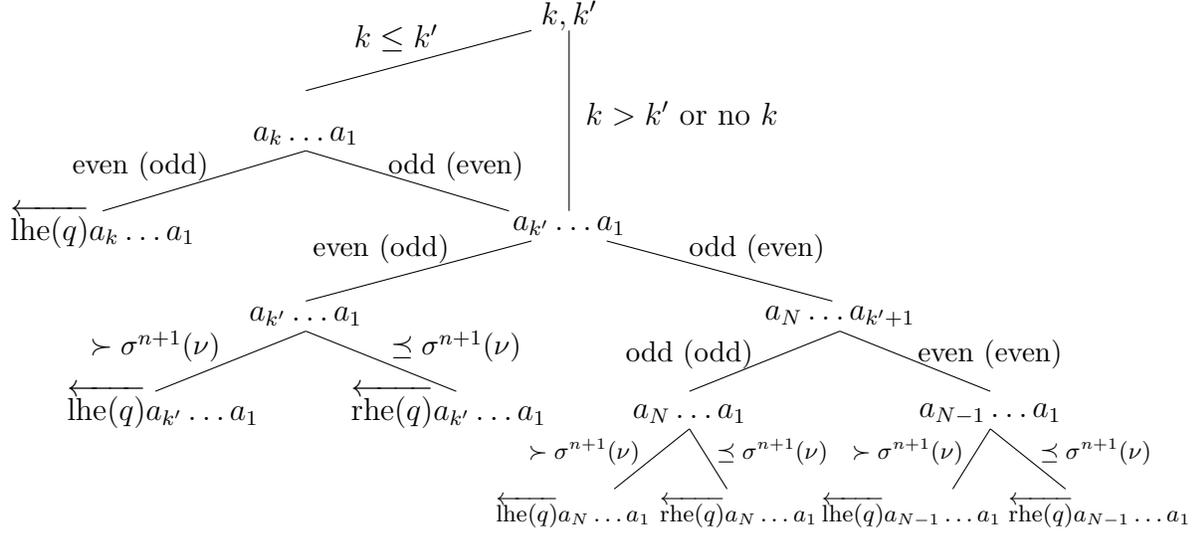
\begin{figure}
	   	\begin{center}
	   		\begin{tikzpicture}
	   		\node at (0,-0.2) {$k,k'$};
	   		\draw (-0.5,-0.4)--(-3.5,-1.2);
	   		\node at (-2.3,-0.5) {$k\leq k'$};
	   		\draw (0,-0.4)--(0,-2.8);
	   		\node at (1.5,-1.5) {$k>k'$ or no $k$};
	   		\node at (-3.5,-1.8) {$a_k\ldots a_1$};
	   		\draw (-3.5,-2)--(-0.8,-2.8);
	   		\node at (0,-3) {$a_{k'}\ldots a_1$};
	   		\node at (-1.5,-2.2) {\small{odd (even)}};
	   		\draw (-3.5,-2)--(-6.2,-2.8);
	   		\node at (-6.2,-3) {$\ovl{\lhe(q)}a_{k}\ldots a_1$};
	   		\node at (-5.7,-2.2) {\small{even (odd)}};
	   		\draw (-0.5,-3.2)--(-3.5,-4);
	   		\node at (-3.5,-4.2) {$a_{k'}\ldots a_1$};
	   		\node at (-2.5,-3.3) {\small{even (odd)}};
	   		\draw (0.5,-3.2)--(3.5,-4);
	   		\node at (2.5,-3.3) {\small{odd (even)}};
	   		\node at (3.6,-4.2) {$a_{N}\ldots a_{k'+1}$};
	   		\draw (-3.5,-4.4)--(-5.5,-5.2);
	   		\node at (-5.5,-4.6) {\small{$\succ\sigma^{n+1}(\nu)$}};
	   		\node at (-5.4,-5.4) {$\ovl{\lhe(q)}a_{k'}\ldots a_1$};
	   		\draw (-3.5,-4.4)--(-1.5,-5.2);
	   		\node at (-1.6,-5.4) {$\ovl{\rhe(q)}a_{k'}\ldots a_1$};
	   		\node at (-1.5,-4.6) {\small{$\preceq\sigma^{n+1}(\nu)$}};
	   		\draw (3.6,-4.4)--(1.6,-5.2);
	   		\node at (1.6,-5.5) {$a_{N}\ldots a_1$};
	   		\node at (1.6,-4.7) {\small{odd (odd)}};
	   		\draw (3.6,-4.4)--(5.6,-5.2);
	   		\node at (5.6,-5.5) {$a_{N-1}\ldots a_1$};
	   		\node at (5.6,-4.7) {\small{even (even)}};
	   		\draw (5.6,-5.7)--(5.1,-6.5);
	   		\node at (4.5,-6) {\scriptsize{$\succ\sigma^{n+1}(\nu)$}};
	   		\node at (7,-6.8) {\scriptsize{ $\ovl{\rhe(q)}a_{N-1}\ldots a_1$}};
	   		\draw (5.6,-5.7)--(6.6,-6.5);
	   		\node at (7,-6) {\scriptsize{$\preceq\sigma^{n+1}(\nu)$}};
	   		\node at (4.5,-6.8) {\scriptsize{ $\ovl{\lhe(q)}a_{N-1}\ldots a_1$}};
	   		\draw (1.6,-5.7)--(0.6,-6.5);
	   		\node at (0.2,-6) {\scriptsize{$\succ\sigma^{n+1}(\nu)$}};
	   		\node at (0,-6.8) {\scriptsize{ $\ovl{\lhe(q)}a_{N}\ldots a_1$}};
	   		\draw (1.6,-5.7)--(2.1,-6.5);
	   		\node at (2.7,-6) {\scriptsize{$\preceq\sigma^{n+1}(\nu)$}};
	   		\node at (2.2,-6.8) {\scriptsize{ $\ovl{\rhe(q)}a_{N}\ldots a_1$}};
	   		\end{tikzpicture}
	   	\end{center}
	   	\caption{Calculating the $L_{a_N\ldots a_1}$ and $S_{a_N\ldots a_1}$ in the rational interior case. The graph should be read as follows: if we want to calculate $L_{a_N\ldots a_1}$ we read the terms outside of the brackets and to calculate $S_{a_N\ldots a_1}$ we read the terms inside the brackets. Say we want to calculate $L_{a_N\ldots a_1}$ ($S_{a_N\ldots a_1}$). We first calculate $k$ and $k'$ and compare them. Say $k>k'$ or $k$ does not exist. We move down the right branch. Next we calculate the parity of $a_{k'}\ldots a_1$. Say it is even (odd), then we move down the left branch. If $a_{k'}\ldots a_1\succ\sigma^{n+1}(\nu)\black{[k']}$ then $L_{a_N\ldots a_1}=\protect\ovl{\lhe}a_{k'}\ldots a_1$ ($S_{a_N\ldots a_1}=\protect\ovl{\lhe}a_{k'}\ldots a_1$) and if $a_{k'}\ldots a_1\preceq\sigma^{n+1}(\nu)\black{[k']}$ then $L_{a_N\ldots a_1}=\protect\ovl{\rhe}a_{k'}\ldots a_1$ ($S_{a_N\ldots a_1}=\protect\ovl{\rhe}a_{k'}\ldots a_1$). \black{For the sake of presentation we did not include \black{$[\bullet]$} \black{in the figure.}}}
	   	\label{fig:horrorgram}
	   \end{figure}
	   
	   To see that the defined sequences are indeed $L_{a_N\ldots a_1}$ ($S_{a_N\ldots a_1}$) we use the first part of the proof. For example, take the case where the algorithm gives $\ovl{\lhe}a_N\ldots a_1$. Since $a_N\ldots a_1\succ\sigma^{n+1}(\nu)\black{[N]}$ and $\#_1(a_N\ldots a_1=a_N\ldots a_{k'+1}a_{k'}\ldots a_1)$ is even (odd), if $\ovl{\lhe}a_N\\\ldots a_1$ is admissible then it equals $L_{a_N\ldots a_1}$ ($S_{a_N\ldots a_1}$). To see that it satisfies $(a)$, note that $\#_1(a_{N}\ldots a_{k'+1})$ is odd by assumption. To see that is satisfies $(b)$, assume first that there exists $k$ and $k\leq k'$. Then $\#_1(a_{k+1}\ldots a_{k'})$ is even and thus $\#_1(a_N\ldots a_{k+1})$ is odd. If $k$ does not exists, we are done. If $k>k'$, then since $a_{k'+1}\ldots a_N$ is the beginning of $\rhe(q)$ and $a_{k+1}\ldots a_N$ is the beginning of $\lhe(q)$ it follows that $\#_1(a_{k'+1}\ldots a_{k})$ is even and thus $\#_1(a_{k+1}\ldots a_N)$ is of the same parity as $\#_1(a_{k'+1}\ldots a_N)$, which is odd. That finishes the proof in this case. Other cases follow using analogous computations. Note that if $\#_1(a_N\ldots a_{k'+1})$ is even, then since $a_{k'+1}\ldots a_N$ is the beginning of $\rhe(q)$ it follows that $a_N=1$ and thus $\#_1(a_{N-1}\ldots a_{k'+1})$ is odd (this is needed in the proof of the two cases in the right branch of Figure~\ref{fig:horrorgram}).    
\end{proof}

\begin{lemma}\label{lem:nlines}
	Say that $q=m/n$, where $m$ and $n$ are relatively prime.
	If $\lhe(q)\prec\nu\prec\rhe(q)$, then every admissible itinerary in $\sigma^i(\mathcal{U}_{L'})$ is realized as an extremum of a cylinder of $\BD(X')$.
\end{lemma}
\begin{proof}
	Assume that $\ovl x=\ldots x_{2}x_1$ is an admissible tail and that there exists $K\in\N_0$ such that $\ldots x_{K+2}x_{K+1}=\ovl{\lhe(q)}$ and take $K$ the smallest index with that property. Denote by $\lhe(q)=(w_q1)^{\infty}=(y_1\ldots y_n)^{\infty}$ and note that $\rhe(q)=10(\hat{w_q}1)^{\infty}$ and thus $\sigma^{n+1}(\rhe(q))=(1\hat{w_q})^{\infty}=(y_n\ldots y_1)^{\infty}$. Since $\rhe(q)\succ\nu$ and they agree on the first $n+1$ places (which equal $c_q$ and which is a word of even parity, for details see \eg \cite{3G}), it follows that $\sigma^{n+1}(\rhe(q))\succ\sigma^{n+1}(\nu)$. Let $J\in\N$ be the smallest natural number such that $(y_n\ldots y_1)^J\succ\sigma^{n+1}(\nu)\black{[nJ]}$. We study the cylinder $Y=\f y_n\ldots y_1(y_n\ldots y_1)^Jx_K\ldots x_1]$. Note that $x_i\ldots x_K(y_1\ldots y_n)^{J+1}$ does not agree with the beginning of $\lhe(q)$ for any $i\in\{1, \ldots, K\}$. Also $y_i\ldots y_n(y_1\ldots y_n)^j$ does not agree with the beginning of $\lhe(q)$ for any $i\in\{2, \ldots, n\}$ and any $j\in\N$.
	Denote by $a_{N}\ldots a_1=y_n\ldots y_1(y_n\ldots y_1)^Jx_K\ldots x_1$. Let $k\in\{1, \ldots, N\}$ be, if existent, the smallest index such that $a_k\ldots a_1\succ\sigma^{n+1}(\nu)\black{[k]}$ and $a_{k+1}\ldots a_N$ is the beginning of $\lhe(q)$ (compare with the definition of $k$ in the proof of Lemma~\ref{lem:horrorgram}). By the choice of $J$ it follows that $k$ indeed exists and $k\in\{K+Mn: M\in\{0, \ldots, J\}\}$. So, if for any $i\in\{0, \ldots, K-1\}$ the word  $x_{i+1}\ldots x_K(y_1\ldots y_n)^{J+1}$ does not equal the beginning of $\rhe(q)$, then Lemma~\ref{lem:horrorgram} implies that $\ovl x=L_Y$ or $\ovl x=S_Y$, depending on the parity of $\#(x_K\ldots x_1)$.\\
	 If there is $\alpha\in\{0, \ldots, K-1\}$ such that the word  $x_{\alpha+1}\ldots x_K(y_1\ldots y_n)^{J+1}$ equals the beginning of $\rhe(q)$, then $x_{\alpha}\ldots x_1\preceq\sigma^{n+1}(\nu)\black{[\alpha]}$ (otherwise $Y$ does not satisfy $(b)$ from Lemma~\ref{lem:BdCH} and is thus not admissible). Lemma~\ref{lem:horrorgram} implies that $\ovl{\rhe(q)}x_{\alpha}\ldots x_1$ equals $L_Y$ or $S_Y$, depending on the parity of $\#(x_{\alpha}\ldots x_1)$. Since the tails of $\rhe(q)$ and $\lhe(q)$ are shifts of one another and $J\geq 1$ it follows that $\ovl x=\ovl{\rhe(q)}x_{\alpha}\ldots x_1$, which concludes the proof.
\end{proof}

\begin{theorem}
	Say that $q=m/n$, where $m$ and $n$ are relatively prime.
	If $\lhe(q)\prec\nu\prec\rhe(q)$, then in $\BD(X')$ there exist $n$ fully accessible arc-components which are dense lines in $X'$ and $n$ simple dense canals. Moreover, a point from $\BD(X')$ is accessible if and only if it belongs to one of these $n$ lines.
\end{theorem}
\begin{proof}
	Lemma~\ref{lem:horrorgram} shows that all the extrema of cylinders have tails in $\sigma^i(L')$ for some $i\in\Z$ and Lemma~\ref{lem:nlines} shows that every admissible itinerary in $\sigma^i(\mathcal{U}_{L'})$ is realized as an extremum of a cylinder. Since $L'$ is preperiodic of preperiod $n$, we obtain $n$ fully accessible lines in $\BD(X')$. Since $\sigma$ is extendable, no other non-degenerate arc can be accessible. Thus the circle of prime ends can be decomposed into $n$ open intervals and their $n$ endpoints. We claim that the endpoints correspond to simple dense canals.\\
	Assume by contradiction that a folding point $x\in \BD(X')$ is accessible. Then its every shift $\sigma^j(x)$ needs to be accessible for some natural number $j$ which divides $n$ (denoted from now onwards by $j|n$). We conclude that the tail corresponding to the point $x$ must be periodic of period $j|n$, \ie $\sigma^j(x)=x$. Note that there are no periodic kneading sequences $\nu$  of period $j|n$ for $\lhe(q)\prec\nu\prec\rhe(q)$ since $\lhe(q)$, $\rhe(q)$ and $\nu$ agree on the first $n-1$ places. Thus the basic arc $\black{A(\ovl x)}$ has $\tau_L(\ovl{x}), \tau_R(\ovl{x})$ finite. Specially, the basic arc $\black{A(\ovl x)}$ contains no endpoint of $X'$ and $x$ is the only accessible point in $\black{A(\ovl x)}$ and it thus needs to be Type 3 folding point. Write $\ovl x=\ldots x_3x_2x_1$. Since $x$ is a folding point and not an endpoint, there exist arbitrarily large $M, k_i\in\N$ such that $x_M\ldots x_1=c_{k_i+1}\ldots c_{k_i+M}$ and $x_{M+1}\neq c_{k_i}$. Now we proceed similarly as in Proposition~\ref{prop:restriction}. Fix a cylinder around $\ovl x$ and assume that all long basic arcs in that cylinder lie below (above) $\black{A(\ovl x)}$. Here long basic arcs $\ovl y$ are such that $\pi_0(x)\in\Int(\pi_0(\ovl y))$.\\
	 Specially, for $M$ large enough and when $c_{k_i}c_{k_i+1}\ldots c_{k_i+M}\neq c_2\ldots c_{M+2}$, the basic arcs $1^{\infty}c_{k_i}c_{k_i+1}\ldots c_{k_i+M}$ are long (if $M>\tau_L(x),\tau_R(x)$ then $\pi_0(1^{\infty}c_{k_i}c_{k_i+1}\ldots c_{k_i+M})=[T^{\tau_L(x)}, T^{\tau_R(x)}]$). Basic arcs in the chosen cylinder which do not project to $[T^{\tau_L(x)}, T^{\tau_R(x)}]$ are of the form $\black{A(\ldots\frac{0}{1}c_1c_2\ldots c_{k_i}c_{k_i+1}\ldots c_{k_i+M})}$, \black{where $\frac{0}{1}$ stands for either $0$ or $1$ on this entry}. Since $c_{k_i}\neq x_{M+1}$, it follows that those arcs are on the same side of $\ovl x$ as long arcs $1^{\infty}c_{k_i}c_{k_i+1}\ldots c_{k_i+M}$. Since we assumed that all long basic arcs lie on the same side of $\ovl x$ it follows that $\ovl x$ is an extremum of a cylinder, a contradiction.\\
	  The remaining case is when $c_{k_i}\ldots c_{k_i+M}=c_2\ldots c_{M+2}$ for all (but finitely many) $i\in\N$. That is, whenever $x_{M}\ldots x_1$ appears in the kneading sequence, then $x_{M}\ldots x_1=c_3\ldots c_{M+2}$ and $x_{M+1}\neq c_2=0$. However, $\ovl x$ is periodic of period $j|n$ and $x$ is a folding point, from which we conclude that $T^3(c)$ is periodic of period $j|n$ and $\ovl x=(c_3\ldots c_{n+2})^{\infty}$. Note that the only kneading sequence $\lhe(q)\prec\nu\prec\rhe(q)$ for which $T^3(c)$ is periodic of period $j|n$ is $10(\hat w_q0)^{\infty}$ which is actually periodic of period $n$. But there are no periodic kneading sequences $\nu$ of period $n$ such that $\lhe(q)\prec\nu\prec\rhe(q)$, a contradiction. Thus no folding point $x\in \BD(X')$ is accessible.
	
	We need to show that the $n$ accessible lines $\mathcal{U}^{i}\subset X'$ for $i\in\{0,\ldots, n-1\}$ are indeed dense in $X'$. It follows from Lemma~\ref{lem:horrorgram} that the symbolic code of $\mathcal{U}^{i}$ is eventually $\sigma^{i}(\ovl{\lhe(q)})$ for $i\in\{0,\ldots, n-1\}$. Let $a\in X'$ be a point with the backward itinerary $\ovl a=\ldots a_2a_1$. Note that for every natural number $\beta$, every $i\in\{0,\ldots, n-1\}$ and large enough natural number $\gamma$ the left infinite sequences $\sigma^{i}(\ovl{\lhe(q)})1^{\gamma}a_{\beta}\ldots a_1$ are admissible since they satisfy conditions $(a)$ and $(b)$ from Lemma~\ref{lem:BdCH}. Thus, letting $\beta\to \infty$ we get a sequence of basic arcs from $\mathcal{U}^i$ converging to $A(\ovl a)$ such that their $\pi_0$-th projections contain $\pi_0(a)$. 
	
	Therefore $n$ prime ends $P_1,\ldots, P_n$ on the circle of prime ends are either of the third or the fourth kind. Since the shores of the canal are lines which are dense in both directions it follows that $\Pi(P_i)=I(P_i)=\BD(X')$ for every $i\in\{1,\ldots,n\}$. Therefore, there are $n$ simple dense canals for every $\BD(X')$.
\end{proof}
\black{\section*{Acknowledgments} We would like to thank the referee for constructive comments which improved the exposition of this paper. We would also like to thank Henk Bruin for asking interesting questions and his constant encouragement.}

\end{document}